\newif\ifjar\jarfalse
\newif\ifreferee\refereefalse
\newif\ifmydraft\mydraftfalse   %
\makeatletter \@input{texdirectives.tex} \makeatother

\documentclass[11pt]{amsart}

\usepackage[T1]{fontenc}
\usepackage[utf8]{inputenc}
\usepackage[english]{babel}

\usepackage{amsmath,amssymb,amsthm}
\usepackage{mathtools} %
\usepackage{upgreek} %
\ifjar \else \usepackage{newtxtext,newtxmath} \fi %
\usepackage{thmtools, thm-restate} %
\usepackage{dutchcal} %

\usepackage{enumitem}
\setlist{itemsep=0pt}

\usepackage{xcolor} %
\definecolor{rulenameColor}{rgb}{0.5,0.5,0.5}

\usepackage{mdframed}

\ifjar
\usepackage{hyperref}
\else
\definecolor{darkgreen}{rgb}{0,0.2,0}
\definecolor{darkred}{rgb}{0.25,0,0}
\definecolor{darkblue}{rgb}{0,0,0.3}
\PassOptionsToPackage{obeyspaces}{url}
\usepackage[breaklinks,colorlinks,%
citecolor=darkgreen,linkcolor=darkblue,urlcolor=darkblue,%
bookmarksdepth=3,bookmarksopen]{hyperref}
\fi

\usepackage[capitalise,noabbrev,nameinlink]{cleveref} %

\ifjar
\newcommand{\orcidlogo}{\includegraphics[height=\fontcharht\font`A]{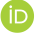}}
\else
\newcommand{\orcidlogo}{\resizebox{1.2\fontcharht\font`A}{!}{\includegraphics{orcidlogo}}}
\fi
\newcommand{\orcidlink}[1]{\href{https://orcid.org/#1}{\orcidlogo}}

\usepackage{mathpartir} %

\usepackage{caption}
\numberwithin{equation}{section}
\numberwithin{figure}{section}

\declaretheorem[numberwithin=section, style=plain]{theorem}
\declaretheorem[sibling=theorem, style=plain]{proposition}
\declaretheorem[sibling=theorem, style=plain]{lemma}
\declaretheorem[sibling=theorem, style=plain]{corollary}
\declaretheorem[sibling=theorem, style=definition]{definition}
\declaretheorem[sibling=theorem, style=definition]{example}

\ifmydraft
  \usepackage{showframe}
\fi

\newcommand{\defemph}[1]{\textbf{\emph{#1}}} %

\newcommand{\inCaseCustom}[1]{\smallskip\par\noindent\textit{#1}:}
\newcommand{\inCase}[1]{\inCaseCustom{Case \rref{#1}}}
\newcommand{\inCasesText}[1]{\inCaseCustom{Cases #1}}
\newcommand{\inCaseText}[1]{\inCaseCustom{Case #1}}

\newcommand{\mkset}[1]{\{#1\}} %
\newcommand{\such}{\mid} %
\newcommand{\pow}[1]{\mathcal{P}#1} %

\newcommand{\mto}{{\mapsto}} %
\newcommand{\finmap}[1]{\langle #1 \rangle} %
\newcommand{\extend}[3]{\finmap{#1, #2 {\mapsto} #3}} %
\newcommand{\extendMany}[2]{\finmap{#1, #2}} %

\newcommand{\dummy}{{\star}} %

\newcommand{\act}[1]{#1_{*}} %
\newcommand{\upto}[2]{#1_{(#2)}} %
\newcommand{\upact}[2]{#1_{(#2){*}}} %

\newcommand{\all}[1]{\forall #1 \,.\,}
\newcommand{\some}[1]{\exists #1 \,.\,}
\newcommand{\lthen}{\Rightarrow}

\newcommand{\Clos}[1]{\mathsf{Clos}(#1)}
\newcommand{\closure}[1]{\overline{#1}}
\newcommand{\Der}[2]{\mathsf{Der}_{#1}(#2)}

\newcommand{\bnfis}{\mathrel{\;{:}{:}{=}\ }}
\newcommand{\bnfor}{\mathrel{\;\big\vert\ \ }}

\newcommand{\Ty}{\mathsf{Ty}} %
\newcommand{\Tm}{\mathsf{Tm}} %
\newcommand{\EqTy}{\mathsf{EqTy}} %
\newcommand{\EqTm}{\mathsf{EqTm}} %

\newcommand{\sym}[1]{\mathsf{#1}} %
\newcommand{\symS}{\sym{S}} %
\newcommand{\symM}{\sym{M}} %

\newcommand{\arity}[1]{\mathsf{ar}(#1)} %
\newcommand{\rawRule}[2]{#1 \Longrightarrow #2} %

\newcommand{\genapp}[1]{\widehat{#1}} %

\newcommand{\abstr}[1]{\{#1\}} %

\newcommand{\types}{\vdash} %

\newcommand{\suitable}[1]{\text{$#1$ suitable}} %

\newcommand{\mkvar}[1]{\mathsf{#1}} %
\newcommand{\avar}[2]{\mkvar{#1}^{#2}} %
\newcommand{\avars}[3]{\mkvar{#1}_{#2}^{#3}} %

\newcommand{\bdry}{\mathcal{b}} %
\newcommand{\BB}{\mathcal{B}} %
\newcommand{\DD}{\mathcal{D}} %

\newcommand{\J}{\mathcal{j}} %
\newcommand{\JJ}{\mathcal{J}} %

\newcommand{\isType}[1]{#1\;\mathsf{type}}
\newcommand{\isVCtx}[1]{#1\;\mathsf{vctx}}
\newcommand{\isMCtx}[1]{#1\;\mathsf{mctx}}

\newcommand{\emptyExt}{[\,]} %
\newcommand{\emptyCtx}{[\,]} %
\newcommand{\of}{{:}} %

\newcommand{\plug}[2]{#1{\setlength{\fboxrule}{0.5pt}\setlength{\fboxsep}{1pt}\fbox{\vphantom{$#1$}$#2$}}} %

\newmdenv[linecolor=rulenameColor]{ruleframe}

\newcommand{\natty}[2]{\tau_{#1}(#2)} %

\newcommand{\asset}[1]{\{\kern-0.237em\vert #1 \vert\kern-0.237em\}} %
\newcommand{\asm}[1]{\mathsf{asm}(#1)} %
\newcommand{\fvz}[1]{\mathsf{fv}_0(#1)} %
\newcommand{\fvt}[1]{\mathsf{fvt}(#1)} %
\newcommand{\fv}[1]{\mathsf{fv}(#1)} %
\newcommand{\mv}[1]{\mathsf{mv}(#1)} %
\newcommand{\bv}[1]{\mathsf{bv}(#1)} %
\newcommand{\erase}[1]{\lfloor #1 \rfloor} %
\newcommand{\erasex}[1]{\lfloor\!\!\lfloor #1 \rfloor\!\!\rfloor} %
\newcommand{\strip}[1]{{\mathcal s}(#1)} %
\newcommand{\residue}[1]{{\mathcal r}(#1)} %

\newcommand{\by}{\;\mathsf{by}\;} %

\newcommand{\convert}[2]{\upkappa(#1, #2)} %

\newcommand{\trantt}[1]{{#1}_{\mathrm{tt}}} %
\newcommand{\rulename}[1]{\textnormal{\textsc{#1}}}

\newcommand{\inferenceRuleBare}[5][]{\inferrule*[#2={\color{rulenameColor}\rulename{#3}},#1]{#4}{#5}}
\newcommand{\inferenceRuleT}[4][]{\inferenceRuleBare[#1]{lab}{#2}{#3}{#4}}
\newcommand{\inferenceRuleLabel}[5][]{\inferenceRuleT[#1]{#2}{\label{#3}#4}{#5}}
\newcommand{\inferenceRule}[4][]{\inferenceRuleLabel[#1]{#2}{#2}{#3}{#4}}

\newcommand{\infRL}[4][]{\inferenceRuleBare[#1]{Left}{{\scriptsize #2}}{#3}{#4}}
\newcommand{\infRR}[4][]{\inferenceRuleBare[#1]{Right}{{\scriptsize #2}}{#3}{#4}}

\newcommand{\rref}[1]{\hyperref[#1]{\rulename{#1}}} %

\begin{document}
\title[Finitary type theories]{Finitary type theories with and without contexts}
\author[Philipp~G.\ Haselwarter]{Philipp~G.\ Haselwarter \orcidlink{0000-0003-0198-7751}}
\author[Andrej Bauer]{Andrej Bauer \orcidlink{0000-0001-5378-0547}}

\begin{abstract}
We give a definition of finitary type theories that subsumes many examples of dependent
type theories, such as variants of Martin-Löf type theory, simple type theories, first-order and higher-order logics, and homotopy type theory.
We prove several general meta-theorems about finitary type theories: weakening, admissibility of substitution and instantiation of metavariables, derivability of presuppositions, uniqueness of typing, and inversion principles.

We then give a second formulation of finitary type theories in which there are no explicit contexts. Instead, free variables are explicitly annotated with their types. We provide translations between finitary type theories with and without contexts, thereby showing that they have the same expressive power. The context-free type theory is implemented in the nucleus of the Andromeda~2 proof assistant.
\end{abstract}


\maketitle

\section{Introduction}
\label{sec:introduction}

We present a general definition of a class of dependent type theories which we call \emph{finitary type theories}.
In fact, we provide two variants of such type theories, with and without typing contexts, and show that they are equally expressive by providing translations between them.
Our definition broadly follows the development of general type theories~\cite{gtt}, but is specialized to serve as a formalism for implementation of a proof assistant. Indeed, the present paper is the theoretical foundation of the Andromeda~2 proof assistant, in which type theories are entirely defined by the user.

To be quite precise, we shall study \emph{syntactic presentations} of type theories, in the sense that theories are seen as syntactic constructions, and the meta-theorems conquered by a frontal assault on abstract syntax. Even though this may not be the most fashionable approach to type theory, we were lead to it by our determination to understand precisely what we were implementing in Andromeda~2. We certainly expect that the syntactic presentations will match nicely with some of the modern semantic accounts of type theories, and that the usefulness of finitary type theories will transcend mere theoretical support for proof assistants.

We thus present our development of type theories in an elementary style, preferring concrete to abstract definitions and constructions, without compromising generality.
In particular, this means that we first define ``raw'' terms, judgements, rules, and the like, and then proceed in stages to carve out the well-behaved fragment via predicates.
Our motivations for this choice are threefold.
First, an elementary definition requires only very modest meta-mathematical foundations and lends itself to interpretation in various foundational systems.
Second, by eschewing intermediate surrogates such as logical frameworks~\cite{harper-honsell-plotkin:framework,Uemura:General:2019} or quotient inductive-inductive types~\cite{Altenkirch:Type:2016}, the semantics of finitary type theories may be addressed directly, without recourse to the interpretation of such intermediates. And in any case, even the intermediates must eventually be syntactically presented if they are to be used at all.
Third, the programming languages available to us are not sufficiently expressive to isolate the well-formed fragment of type theory in one fell swoop. They enable and insist on a more traditional approach, in which the input strings are converted to syntactic trees, and the type theoretic entities presented in their ``raw'' form, as values of inductively defined datatypes.
The concrete nature of our constructions and meta-theorems then makes it possible to transcribe them to code in a straightforward fashion.
Further discussion of alternative approaches is postponed to \cref{sec:conclusion}.

Our definition captures dependent type theories of Martin-Löf style, i.e.\ theories that strictly separate terms and types, have four judgement forms (for terms, types, type equations, and typed term equations), and hypothetical judgements standing in intuitionistic contexts. Among examples are the intensional and extensional Martin-Löf type theory, possibly with Tarski-style universes, homotopy type theory, Church's simple type theory, simply typed $\uplambda$-calculi, and many others. Counter-examples can be found just as easily: in cubical type theory the interval type is special, cohesive and linear type theories have non-intuitionistic contexts, polymorphic $\uplambda$-calculi quantify over all types, pure type systems organize the judgement forms in their own way, and so on.

\subsection*{Contributions}

In \cref{sec:finitary-type-theories} we give an account of dependent type theories that is close to how they are traditionally presented.
A type theory should verify certain meta-theoretical properties: the constituent parts of any derivable judgement should be well-formed, substitution rules should be admissible, and each term should have a unique type.
The definition of \emph{finitary type theories} proceeds in stages.
Each of the stages refines the notion of \emph{rule} and \emph{type theory} by specifying conditions of well-formedness.
We start with the raw syntax~(\cref{sec:raw-syntax}) of expressions and formal metavariables, out of which contexts, substitutions, and judgements are formed.
Next we define \emph{raw rules}~(\cref{sec:raw-rules}), a formal notion of what is commonly called ``schematic inference rule''. We introduce the \emph{structural rules}~(\cref{fig:struct-rules,fig:equality-rules,fig:well-formed-boundaries}) that are shared by all type theories, and define \emph{congruence rules}~(\cref{def:congruence-rule}). These rules are then collected into raw type theories~(\cref{def:raw-type-theory}).
The definition of raw rules ensures the well-typedness of each constituent part of a raw rule, by requiring the derivability of the presuppositions of a rule.
In order to rule out circularities in the derivations of well-typedness, and to provide an induction principle for finitary type theories, we introduce \emph{finitary rules} and \emph{finitary type theories}~(\cref{sec:finitary-rules-and-tts}).
Finally, \emph{standard} type theories are introduced~(\cref{def:standard-type-theory}) to enforce that each symbol is associated to a unique rule.

We prove the following metatheorems about raw~(\cref{sec:meta-theorems-raw}), finitary~(\cref{sec:meta-thm-finitary}), and standard type theories~(\cref{sec:meta-theorems-about}):
admissibility of substitution and equality substitution~(\cref{thm:substitution-admissible}),
admissibility of instantiation of metavariables~(\cref{prop:instantiation-admissible}) and equality instantiation~(\cref{thm:admissibility-equality-instantiation}),
derivability of presuppositions~(\cref{prop:presuppositivity}),
admissibility of ``economic'' rules~(\cref{prop:tt-specific-eco,prop:tt-meta-eco,def:congruence-rule-eco}),
inversion principles~(\cref{thm:inversion}),
uniqueness of typing~(\cref{thm:uniqueness-of-typing}).

The goal of \cref{sec:context-free-tt} is the development of a context-free presentation of finitary type theories that can serve as foundation of the implementation of a proof assistant.
The definition of finitary type theories in \cref{sec:finitary-type-theories} is well-suited for the metatheoretic study of type theory, but does not directly lend itself to implementation. For instance, in keeping with traditional accounts of type theory, contexts are explicitly represented as lists.

In \emph{context-free type theories}, the syntax of expressions~(\cref{sec:context-free-raw-syntax}) is modified so that each free variable is annotated with its type $\avar a A$ rather than being assigned a type by a context. As the variables occurring in the type annotation $A$ are also annotated, the dependency between variables is recorded. Judgements in context-free type theories thus do not carry an explicit context. Metavariables are treated analogously.
To account for the possibility of proof-irrelevant rules like equality reflection, where not all of the variables used to derive the premises are recorded in the conclusion, we augment type and term equality judgements with \emph{assumption sets}~(\cref{sec:context-free-judg-bound}). Intuitively, in a judgement $\types A \equiv B \by \alpha$, the assumption set $\alpha$ contains the (annotated) variables that were used in the derivation of the equation but may not be amongst the free variables of $A$ and $B$.
The conversion rule of type theory allows the use of a judgemental equality to construct a term judgement. To ensure that assumption sets on equations are not lost as a result of conversion, we include \emph{conversion terms}~(\cref{fig:syntax-context-free-type-theories}).

Following the development of finitary type theories, we introduce raw context-free rules and type theories~(\cref{sec:context-free-rules-type}).
We proceed to define \emph{context-free finitary rules} and type theories whose well-formedness is derivable with respect to a well-founded order~(\cref{def:context-free-finitary}), and \emph{standard} theories~(\cref{def:context-free-standard-type-theory}).

Subsequently, we prove metatheorems about context-free raw~(\cref{sec:meta-theorems-cf-raw}), finitary~(\cref{sec:meta-theorems-cf-finitary}), and standard type theories~(\cref{sec:meta-theorems-cf-standard}).
The metatheorems in this section are similar to those obtained for finitary type theories, with the exception of the metatheorems specific to context-free type theorems~(\cref{sec:meta-theorems-cf-specific}).
In particular, and contrary to finitary type theories, context-free raw type theories satisfy strengthening (\cref{thm:context-free-strengthening}).
We further prove that conversion terms do not ``get in the way'' when working in context-free type theory (\cref{lem:boundary-convert}).
The constructions underlying these metatheorems are defined on judgements rather than derivations, and can thus be implemented effectively in a proof assistant for context-free type theories without storing derivation trees.

In \cref{sec:context-free-cons-comp}, we establish a correspondence between type theories with and without contexts by constructing translations back and forth (\cref{thm:cf-to-tt-bdry-jdg,thm:tt-to-cf}).

\subsubsection*{Acknowledgements}

The present work draws its inspiration from our joint work with Peter LeFanu Lumsdaine on general type theories~\cite{gtt}. We thank Peter for spearheading the development of general type theories, which inspired us to implement user-definable dependent type theories in Andromeda~2.
We also thank Anja Petković Komel for numerous fruitful discussions, and for pushing through even the most horrid technicalities with us. The theorems about admissibility of substitutions and instantiations are to be considered joint work with Anja.
We are grateful to Robert Harper and Matija Pretnar for valuable comments on an earlier version of this material as included in~\cite{pgh:thesis}.

This material is based upon work supported by the Air Force Office of Scientific Research under award numbers FA9550-14-1-0096 and FA9550-21-1-0024.

\section{Finitary type theories}
\label{sec:finitary-type-theories}

Our treatment of type theories follows in essence the definition of general type theories carried out in~\cite{gtt}, but is tailored to support algorithmic derivation checking in three respects: we limit ourselves to finitary symbols and rules, construe metavariables as a separate syntactic class rather than extensions of signatures by fresh symbols, and take binding of variables to be a primitive operation on its own.

\subsection{Raw syntax}
\label{sec:raw-syntax}

In this section we describe the \emph{raw} syntax of fintary type theories, also known as pre-syntax. We operate at the level of \emph{abstract} syntax, i.e.\ we construe syntactic entities as syntax trees generated by grammatical rules in inductive fashion. Of course, we still display such trees \emph{concretely} as string of symbols, a custom that should not detract from the abstract view.

Raw expressions are formed without any typing discipline, but they have to be syntactically well-formed in the sense that free and bound variables must be well-scoped and that all symbols must be applied in accordance with the given signature. We shall explain the details of these conditions after a short word on notation.

We write $[X_1, \ldots, X_n]$ for a finite sequence and $f = \finmap{X_1 \mto Y_1, \ldots, X_n \mto Y_n}$ for a sequence of pairs $(X_i, Y_i)$ that represents a map taking each~$X_i$ to~$Y_i$. An alternative notation is $\finmap{X_1 \of Y_1, \ldots, X_n \of Y_n}$, and we may elide the parenthess $[\cdots]$ and $\finmap{\cdots}$.
The \defemph{domain} of such~$f$ is the set $|f| = \mkset{X_1, \ldots, X_n}$, and it is understood that all $X_i$ are different from one another.
Given $X \not\in |f|$, the \defemph{extension $\extend{f}{X}{Y}$} of $f$ by $X \mapsto Y$ is
the map
\begin{equation*}
  \extend{f}{X}{Y} : Z \mapsto
  \begin{cases}
    Y    & \text{if $Z = X$,}\\
    f(Z) & \text{if $Z \in |f|$.}
  \end{cases}
\end{equation*}
Given a list $\ell = [\ell_1, \ldots, \ell_n]$, we write $\upto{\ell}{i} = [\ell_1, \ldots, \ell_{i-1}]$ for its $i$-th initial segment. We use the same notation in other situations, for example $\upto{f}{i} = \finmap{X_1 \mapsto Y_1, \ldots, X_{i-1} \mapsto Y_{i-1}}$ for~$f$ as above.

\subsubsection{Variables and substitution}
\label{sec:variables-and-substitution}

We distinguish notationally between the disjoint sets of \emph{free variables} $\mkvar{a}, \mkvar{b}, \mkvar{c}, \ldots$ and \emph{bound variables} $x, y, z, \ldots$, each of which are presumed to be available in unlimited supply. The free variables are scoped by variable contexts, while the bound ones are always captured by abstractions.

The strict separation of free and bound variables is fashioned after \emph{locally nameless syntax}~\cite{mckinna,chargueraud}, a common implementation technique of variable binding in which free variables are represented as names and the bound ones as de~Bruijn indices~\cite{debruijn}. In \cref{sec:context-free-tt} the separation between free and bound variables will be even more pronounced, as only the former ones are annotated with types.

We write $e[s/x]$ for the substitution of an expression~$s$ for a bound variable~$x$ in expression~$e$ and $e[\vec s/\vec x]$ for the (parallel) substitution of $s_1, \ldots, s_n$ for $x_1, \ldots, x_n$, with the usual proviso about avoiding the capture of bound variables.
In \cref{sec:meta-theorems-raw}, when we prove admissibility of substitution, we shall also substitute expressions for free variables, which of course is written as $e[s/\mkvar{a}]$. Elsewhere we avoid such substitutions and only ever replace free variables by bound ones, in which case we write $e[x/\sym{a}]$. This typically happens when an expression with a free variable is used as part of a binder, such as the codomain of a $\Pi$-type or the body of a lambda. We take care to always keep bound variables well-scoped under binders.

\subsubsection{Arities and signatures}
\label{sec:signatures}

The raw expressions of a finitary type theory are formed using \defemph{symbols} and \defemph{metavariables}, which constitute two separate syntactic classes. Each symbol and metavariable has an associated arity, as follows.

The \defemph{symbol arity $(c, [(c_1, n_1), \ldots, (c_k, n_k)])$} of a symbol~$\symS$ tells us that
\begin{enumerate}
\item the syntactic class of $\symS$ is $c \in \mkset{\Ty, \Tm}$,
\item $\symS$ accepts $k$ arguments,
\item the $i$-th argument must have syntactic class~$c_i \in \mkset{\Ty, \Tm, \EqTy, \EqTm}$ and binds~$n_i$ variables.
\end{enumerate}
The syntactic classes $\Ty$ and $\Tm$ stand for type and term expressions, and~$\EqTy$ and~$\EqTm$ for type and term equations, respectively. For the time being the latter two are mere formalities, as the only expression of these syntactic classes are the dummy values~$\dummy_\Ty$ and $\dummy_\Tm$. However, in \cref{sec:context-free-tt} we will introduce genuine expressions of syntactic classes~$\EqTy$ and~$\EqTm$.

The information about symbol arities is collected in a \defemph{signature~$\Sigma$}, which maps each symbol to its arity. When discussing syntax, it is understood that such a signature has been given, even if we do not mention it explicitly.

\begin{example}
  The arity of a type constant such as $\mathsf{bool}$ is $(\Ty, [])$, the arity of a binary term operation such as~$+$ is $(\Tm, [(\Tm, 0), (\Tm, 0)])$, and the arity of a quantifier such as the dependent product~$\Uppi$ is $(\Ty, [(\Ty, 0), (\Ty, 1)])$ because it is a type former taking two type arguments, with the second one binding one variable.
\end{example}

The \defemph{metavariable arity} associated to a metavariable $\symM$ is a pair $(c, n)$, where the syntactic class $c \in \mkset{\Ty, \Tm, \EqTy, \EqTm}$ indicates whether~$\symM$ is respectively a type, term, type equality, or term equality metavariable, and $n$ is the number of term arguments it accepts. The metavariables of syntactic classes $\Ty$ and $\Tm$ are the \defemph{object} metavariables, and can be used to form expressions. The metavariable of syntactic classes $\EqTy$ and $\EqTm$ are the \defemph{equality} metavariables, and do not participate in formation of expressions. We introduce them to streamline several definitions, and to have a way of referring to equational premises in \cref{sec:context-free-tt}.
The information about metavariable arities is collected in a metavariable context, cf.\ \cref{sec:judg-bound}.

A metavariable $\symM$ of arity $(c, n)$ could be construed as a symbol of arity
\begin{equation*}
(c, [\underbrace{(\Tm,0), \ldots, (\Tm,0)]}_n) .
\end{equation*}
This apporach is taken in~\cite{gtt}, but we keep metavariables and symbols separate because they play different roles, especially in context-free type theories in \cref{sec:context-free-tt}.

\subsubsection{Raw expressions}
\label{sec:expressions}

The raw syntactic constituents of a finitary type theory, with respect to a given signature~$\Sigma$, are outlined in \cref{fig:syntax-general-type-theories}. In this section we discuss the top part of the figure, which involves the syntax of term and type expressions, and arguments.

\begin{figure}[thbp]
  \centering
  \small
  \begin{ruleframe}[innerleftmargin=6pt,innerrightmargin=6pt]
  \begin{align*}
  \text{Type expression}\ A, B
  \bnfis& \symS(e_1, \ldots, e_n)   &&\text{type symbol application}\\
  \bnfor& \symM(t_1, \ldots, t_n)   &&\text{type metavariable application}
  \\
  \text{Term expression}\ s, t
  \bnfis& \mkvar a                         &&\text{free variable}\\
  \bnfor& x                              &&\text{bound variable}\\
  \bnfor& \symS(e_1, \ldots, e_n)   &&\text{term symbol application}\\
  \bnfor& \symM(t_1, \ldots, t_n)   &&\text{term metavariable application}
  \\
  \text{Argument}\ e
  \bnfis& A           &&\text{type argument} \\
  \bnfor& t           &&\text{term argument} \\
  \bnfor& \dummy_\Ty  &&\text{dummy argument} \\
  \bnfor& \dummy_\Tm  &&\text{dummy argument} \\
  \bnfor& \abstr x e  &&\text{abstracted arg. ($x$ bound in $e$)}
  \\[1ex]
  \text{Judgement thesis}\ \J
  \bnfis& \isType{A}                    && \text{$A$ is a type} \\
  \bnfor& t : A                         && \text{$t$ has type $T$} \\
  \bnfor& A \equiv B \by \dummy_\Ty     && \text{$A$ and $B$ are equal types} \\
  \bnfor& s \equiv t : A \by \dummy_\Tm && \text{$s$ and $t$ are equal terms at $A$}
  \\
  \text{Abstracted judgement}\ \JJ
  \bnfis& \J                   &&\text{judgement thesis} \\
  \bnfor& \abstr{x \of A}\; \JJ  &&\text{abstracted jdgt. ($x$ bound in $\JJ$)}
  \\[1ex]
  \text{Boundary thesis}\ \bdry
  \bnfis& \isType{\Box}            &&\text{a type}\\
  \bnfor& \Box :  A                &&\text{a term of type $A$}\\
  \bnfor& A \equiv B \by \Box      &&\text{type equation boundary}\\
  \bnfor& s \equiv t : B \by \Box  &&\text{term equation boundary}
  \\
  \text{Abstracted boundary}\ \BB
  \bnfis& \bdry                   &&\text{boundary thesis} \\
  \bnfor& \abstr{x \of A}\; \BB  &&\text{abstracted bdry. ($x$ bound in $\BB$)}
  \\[1ex]
  \text{Variable context}\ \Gamma
  \bnfis& \mathrlap{[\mkvar{a}_1 \of A_1, \ldots, \mkvar{a}_n \of A_n]}
  \\[1ex]
  \text{Metavariable context}\ \Theta
  \bnfis& \mathrlap{[\symM_1 \of \BB_1, \ldots, \symM_n \of \BB_n]}
  \\[1ex]
  \text{Hypothetical judgement}\
  \phantom{\bnfis}& \Theta; \Gamma \types \JJ \\
  \text{Hypothetical boundary}\
  \phantom{\bnfis}& \Theta; \Gamma \types \BB
  \end{align*}
  \end{ruleframe}
  \caption{The raw syntax of expressions, boundaries and judgements.}
  \label{fig:syntax-general-type-theories}
\end{figure}

A \defemph{type expression}, or just a \defemph{type}, is formed by an application $\symS(e_1, \ldots, e_n)$ of a type symbol to arguments, or an application $\symM(t_1, \ldots, t_n)$ of a type metavariable to term expressions.
A \defemph{term expression}, or just a \defemph{term}, is a free variable~$\mkvar{a}$, a bound variable~$x$, an application $\symS(e_1, \ldots, e_n)$ of a term symbol to arguments, or an application $\symM(t_1, \ldots, t_n)$ of a term metavariable to term expressions.

An \defemph{argument} is a type or a term expression, the dummy argument~$\dummy_\Ty$ of syntactic class~$\EqTy$, or the dummy argument~$\dummy_\Tm$ of syntactic class~$\EqTm$.
We write just $\dummy$ when it is clear which of the two should be used.
Another kind of argument is an \defemph{abstraction} $\abstr{x} e$, which binds~$x$ in~$e$.
An iterated abstraction $\abstr{x_1} \abstr{x_2} \cdots \abstr{x_n} e$ is abbreviated as $\abstr{\vec{x}} e$.
Note that abstraction is a primitive syntactic operation, and that it provides no typing information about~$x$.

\begin{example}
  In our notation a dependent product is written as $\Uppi(A, \abstr{x} B)$, and a fully annotated function as $\uplambda(A, \abstr{x} B, \abstr{x} e)$. The fact that $x$ ranges over~$A$ is not part of the raw syntax and will be specified later by an inference rule.
\end{example}

In all cases, in order for an expression to be well-formed, the arities of symbols and metavariables must be respected. If~$\symS$ has arity $(c, [(c_1, n_1), \ldots, (c_k, n_k)])$, then it must be applied to~$k$ arguments $e_1, \ldots, e_k$, where each $e_i$ is of the form $\abstr{x_1} \cdots \abstr{x_{n_i}} e_i'$ with $e_i'$ a non-abstracted argument of syntactic class~$c_i$.
Similarly, a metavariable $\symM$ of arity $(n, c)$ must be applied to~$n$ term expressions.
When a symbol~$\symS$ takes no arguments, we write the corresponding expression as $\symS$ rather than $\symS()$, and similarly for metavariables.

As is usual, expressions which differ only in the choice of names of bound variables are considered syntactically equal, e.g., $\abstr{x} \symS(\mkvar{a}, x)$ and $\abstr{y} \symS(\mkvar{a}, y)$ are syntactically equal and we may write $(\abstr{x} \symS(\mkvar{a}, x)) = (\abstr{y} \symS(\mkvar{a}, y))$.

For future reference we define in \cref{fig:variable-occurrences} the sets of free variable, bound variable, and metavariable occurrences, where we write set comprehension as $\asset{\cdots}$ in order to distinguish it from abstraction. A syntactic entity is said to be \defemph{closed} if no free variables occur in it.

\begin{figure}[htbp]
  \centering
  \begin{ruleframe}
    \small
    \centering

    \textbf{Free variables:}
    \begin{gather*}
      \fv{\mkvar{a}} = \asset{\mkvar{a}} \qquad
      \fv{x} = \asset{} \qquad
      \fv{\abstr{x} e} = \fv{e} \\
      \begin{aligned}
        \fv{\symS(e_1 \ldots e_n)} &= \fv{e_1} \cup \cdots \cup \fv{e_n} \\
        \fv{\symM(t_1 \ldots t_n)} &= \fv{t_1} \cup \cdots \cup \fv{t_n} \\
      \end{aligned}
      \\[1ex]
      \fv{\isType A} = \fv{A} \qquad
      \fv{t : A} = \fv{t} \cup \fv{A} \\
      \begin{aligned}
        \fv{A \equiv B \by \dummy} &= \fv{A} \cup \fv{B} \\
        \fv{s \equiv t : A \by \dummy} &= \fv{s} \cup \fv{t} \cup \fv{A} \\
        \fv{\abstr{x : A}\; \JJ} &= \fv{A} \cup \fv{\JJ} \\
      \end{aligned}
      \\[1ex]
      \fv{\isType \Box} = \asset{} \qquad
      \fv{\Box : A} = \fv{A} \\
      \begin{aligned}
        \fv{A \equiv B \by \Box} &= \fv{A} \cup \fv{B} \\
        \fv{s \equiv t : A \by \Box} &= \fv{s} \cup \fv{t} \cup \fv{A} \\
        \fv{\abstr{x : A\;} \BB} &= \fv{A} \cup \fv{\BB}
      \end{aligned}
    \end{gather*}

    \medskip

    \textbf{Bound variables:}
    \begin{gather*}
      \bv{\mkvar{a}} = \asset{} \qquad
      \bv{x} = \asset{x} \qquad
      \bv{\abstr{x} e} = \bv{e} \setminus \asset{x} \\
      \begin{aligned}
        \bv{\symS(e_1 \ldots e_n)} &= \bv{e_1} \cup \cdots \cup \bv{e_n} \\
        \bv{\symM(t_1 \ldots t_n)} &= \bv{t_1} \cup \cdots \cup \bv{t_n}
      \end{aligned}
    \end{gather*}

    \medskip

    \textbf{Metavariables:}
    \begin{gather*}
      \mv{\mkvar{a}} = \asset{} \qquad
      \mv{x} = \asset{} \qquad
      \mv{\abstr{x} e} = \mv{e} \\
      \begin{aligned}
        \mv{\symS(e_1 \ldots e_n)} &= \mv{e_1} \cup \cdots \cup \mv{e_n} \\
        \mv{\symM(t_1 \ldots t_n)} &= \asset{\symM} \cup \mv{t_1} \cup \cdots \cup \mv{t_n} \\
      \end{aligned}
      \\[1ex]
      \mv{\isType{A}} = \mv{A} \qquad
      \mv{t : A} = \mv{t} \cup \mv{A} \\
      \begin{aligned}
        \mv{A \equiv B \by \dummy} &= \mv{A} \cup \mv{B} \\
        \mv{s \equiv t : A \by \dummy} &= \mv{s} \cup \mv{t} \cup \mv{A} \\
        \mv{\abstr{x \of A}\; \JJ} &= \mv{A} \cup \mv{\JJ}
      \end{aligned}
      \\[1ex]
      \mv{\isType \Box} = \asset{} \qquad
      \mv{\Box : A} = \mv{A} \\
      \begin{aligned}
        \mv{A \equiv B \by \Box} &= \mv{A} \cup \mv{B} \\
        \mv{s \equiv t : A \by \Box} &= \mv{s} \cup \mv{t} \cup \mv{A} \\
        \mv{\abstr{x : A}\; \BB} &= \mv{A} \cup \mv{\BB}
      \end{aligned}
    \end{gather*}
  \end{ruleframe}
  \caption{Free, bound, and metavariable occurrences}
  \label{fig:variable-occurrences}
\end{figure}

\subsubsection{Judgements and boundaries}
\label{sec:judg-bound}

The bottom part of \cref{fig:syntax-general-type-theories} displays the syntax of judgements and boundaries, which we discuss next.

There are four \defemph{judgement forms}: ``$\isType{A}$'' asserts that $A$ is a type; ``$t : A$'' that $t$ is a term of type~$A$; ``$A \equiv B \by \dummy_\Ty$'' that types $A$ and $B$ are equal; and ``$s \equiv t : A \by \dummy_\Tm$'' that terms $s$ and $t$ of type $A$ are equal.
We may shorten the equational forms to ``$A \equiv B$'' and ``$s \equiv t : A$'' in this section, as the only possible choice for $\by$ is $\dummy$.

Less familiar, but equally fundamental, is the notion of a \defemph{boundary}.
Whereas a judgement is an assertion, a boundary is a \emph{question} to be answered, a \emph{promise} to be fulfilled, or a \emph{goal} to be accomplished:
``$\isType{\Box}$'' asks that a type be constructed; ``$\Box : A$'' that the type $A$ be inhabited; and ``$A \equiv B \by \Box$'' and ``$s \equiv t : A \by \Box$'' that equations be proved.

An \defemph{abstracted judgement} has the form $\abstr{x \of A}\; \JJ$, where $A$ is a type expression and $\JJ$ is a (possibly abstracted) judgement. The variable~$x$ is bound in $\JJ$ but not in~$A$. Thus in general an abstracted judgement has the form
\begin{equation*}
  \abstr{x_1 \of A_1} \cdots \abstr{x_n \of A_n} \; \J,
\end{equation*}
where $\J$ is a \defemph{judgement thesis}, i.e.\ an expression taking one of the four (non-abstracted) judgement forms. We may abbreviate such an abstraction as $\abstr{\vec{x} \of \vec{A}} \; \J$.
Analogously, an \defemph{abstracted boundary} has the form
\begin{equation*}
  \abstr{x_1 \of A_1} \cdots \abstr{x_n \of A_n} \; \bdry,
\end{equation*}
where~$\bdry$ is a \defemph{boundary thesis}, i.e.\ it takes one of the four (non-abstracted) boundary forms. The reason for introducing abstracted judgements and boundaries will be explained shortly.

An abstracted boundary has the associated metavariable arity
\begin{equation*}
  \arity{\abstr{x_1 \of A_1} \cdots \abstr{x_n \of A_n}\; \bdry}
  = (c, n)
\end{equation*}
where $c \in \mkset{\Ty, \Tm, \EqTy, \EqTm}$ is the syntactic class of~$\bdry$. Similarly, the associated metavariable arity of an argument is
\begin{equation*}
  \arity{\abstr{x_1} \cdots \abstr{x_n} e} = (c, n)
\end{equation*}
where $c \in \mkset{\Ty, \Tm}$ is the syntactic class of the (non-abstracted) expression~$e$.

The placeholder $\Box$ in a boundary $\BB$ may be \defemph{filled} with an argument~$e$, called the \defemph{head}, to give a judgement $\plug{\BB}{e}$, provided that the arities of~$\BB$ and~$e$ match.
Because equations are proof irrelevant, their placeholders can be filled uniquely with (suitably abstracted) dummy value~$\dummy$.
Filling is summarized in \cref{fig:boundary-filling}, where we also include notation for filling an object boundary with an equation that results in the corresponding equation. The figure rigorously explicates the dummy values, but we usually omit them.
Filling may be inverted: given an abstracted judgement $\JJ$ there is a unique abstracted boundary~$\BB$ and a unique argument~$e$ such that $\JJ = \plug{\BB}{e}$.

\begin{figure}[htbp]
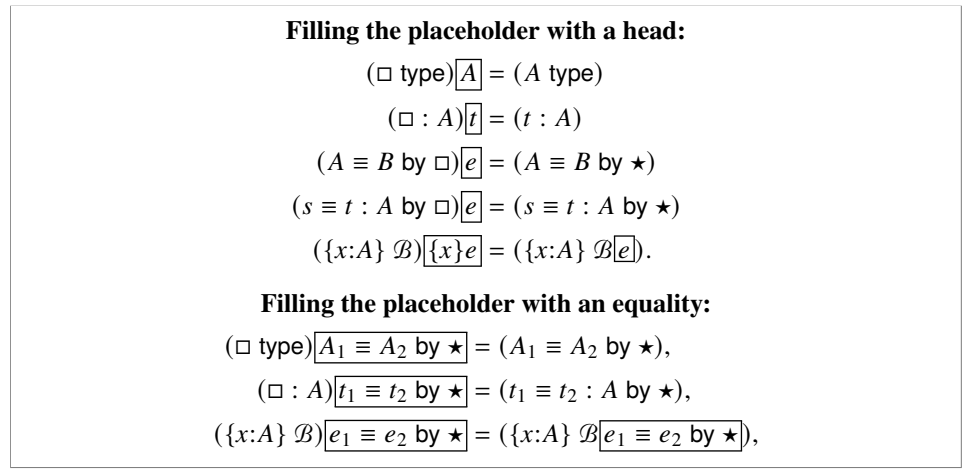

  \centering
  \begin{ruleframe}
    \small
    \centering

    \textbf{Filling the placeholder with a head:}
    \begin{align*}
      \plug{(\isType \Box)}{A} &= (\isType A) \\
      \plug{(\Box : A)}{t} &= (t : A) \\
      \plug{(A \equiv B \by \Box)}{e} &= (A \equiv B \by \dummy) \\
      \plug{(s \equiv t : A \by \Box)}{e} &= (s \equiv t : A \by \dummy) \\
      \plug{(\abstr{x \of A}\; \BB)}{\abstr{x} e} &= (\abstr{x \of A}\; \plug{\BB}{e}).
    \end{align*}

    \medskip

    \textbf{Filling the placeholder with an equality:}
    \begin{align*}
      \plug{(\isType{\Box})}{A_1 \equiv A_2 \by \dummy} &= (A_1 \equiv A_2 \by \dummy), \\
      \plug{(\Box : A)}{t_1 \equiv t_2 \by \dummy} &= (t_1 \equiv t_2 : A \by \dummy),\\
      \plug{(\abstr{x \of A}\; \BB)}{e_1 \equiv e_2 \by \dummy} &=
          (\abstr{x \of A}\; \plug{\BB}{e_1 \equiv e_2 \by \dummy}),
    \end{align*}
  \end{ruleframe}
  \caption{Filling the head of a boundary}
  \label{fig:boundary-filling}
\end{figure}

\begin{example}
  If the symbols $\sym{A}$  and  $\sym{Id}$ have arities
  \begin{equation*}
    (\Ty, []),
    \quad\text{and}\quad
    (\Ty, [(\Ty, 0), (\Tm, 0), (\Tm, 0)]),
  \end{equation*}
  respectively, then the boundaries
  \begin{equation*}
    \abstr{x \of \sym{A}} \abstr{y \of \sym{A}} \; \Box : \sym{Id}(\sym{A}, x, y)
    \qquad\text{and}\qquad
    \abstr{x \of \sym{A}} \abstr{y \of \sym{A}} \; x \equiv y : \sym{A} \by \Box
  \end{equation*}
  may be filled with heads $\abstr{x} \abstr{y} x$ and $\abstr{x} \abstr{y} \dummy$ to yield abstracted judgements
  \begin{equation*}
    \abstr{x \of \sym{A}} \abstr{y \of \sym{A}} \;  x : \sym{Id}(\sym{A}, x, y)
    \qquad\text{and}\qquad
    \abstr{x \of \sym{A}} \abstr{y \of \sym{A}} \; x \equiv y : \sym{A} \by \dummy.
  \end{equation*}
  Names of bound variables are immaterial, we would still get the same judgement if we filled the left-hand boundary with $\abstr{u} \abstr{v} u$ or $\abstr{y} \abstr{x} y$,
  but not with $\abstr{x} \abstr{y} y$.
\end{example}

Information about available metavariables is collected by a \defemph{metavariable context}, which is a finite list $\Theta = [\symM_1 \of \BB_1, \ldots, \symM_n \of \BB_n]$, also construed as a map, assigning to each metavariable $\symM_i$ a boundary $\BB_i$.
In \cref{sec:raw-rules}, the assigned boundaries will assign the typing of metavariable, while at the level of raw syntax they determine metavariable arities. That is, $\Theta$ assigns the metavariable arity $\arity{\BB_i}$ to~$\symM_i$.

A metavariable context $\Theta = [\symM_1 \of \BB_1, \ldots, \symM_n \of \BB_n]$ may be \emph{restricted} to a metavariable context $\upto{\Theta}{i} = [\symM_1 \of \BB_1, \ldots, \symM_{i-1} \of \BB_{i-1}]$.

The metavariable context~$\Theta$ is syntactically well formed when each~$\BB_i$ is a
syntactically well-formed boundary over~$\Sigma$ and~$\upto{\Theta}{i}$. In addition each $\BB_i$ must be closed, i.e.\ contain no free variables.

A \defemph{variable context}~$\Gamma = [\mkvar{a}_1 \of A_1, \ldots, \mkvar{a}_n \of A_n]$ over a metavariable context~$\Theta$ is a finite list of pairs written as $\mkvar{a}_i \of A_i$. It is considered syntactically valid when the variables $\mkvar{a}_1, \ldots, \mkvar{a}_n$ are all distinct, and for each~$i$ the type expression $A_i$ is valid with respect to the signature and the metavariable arities assigned by~$\Theta$, and the free variables occurring in~$A_i$ are among $\mkvar{a}_1, \ldots, \mkvar{a}_{i-1}$. A variable context $\Gamma$ yields a finite map, also denoted $\Gamma$, defined by $\Gamma(\mkvar{a}_i) = A_i$. The \emph{domain} of~$\Gamma$ is the set $\vert \Gamma \vert = \{\mkvar{a}_1, \ldots, \mkvar{a}_n\}$.

A \defemph{context} is a pair $\Theta; \Gamma$ consisting of a metavariable context~$\Theta$ and a variable context~$\Gamma$ over $\Theta$. A syntactic entity is considered syntactically valid over a signature and a context $\Theta; \Gamma$ when all symbol and metavariable applications respect the assigned arities, the free variables are among~$\vert\Gamma\vert$, and all bound variables are properly abstracted. It goes without saying that we always require all syntactic entities to be valid in this sense.

A \defemph{(hypothetical) judgement} has the form
\begin{equation*}
  \Theta ; \Gamma \types \JJ.
\end{equation*}
It differs from traditional notion of a judgement in a non-essential way, which nevertheless requires an explanation. First, the context of a hypothetical judgement
\begin{equation*}
  \Theta ; \mkvar{a}_1 \of A_1, \ldots, \mkvar{a}_m \of A_m
  \types
  \abstr{x_1 \of B_1} \cdots \abstr{x_m \of B_m}\; \J
\end{equation*}
provides information about metavariables, not just the free variables.
Second, the variables are split between the context $\mkvar{a}_1 \of A_1, \ldots, \mkvar{a}_n \of A_n$ on the left of~$\types$, and the abstraction $\abstr{x_1 \of B_1} \cdots \abstr{x_m \of B_m}$ on the right. It is useful to think of the former as the \emph{global} hypotheses that interact with other judgements, and the latter as \emph{local} to the judgement.
We could of course delegate the metavariable context to be part of the signature as is done in \cite{gtt}, and revert to the more familiar form
\begin{equation*}
  \mkvar{a}_1 \of A_1, \ldots, \mkvar{a}_n \of A_n, x_1 \of B_1, \ldots, x_m \of B_m \types \J
\end{equation*}
by joining the variable context and the abstraction, but we would still have to carry the metavariable information in the signature, and would lose the ability to explicitly mark the split between the global and the local parts.
The split will be especially important in \cref{sec:context-free-tt}, where the context will be removed, but the abstraction kept.

\defemph{Hypothetical boundaries} are formed in the same fashion, as
\begin{equation*}
  \Theta; \Gamma \types \BB.
\end{equation*}
The intended meaning is that $\BB$ is a well-typed boundary in context $\Theta; \Gamma$.

\subsubsection{Metavariable instantiations}
\label{sec:meta-vari-inst}

Metavariables are slots that can be instantiated with arguments. Suppose $\Theta = \finmap{\symM_1 \of \BB_1, \ldots, \symM_k \of \BB_k}$ is a metavariable context over a signature~$\Sigma$. An \defemph{instantiation of~$\Theta$ over} a context~$\Xi; \Gamma$ is a seqence $I = \finmap{\symM_1 \mto e_1, \ldots, \symM_k \mto e_k}$, representing a map that takes each $\symM_i$ to an argument~$e_i$ over $\Theta; \Gamma$ such that $\arity{\BB_i} = \arity{e_i}$.

An instantiation~$I = \finmap{\symM_1 \mto e_1, \ldots, \symM_k \mto e_k}$ of~$\Theta$ may be \emph{restricted} to an instantiation~$\upto{I}{i} = \finmap{\symM_1 \mto e_1, \ldots, \symM_{i-1} \mto e_{i-1}}$ of $\upto{\Theta}{i}$.

An instantiation~$I$ of $\Theta$ over $\Xi; \Gamma$ acts on an expression~$e$ over $\Theta; \Delta$ to give an expression $\act{I} e$ in which the metavariables are replaced by the corresponding expressions, as follows:
\begin{gather*}
  \act{I} x = x,\qquad
  \act{I} \mkvar{a} = \mkvar{a},\qquad
  \act{I} \dummy = \dummy,\qquad
  \act{I} (\abstr{x} e) = \abstr{x} (\act{I} e),\\
  \begin{aligned}
  \act{I} (\symS(e_1, \ldots, e_n)) &= \symS(\act{I} e_1, \ldots, \act{I} e_n),\\
  \act{I} (\symM_i(t_1, \ldots, t_n)) &=
    e_i[(\act{I} t_1)/x_1, \ldots, (\act{I} t_{n_i})/x_{n_i}].
  \end{aligned}
\end{gather*}
The instantiated expression $\act{I} e$ is valid for $\Xi ; \Gamma, \act{I} \Delta$.
Abstracted judgements and boundaries may be instantiated too:
\begin{align*}
  \act{I} (\isType A) &= (\isType{\act{I} A}), &
  \act{I} (t : A) &= (\act{I} t : \act{I} A),\\
  \act{I} (A \equiv B \by \dummy) &= (\act{I} A \equiv \act{I} B \by \dummy), &
  \act{I} (\abstr{x \of A}\; \JJ) &= \abstr{x \of \act{I} A}\; \act{I} \JJ,
\end{align*}
and by imagining that $\act{I} \Box = \Box$, the reader can tell how to instantiate a boundary.
Finally, a hypothetical judgement $\Theta; \Delta \types \JJ$ may be instantiated to $\Xi; \Gamma, \act{I} \Delta \types \act{I} \JJ$, and similarly for a hypothetical boundary.

\subsection{Deductive systems}
\label{sec:deductive-systems}

We briefly recall the notions of a deductive system, derivability, and a derivation tree; see for example~\cite{Aczel:introduction:1977} for an introduction.
A \defemph{(finitary) closure rule} on a set $S$ is a pair $([p_1, \ldots, p_n], q)$, also displayed as
\begin{equation*}
  \infer{p_1 \ \cdots \  p_n}{q},
\end{equation*}
where $\{p_1, \ldots, p_n\} \subseteq S$ are the \defemph{premises} and $q \in S$ is the \defemph{conclusion}.
Let $\Clos{S}$ be the set of all closure rules on~$S$.

A \defemph{deductive system} (also called a \emph{closure system}) on a set $S$ is a family of closure rules $C : R \to \Clos{S}$, indexed by a set~$R$ of rule names.
A set $D \subseteq S$ is said to be \defemph{deductively closed for~$C$} when, for all $i \in R$, if $C_i = ([p_1, \ldots, p_n], q)$ and $\mkset{p_1, \ldots, p_n} \subseteq D$, then $q \in D$.
The \defemph{associated closure operator} is the map $\pow{S} \to \pow{S}$ which takes $D \subseteq S$ to the least deductively closed supserset $\closure{D}$ of~$D$, which exists by Tarski's fixed-point theorem~\cite{tarski}.
We say that $q \in S$ is \defemph{derivable from hypotheses $H \subseteq S$} when $q \in \closure{H}$, and that it is \defemph{derivable in~$C$} when $q \in \closure{\emptyset}$.

A closure rule $([p_1, \ldots, p_n], q)$ is \defemph{admissible for $C$} when $q \in \closure{\mkset{p_1, \ldots, p_n}}$. That is, adjoining an admissible closure rule to a closure system has no effect on its associated closure operator.

Derivability is witnessed by well-founded trees, which are constructed as follows.
For each $q \in S$ let $\Der{C}{q}$ be generated inductively by the clause (where $\mathsf{der}$ is a formal tag):
\begin{itemize}
\item for every $i \in R$, if $C_i = ([p_1, \ldots, p_n], q)$ and $t_j \in \Der{C}{p_j}$ for all $j = 1, \ldots, n$, then $\mathsf{der}_i(t_1, \ldots, t_n) \in \Der{C}{q}$.
\end{itemize}
The elements of $\Der{C}{q}$ are \defemph{derivation trees} with \defemph{conclusion}~$q$.
Indeed, we may view $\mathsf{der}_i(t_1, \ldots, t_n)$ as tree with the root labeled by~$i$ and the subtrees $t_1, \ldots, t_n$. A leaf is a tree of the form $\mathsf{der}_j()$, which arises when the corresponding closure rule~$C_j$ has no
premises.

\begin{proposition}
  \label{prop:derivation-tree}%
  Given a closure system~$C$ on $S$, an element $q \in S$ is derivable in~$C$ if, and only if, there exists a derivation tree over $C$ whose conclusion is~$q$.
\end{proposition}

\begin{proof}
  The claim is that $T = \mkset{q \in S \such \some{t \in \Der{C}{q}} \top}$ coincides with~$\closure{C}$.
  The inclusion $\closure{C} \subseteq T$ holds because $T$ is deductively closed.
  The reverse inclusion $T \subseteq \closure{C}$ is established by induction on derivation trees.
\end{proof}

We remark that allowing infinitary closure rules brings with it the need for the axiom of choice, for it is unclear how to prove that~$T$ is deductively closed without the aid of choice.

It is evident that derivability and derivation trees are monotone in all arguments: if $S \subseteq S'$, $R \subseteq R'$, and the closure system $C' : R' \to \Clos{S'}$ restricts to $C : R \to \Clos{S}$, then any $q \in S$ derivable in $C$ is also derivable in $C'$ as an element of~$S'$. Moreover, any derivation tree in $\Der{C}{q}$ may be construed as a derivation tree in $\Der{C'}{q}$.

Henceforth we shall consider solely deductive systems on the set of hypothetical judgements and boundaries. Because we shall vary the deductive system, it is useful to write $\Theta; \Gamma \types_{C} \JJ$ when $(\Theta; \Gamma \types \JJ) \in \closure{C}$, and similarly for $\Theta; \Gamma \types_{C} \BB$.

\subsection{Raw rules and type theories}
\label{sec:raw-rules}

A type theory in its basic form is a collection of closure rules. Some closure rules are specified directly, but many are presented by \emph{inference rules} -- templates whose instantiations yield the closure rules.
We deal with the \emph{raw} syntactic structure of such rules first.

\begin{definition}
  \label{def:raw-rule}%
  A \defemph{raw rule} over a signature~$\Sigma$ is a hypothetical judgement over~$\Sigma$ of the form
  $
    \Theta; \emptyCtx \types \J
  $.
  We notate such a raw rule as
  \begin{equation*}
    \rawRule{\Theta}{\J}.
  \end{equation*}
  The elements of~$\Theta$ are the \defemph{premises} and $\J$ is the \defemph{conclusion}.
  We say that the rule is an \defemph{object rule} when~$\J$ is a type or a term judgement,
  and an \defemph{equality rule} when~$\J$ is an equality judgement.
\end{definition}

Defining inference rules as hypothetical judgements with empty contexts and empty abstractions permits in many situations uniform treatment of rules and judgements.
Note that the premises and the conclusion may not contain any free variables, and that the conclusion must be non-abstracted. Neither condition impedes expressivity of raw rules, because free variables and abstractions may be promoted to premises.

\begin{example}
  \label{ex:raw-rule}%
  To help the readers' intuition, let us see how \cref{def:raw-rule} captures a traditional inference rule, such as product formation
  \begin{equation*}
    \inferenceRuleLabel{Ty-$\Uppi$}{Ty-Pi}{
      \types \isType {\sym{A}} \\
      \types \abstr{x \of \sym{A}}\; \isType {\sym{B}(x)}
    }{
      \types \isType {\Uppi(\sym{A}, \abstr{x} \sym{B}(x))}
    }
  \end{equation*}
  The use of $\sym{A}$ and $\sym{B}$ in the premises reveals that their arities are $(\Ty, 0)$, and $(\Ty, 1)$, respectively. In fact, the premises assign boundaries to metavariables, namely each metavariable, applied generically, is the head of its boundary.
  If we pull out the metavariables from the heads of premises, the assignment becomes explicit:
  \begin{equation*}
    \infer{
      \sym{A} \ : \ (\isType \Box) \\
      \sym{B} \ : \ (\abstr{x \of \sym{A}} \; \isType \Box)
    }{
      \isType {\Uppi(\sym{A}, \abstr{x} \sym{B}(x))}
    }
  \end{equation*}
  This is just a different way of writing the raw rule
  \begin{equation*}
    \rawRule{
      \sym{A} \of (\isType{\Box}),\
      \sym{B} \of (\abstr{x \of A}\; \isType{\Box})
    }{
      \isType {\Uppi(\sym{A}, \abstr{x} \sym{B}(x))}
    }.
  \end{equation*}
\end{example}

\begin{example}%
  \label{ex:tt-equality-reflection}
  We may translate raw rules back to their traditional form by filling the heads with metavariables applied generically. For example, the reader may readily verify that the raw rule
  \begin{equation*}
    \rawRule{
      \sym{A} \of (\isType{\Box}),\
      \sym{s} \of (\Box : \sym{A}),\
      \sym{t} \of (\Box : \sym{A}),\
      \sym{p} \of (\Box : \sym{Id}(\sym{A}, \sym{s}, \sym{t}))
      }{
       \sym{s} \equiv \sym{t} : \sym{A} \by \dummy
      }
  \end{equation*}
  corresponds to the \emph{equality reflection} rule of extensional type theory that is traditionally written as
  \begin{equation*}%
    \inferenceRuleT{Eq-Reflect}{
      \types \isType{\sym{A}} \\
      \types \sym{s} : \sym{A} \\
      \types \sym{t} : \sym{A} \\
      \types \sym{p} : \sym{Id}(\sym{A}, \sym{s}, \sym{t})
    }{
      \types \sym{s} \equiv \sym{t} : \sym{A}
    }
  \end{equation*}
  For everyone's benefit, we shall display raw rules in traditional form, but use \cref{def:raw-rule} when formalities demand so.
\end{example}

It may be mystifying that there is no variable context~$\Gamma$ in a raw rule, for is it not the case that rules may be applied in arbitrary contexts? Indeed, \emph{closure} rules have contexts, but raw rules do not because they are just templates. The context appears once we instantiate the template, as follows.

\begin{definition}
  \label{def:raw-rule-instantiation}
  An \defemph{instantiation} of a raw rule $R = (\rawRule{\symM_1 \of \BB_1, \ldots, \symM_n \of \BB_n}{\plug{\bdry}{e}})$ over context~$\Theta; \Gamma$ is an instantiation $I = \finmap{\symM_1 \mto e_1, \ldots, \symM_n \mto e_n}$ of its premises over~$\Theta; \Gamma$.
  The associated \defemph{closure rule $\act{I} R$} is $([p_1, \ldots, p_n, q], r)$ where $p_i$ is
  $
    \Theta; \Gamma \types \plug{(\upact{I}{i} \BB_i)}{e_i}
  $,
  $q$ is $\Theta; \Gamma \types \act{I} \bdry$,
  and~$r$ is
  $
    \Theta; \Gamma \types \act{I} (\plug{\bdry}{e})
  $.
\end{definition}

We included among the premises the well-formedness of the instantiated boundary $\Theta; \Gamma \types \act{I} \bdry$, so that the conclusion is well-formed. We need the premise as an induction hypothesis in the proof of \cref{prop:presuppositivity}. In \cref{sec:meta-thm-finitary} we shall formulate well-formedness conditions that allow us to drop the boundary premise.

Of special interest are the rules that give type-theoretic meaning to primitive symbols. To define them, we need the boundary analogue of raw rules.

\begin{definition}
  \label{def:raw-rule-boundary}%
  A \defemph{raw rule-boundary} over a signature~$\Sigma$ is a hypothetical boundary over~$\Sigma$ of the form
  $
    \Theta; \emptyCtx \types \bdry
  $.
  We notate such a raw rule-boundary as
  \begin{equation*}
    \rawRule{\Theta}{\bdry}.
  \end{equation*}
  The elements of~$\Theta$ are the \defemph{premises} and $\bdry$ is the \defemph{conclusion boundary}.
  We say that the rule-boundary is an \defemph{object rule-boundary} when~$\bdry$ is a type or a term boundary,
  and an \defemph{equality rule-boundary} when $\bdry$ is an equality boundary.
\end{definition}

Here is how a rule-boundary generates a rule associated to a symbol.

\begin{definition}
  \label{def:symbol-rule}%
  Given a raw object rule-boundary
  \begin{equation*}
    \rawRule{\symM_1 \of \BB_1, \ldots, \symM_n \of \BB_n}{\bdry}
  \end{equation*}
  over~$\Sigma$, the \defemph{associated symbol arity} is $(c, [\arity{\BB_1}, \ldots, \arity{\BB_n}])$, where $c \in \mkset{\Ty, \Tm}$ is the syntactic class of~$\bdry$.
  The \defemph{associated symbol rule} for $\symS \not\in \vert\Sigma\vert$ is the raw rule
  \begin{equation*}
    \rawRule
      {\symM_1 \of \BB_1, \ldots, \symM_n \of \BB_n}
      {\bdry[\symS(\genapp{\symM}_1, \ldots, \genapp{\symM}_n)]}
  \end{equation*}
  over the extended signature $\extendMany{\Sigma}{\symS \mto (c, [\arity{\BB_1}, \ldots, \arity{\BB_n}])}$, where $\genapp{\symM}$ is the \defemph{generic application} of the metavariable~$\symM$ with associated boundary~$\BB$, defined as:
  \begin{enumerate}
  \item
    $\genapp{M} = \abstr{x_1} \cdots \abstr{x_k} \symM(x_1, \ldots, x_k)$ if $\arity{\BB} = (c, k)$ and $c \in \mkset{\Ty, \Tm}$,
  \item
    $\genapp{M} = \abstr{x_1} \cdots \abstr{x_k} \dummy$ if $\arity{\BB} = (c, k)$ and $c \in \mkset{\EqTy, \EqTm}$.
  \end{enumerate}
  A raw rule is said to be a \defemph{symbol rule} if it is the associated symbol rule for some symbol $\symS$.
\end{definition}

The above definition separates the rule-boundary from the head of the conclusion because the latter can be calculated from the former. It would be less economical to define a symbol rule directly as a raw rule, as we would still have to verify that the supplied head is the expected one. In examples we shall continue to display symbol rules in their traditional form.

\begin{example}
  According to \cref{def:symbol-rule}, the symbol rule for $\Uppi$ is generated by the rule-boundary
  \begin{equation*}
    \infer{
      \types \isType{\sym{A}} \\
      \types \abstr{x \of \sym{A}} \; \isType{\sym{B}(x)}
    }{
      \types \isType{\Box}
    }
  \end{equation*}
  Indeed, the associated symbol rule for $\Uppi$ is
  \begin{equation*}
    \infer{
      \types \isType{\sym{A}} \\
      \types \abstr{x \of \sym{A}} \; \isType{\sym{B}(x)}
    }{
      \types \isType{\Uppi(\sym{A}, \abstr{x} \sym{B}(x))}
    }
  \end{equation*}
  We allow equational premises in object rules. For example,
  \begin{equation*}
    \inferenceRule{Refl'}{
      \types \isType{\sym{A}} \\
      \types \sym{s} : \sym{A} \\
      \types \sym{t} : \sym{A} \\
      \types \sym{s} \equiv \sym{t} : \sym{A}
    }{
      \types \sym{refl}(\sym{A}, \sym{s}, \sym{t}, \dummy) : \sym{Id}(\sym{A}, \sym{s}, \sym{t})
    }
  \end{equation*}
  is a valid symbol rule, assuming $\sym{Id}$ and $\sym{refl}$ have their usual arities.
\end{example}

We also record the analogous construction of an equality rule from a given equality rule-boundary.

\begin{definition}
  \label{def:equality-rule}%
  Given an equality rule-boundary
  \begin{equation*}
    \rawRule
      {\symM_1 \of \BB_1, \ldots, \symM_n \of \BB_n}
      {\bdry},
  \end{equation*}
  the \defemph{associated equality rule} is
  \begin{equation*}
    \rawRule
      {\symM_1 \of \BB_1, \ldots, \symM_n \of \BB_n}
      {\plug{\bdry}{\dummy}}.
  \end{equation*}
\end{definition}

We next formulate the rules that all type theories share, starting with the most nitty-gritty ones, the congruence rules.

\begin{definition}
  \label{def:congruence-rule}
  The \defemph{congruence rules} associated with a raw object rule~$R$
  \begin{equation*}
    \rawRule{\symM_1 \of \BB_1, \ldots, \symM_n \of \BB_n}{\plug{\bdry}{e}}
  \end{equation*}
  are closure rules, with
  \begin{equation*}
    I = \finmap{\symM_1 \mto f_1, \ldots, \symM_n \mto f_n}
    \quad\text{and}\quad
    J = \finmap{\symM_1 \mto g_1, \ldots, \symM_n \mto g_n},
  \end{equation*}
  of the form
  \begin{equation*}
    \infer{
      {\begin{aligned}
      &\Theta; \Gamma \types \plug{(\upact{I}{i} \BB_i)}{f_i}  &&\text{for $i = 1, \ldots, n$}\\
      &\Theta; \Gamma \types \plug{(\upact{J}{i} \BB_i)}{g_i}  &&\text{for $i = 1, \ldots, n$}\\
      &\Theta; \Gamma \types \plug{(\upact{I}{i} \BB_i)}{f_i \equiv g_i} &&\text{for object boundary $\BB_i$} \\
      &\Theta; \Gamma \types \act{I} B \equiv \act{J} B        &&\text{if $\bdry = (\Box : B)$}
    \end{aligned}}
    }{
      \Theta; \Gamma \types \plug{(\act{I} \bdry)}{\act{I} e \equiv \act{J} e}
    }
  \end{equation*}
\end{definition}

In case of a term equation at type~$B$, the congruence rule has the additional premise $\Theta; \Gamma \types \act{I} B \equiv \act{J} B$, which ensures that the right-hand side of the conclusion $\act{J} e$ has type $\act{I} B$. Having the equation avaliable as a premise allows us to use it in the inductive proof of \cref{prop:presuppositivity}. In \cref{sec:meta-thm-finitary} we show that the rule without the premises is admissible under suitable conditions.

\begin{example}
  \label{ex:pi-congruence-rule}
  The congruence rule associated with the product formation rule from \cref{ex:raw-rule} is
  \begin{equation}
    \label{eq:pi-congruence-rule}
    \infer{
      {
        \begin{aligned}
          &\Theta; \Gamma \types \isType {A_1} &
          &\Theta; \Gamma \types \abstr{x \of A_1} \; \isType {B_1}
          \\
          &\Theta; \Gamma \types \isType {A_2} &
          &\Theta; \Gamma \types \abstr{x \of A_2} \; \isType {B_2}
          \\
          &\Theta; \Gamma \types A_1 \equiv A_2 &
          &\Theta; \Gamma \types \abstr{x \of A_1} \; B_1 \equiv B_2
        \end{aligned}
      }
    }{
      \Theta; \Gamma \types
      \Uppi(A_1, \abstr{x} B_1)
      \equiv
      \Uppi(A_2, \abstr{x} B_2)
    }
  \end{equation}
\end{example}

Next we have formation and congruence rules for the metavariables.
As metavariables are like symbols whose arguments are terms, it is not suprising that their rules are quite similar to symbol rules.

\begin{definition}
  \label{def:metavariable-rule}%
  Given a context $\Theta; \Gamma$ over~$\Sigma$ with
  $\Theta = [\symM_1 \of \BB_1, \ldots, \symM_n \of \BB_n]$,
  and $\BB_k = (\abstr{x_1 \of A_1} \cdots \abstr{x_m \of A_m}\; \bdry)$, the \defemph{metavariable rules} for $\symM_k$ are the closure rules of the form
  \begin{equation*}
    \inferenceRuleT{TT-Meta}
    {{\begin{aligned}
     &\mathrlap{\Theta(\symM_k) = \abstr{x_1 \of A_1} \cdots \abstr{x_m \of A_m}\; \bdry} \\
     &\Theta; \Gamma \types t_j : A_j[\upto{\vec{t}}{j}/\upto{\vec{x}}{j}]
      &&\text{for $j = 1, \ldots, m$}
     \\
     &\Theta; \Gamma \types \bdry[\vec{t}/\vec{x}]
    \end{aligned}}
    }{
      \Theta; \Gamma \types \plug{(\bdry[\vec{t}/\vec{x}])}{\symM_k(\vec{t})}
    }
  \end{equation*}
  where $\vec{x} = (x_1, \ldots, x_m)$ and $\vec{t} = (t_1, \ldots, t_m)$.
  Recall that $\upto {\vec t} j$ stands for $[t_1, \ldots, t_{j-1}]$.
  In the second line of premises, we thus substitute the preceding term arguments $t_1, \ldots, t_{j-1}$ for the bound variables $x_1, \ldots, x_{j-1}$ in each type $A_j$.
  The last premise ensures the well-formedness of the boundary of the conclusion, just like the definition of the closure rule associated to a raw rule (Def.~\ref{def:raw-rule-instantiation}).

  Furthermore, if $\bdry$ is an object boundary, then the \defemph{metavariable congruence
    rules} for $\symM_k$ are the closure rules of the form
  \begin{equation*}
    \inferenceRuleT{TT-Meta-Congr}
    {
     { \begin{aligned}
          &\Theta(\symM_k) = \abstr{x_1 \of A_1} \cdots \abstr{x_m \of A_m}\; \bdry\\
          &\Theta; \Gamma \types s_j :
              A_j[\upto{\vec{s}}{j}/\upto{\vec{x}}{j}]
          &\text{for $j = 1, \ldots, m$}
          \\
          &\Theta; \Gamma \types t_j :
              A_j[\upto{\vec{t}}{j}/\upto{\vec{x}}{j}]
          &\text{for $j = 1, \ldots, m$}
          \\
          &\Theta; \Gamma \types s_j \equiv t_j :
              A_j[\upto{\vec{s}}{j}/\upto{\vec{x}}{j}]
          &\text{for $j = 1, \ldots, m$}
          \\
          &\Theta; \Gamma \types C[\vec{s}/\vec{x}] \equiv C[\vec{t}/\vec{x}]
          &\text{if $\bdry = (\Box : C)$}
        \end{aligned} }
    }{
      \Theta; \Gamma \types
      \plug
      {(\bdry[\vec{s}/\vec{x}])}
      {\symM_k(\vec{s}) \equiv \symM_k(\vec{t})}
    }
  \end{equation*}
  where
  $\vec{s} = (s_1, \ldots, s_m)$ and
  $\vec{t} = (t_1, \ldots, t_m)$.
\end{definition}

We are finally ready to give a definition of type theory which is sufficient for explaining derivability.

\begin{definition}
  \label{def:raw-type-theory}
  A \defemph{raw type theory~$T$} over a signature~$\Sigma$ is a family of raw rules over~$\Sigma$, called the \defemph{specific rules} of~$T$.
  The \defemph{associated deductive system} of~$T$ consists of:
  \begin{enumerate}
  \item the \defemph{structural rules} over~$\Sigma$:
    \begin{enumerate}
    \item the \emph{variable}, \emph{metavariable}, \emph{metavariable congruence}, and \emph{abstraction} closure rules (\cref{fig:struct-rules}),
    \item the \emph{equality} closure rules, (\cref{fig:equality-rules}),
    \item the \emph{boundary} closure rules (\cref{fig:well-formed-boundaries});
    \end{enumerate}
  \item the instantiations of the specific rules of~$T$
    (\cref{def:raw-rule-instantiation});
  \item for each specific object rule of~$T$, the instantiations of the associated congruence rule (\cref{def:congruence-rule}).
  \end{enumerate}
  We write $\Gamma \types_{T} \JJ$ when $\Gamma \types \JJ$ is derivable with respect to the deductive system associated to~$T$, and similarly for $\Gamma \types_{T} \BB$.
\end{definition}

\begin{figure}[pt]
  \centering
  \small
  \begin{ruleframe}
  \begin{mathpar}
    \inferenceRule{TT-Var}
    {
      \mkvar a \in \vert\Gamma\vert
    }{
      \Theta; \Gamma \types \mkvar{a} : \Gamma(\mkvar{a})
    }

    \inferenceRule{TT-Meta}
    {{\begin{aligned}
     &\mathrlap{\Theta(\symM_k) = \abstr{x_1 \of A_1} \cdots \abstr{x_m \of A_m}\; \bdry} \\
     &\Theta; \Gamma \types t_j : A_j[\upto{\vec{t}}{j}/\upto{\vec{x}}{j}]
      &&\text{for $j = 1, \ldots, m$}
     \\
     &\Theta; \Gamma \types \bdry[\vec{t}/\vec{x}]
    \end{aligned}}
    }{
      \Theta; \Gamma \types \plug{(\bdry[\vec{t}/\vec{x}])}{\symM_k(\vec{t})}
    }

    \inferenceRule{TT-Meta-Congr}
    {
     { \begin{aligned}
          &\Theta(\symM_k) = \abstr{x_1 \of A_1} \cdots \abstr{x_m \of A_m}\; \bdry\\
          &\Theta; \Gamma \types s_j :
              A_j[\upto{\vec{s}}{j}/\upto{\vec{x}}{j}]
          &\text{for $j = 1, \ldots, m$}
          \\
          &\Theta; \Gamma \types t_j :
              A_j[\upto{\vec{t}}{j}/\upto{\vec{x}}{j}]
          &\text{for $j = 1, \ldots, m$}
          \\
          &\Theta; \Gamma \types s_j \equiv t_j :
              A_j[\upto{\vec{s}}{j}/\upto{\vec{x}}{j}]
          &\text{for $j = 1, \ldots, m$}
          \\
          &\Theta; \Gamma \types C[\vec{s}/\vec{x}] \equiv C[\vec{t}/\vec{x}]
          &\text{if $\bdry = (\Box : C)$}
        \end{aligned} }
    }{
      \Theta; \Gamma \types
      \plug
      {(\bdry[\vec{s}/\vec{x}])}
      {\symM_k(\vec{s}) \equiv \symM_k(\vec{t})}
    }

    \inferenceRule{TT-Abstr}
    {
      \Theta; \Gamma \types \isType A \\
      \mkvar{a} \not\in \vert\Gamma\vert \\
      \Theta; \Gamma, \mkvar{a} \of A \types \JJ[\mkvar{a}/x]
    }{
      \Theta; \Gamma \types \abstr{x \of A} \; \JJ
    }

  \end{mathpar}
  \end{ruleframe}
  \caption{Variable, metavariable and abstraction closure rules}
  \label{fig:struct-rules}
\end{figure}

\begin{figure}[pht]
  \centering
  \small
  \begin{ruleframe}
  \begin{mathpar}
  \inferenceRule{TT-EqTy-Refl}
  { \Theta; \Gamma \types \isType { A } }
  { \Theta; \Gamma \types A \equiv A }

  \inferenceRule{TT-EqTy-Sym}
  { \Theta; \Gamma \types A \equiv B }
  { \Theta; \Gamma \types B \equiv A }

  \inferenceRule{TT-EqTy-Trans}
  { \Theta; \Gamma \types A \equiv B \\
    \Theta; \Gamma \types B \equiv C }
  { \Theta; \Gamma \types A \equiv C }

  \inferenceRule{TT-EqTm-Refl}
  { \Theta; \Gamma \types t : A }
  { \Theta; \Gamma \types t \equiv t : A }

  \inferenceRule{TT-EqTm-Sym}
  { \Theta; \Gamma \types s \equiv t : A }
  { \Theta; \Gamma \types t \equiv s : A}

  \inferenceRule{TT-EqTm-Trans}
  { \Theta; \Gamma \types s \equiv t : A \\
    \Theta; \Gamma \types t \equiv u : A }
  { \Theta; \Gamma \types s \equiv u : A }

  \inferenceRule{TT-Conv-Tm}
  { \Theta; \Gamma \types t : A \\
    \Theta; \Gamma \types A \equiv B }
  { \Theta; \Gamma \types t : B }

  \inferenceRule{TT-Conv-EqTm}
  {
    \Theta; \Gamma \types s \equiv t : A \\
    \Theta; \Gamma \types A \equiv B }
  { \Theta; \Gamma \types s \equiv t : B }
  \end{mathpar}
  \end{ruleframe}
  \caption{Equality closure rules}
  \label{fig:equality-rules}
\end{figure}

\begin{figure}[pht]
  \centering
  \small
  \begin{ruleframe}
  \begin{mathpar}
    \inferenceRule{TT-Bdry-Ty}
    {
    }{
     \Theta; \Gamma \types \isType \Box
    }

    \inferenceRule{TT-Bdry-Tm}
    {
      \Theta; \Gamma \types \isType{A}
    }{
      \Theta; \Gamma \types \Box : A
    }

    \inferenceRule{TT-Bdry-EqTy}
    {
      \Theta; \Gamma \types \isType{A} \\
      \Theta; \Gamma \types \isType{B}
    }{
      \Theta; \Gamma \types A \equiv B \by \Box
    }

    \inferenceRule{TT-Bdry-EqTm}
    {
      \Theta; \Gamma \types \isType{A} \\
      \Theta; \Gamma \types s : A \\
      \Theta; \Gamma \types t : A
    }{
      \Theta; \Gamma \types s \equiv t : A \by \Box
    }

    \inferenceRule{TT-Bdry-Abstr}
    {
      \Theta; \Gamma \types \isType A \\
      \mkvar{a} \not\in \vert\Gamma\vert \\
      \Gamma, \mkvar{a} \of A \types \BB[\mkvar{a}/x]
    }{
      \Theta; \Gamma \types \abstr{x \of A} \; \BB
    }
  \end{mathpar}
  \end{ruleframe}
  \caption{Well-formed abstracted boundaries}
  \label{fig:well-formed-boundaries}
\end{figure}

\begin{figure}[pht]
  \centering
  \small
  \begin{ruleframe}
  \begin{mathpar}
    \inferenceRule{MCtx-Empty}
    {
    }{
      \types \isMCtx{\emptyExt}
    }

    \inferenceRule{MCtx-Extend}
    {
      \types \isMCtx{\Theta} \\
      \Theta ; \emptyCtx \types \BB \\
      \symM \not\in \vert\Theta\vert
    }{
      \types \isMCtx{\finmap{\Theta, \symM \of \BB}}
    }

    \\

    \inferenceRule{VCtx-Empty}
    {
    }{
      \Theta \types \isVCtx{\emptyCtx}
    }

    \inferenceRule{VCtx-Extend}
    {
      \Theta \types \isVCtx{\Gamma} \\
      \Theta; \Gamma \types \isType{A} \\
      \mkvar{a} \not\in \vert\Gamma\vert
    }{
      \Theta \types \isVCtx{(\Gamma, \mkvar{a} \of A)}
    }
  \end{mathpar}
  \end{ruleframe}
  \caption{Well-formed metavariable and variable contexts}
  \label{fig:contexts}
\end{figure}

Several remarks are in order regarding the above definition and the rules in \cref{fig:struct-rules,fig:equality-rules,fig:well-formed-boundaries}:
\begin{enumerate}

\item
  It is assumed throughout that all the entities involved are syntactically valid, i.e.\ that arities are respected and variables are well-scoped.

\item
  The metavariable rules \rref{TT-Meta} and \rref{TT-Meta-Congr} are exactly as in \cref{def:metavariable-rule}.

\item
  The rules \rref{TT-Var}, \rref{TT-Meta}, and \rref{TT-Abstr} contain \emph{side-conditions}, such as $\mkvar{a} \in \vert\Gamma\vert$ and $\Theta(\symM) = \abstr{x_1 \of A_1} \cdots \abstr{x_m \of A_m} \; \bdry$.
  For purely aesthetic reasons, these are written where premises ought to stand.
  For example, the correct way to read \rref{TT-Abstr} is: ``For all $\Theta$, $\Gamma$, $A$, $\mkvar{a}$, $\JJ$, if $\mkvar{a} \not\in \vert\Gamma\vert$, then there is a closure rule with premises $\Theta; \Gamma \types \isType{A}$ and $\Theta; \Gamma, \mkvar{a} \of A \types \JJ[\mkvar{a}/x]$, and the conclusion $\Theta; \Gamma \types \abstr{x \of A}\; \JJ$.''

\item
  The structural rules impose no well-typedness conditions on contexts. Instead, \cref{fig:contexts} provides two auxiliary judgement forms, ``$\types \isMCtx{\Theta}$'' and ``$\Theta \types \isVCtx{\Gamma}$'', stating that $\Theta$ is a well-typed metavariable context, and~$\Gamma$ a well-typed variable context over~$\Theta$, respectively. These will be used as necessary.
  Note that imposing the additional premise $\Theta; \Gamma \types \isType{\Gamma(\mkvar{a})}$ in \rref{TT-Var} would \emph{not} ensure well-formednes of~$\Gamma$, as not all variables need be accessed in a derivation. Requiring that \rref{TT-Meta} check the boundary of the metavariable is similarly ineffective.
\item
  We shall show in \cref{sec:meta-theorems-raw} that substitution rules (\cref{fig:substitution-rules}) are admissible.
\end{enumerate}

This may be a good moment to record the difference between derivability and admissibility.

\begin{definition}
  Consider a raw theory~$T$ and a raw rule~$R$, both over a signature~$\Sigma$:
  \begin{enumerate}
  \item $R$ is \defemph{derivable} in $T$ when it has a derivation in~$T$.
  \item $R$ is \defemph{admissible} in~$T$ when, for every instantiation~$I$ of~$R$, the conclusion of~$\act{I} R$ is derivable in~$T$ from the premises of~$\act{I} R$.
  \end{enumerate}
\end{definition}

\subsection{Finitary rules and type theories}
\label{sec:finitary-rules-and-tts}

Raw rules are \emph{syntactically} well-behaved: the premises and the conclusion are syntactically well-formed entities, and all metavariables, free variable and bound variables well-scoped. Nevertheless, a raw rule may be ill-formed for type-theoretic reasons, a deficiency rectified by the next definition.

Recall that a \defemph{well-founded order} on a set~$I$ is an irreflexive and transitive relation~$\prec$ satisfying, for each $S \subseteq I$,
\begin{equation*}
  (\all{i \in I} (\all{j \prec i} j \in S) \lthen i \in S) \lthen S = I.
\end{equation*}
The logical reading of the above condition is an induction principle: in order to show $\all{x \in I} \phi(x)$ one has to prove, for any $i \in I$, that $\phi(i)$ holds assuming that $\phi(j)$ does for all $j \prec i$.

\begin{definition}
  \label{def:finitary}%
  Given a raw theory~$T$ over a signature~$\Sigma$, a raw rule
  $\rawRule{\Theta}{\plug{\bdry}{e}}$
  over $\Sigma$ is \defemph{finitary} over~$T$ when $\types_T \isMCtx{\Theta}$ and $\Theta; \emptyCtx \types_T \bdry$.
  Similarly, a raw rule-boundary $\rawRule{\Theta}{\bdry}$ is finitary when
  $\types_T \isMCtx{\Theta}$ and $\Theta; \emptyCtx \types_T \bdry$.

  A \defemph{finitary type theory} is a raw type theory $(R_i)_{i \in I}$ for which there exists a well-founded order $(I, {\prec})$ such that each $R_i$ is finitary over $(R_j)_{j \prec i}$.
\end{definition}

\begin{example}
  \label{ex:rule-examples}%
  We take stock by considering several examples of rules.
  The rule
  \begin{equation*}
    \inferenceRule{Unique-Ty}{
      \types \isType{\sym{A}} \\
      \types \isType{\sym{B}} \\
      \types \isType{\sym{t} : \sym{A}} \\
      \types \isType{\sym{t} : \sym{B}}
    }{
      \types \sym{A} \equiv \sym{B}
    }
  \end{equation*}
  is not raw because it introduces the metavariable $\sym{t}$ twice.
  Assuming $\Uppi$ has arity $(\Ty, [(\Ty,0), (\Ty,1)])$,
  consider the rules
  \begin{mathpar}
    \inferenceRuleLabel{Ty-$\Uppi$-Short}{Ty-Pi-Short}{
      \types \abstr{x \of \sym{A}} \; \isType{\sym{B}(x)}
    }{
      \types \Uppi(\sym{A}, \abstr{x} \sym{B}(x))
    }

    \inferenceRuleLabel{Ty-$\Uppi$-Long}{Ty-Pi-Long}{
      \types \isType{\sym{A}} \\
      \types \abstr{x \of \sym{A}} \; \isType{\sym{B}(x)}
    }{
      \types \Uppi(\sym{A}, \abstr{x} \sym{B}(x))
    }
  \end{mathpar}
  The rule \rulename{Ty-$\Uppi$-Short} is not raw because it fails to introduce the metavariable~$\sym{A}$, while \rulename{Ty-$\Uppi$-Long} is finitary over any theory.
  The rule
  \begin{equation*}
    \inferenceRule{Succ-Congr-Typo}{
      \types \sym{m} : \sym{nat} \\
      \types \sym{n} : \sym{bool} \\
      \types \sym{m} \equiv \sym{n} : \sym{nat}
    }{
      \types \sym{succ}(\sym{m}) \equiv \sym{succ}(\sym{n}) : \sym{nat}
    }
  \end{equation*}
  is raw when the symbols $\sym{bool}$, $\sym{nat}$, and $\sym{succ}$ respectively have arities $(\Ty, [])$, $(\Ty, [])$, and $(\Tm, [(\Tm,0)])$.
  Whether it is also finitary depends on a theory. For instance, given the raw rules
  \begin{mathpar}
    \inferenceRule{Ty-Bool}{ }{\types \isType{\sym{bool}}}

    \inferenceRule{Ty-Nat}{ }{\types \isType{\sym{nat}}}

    \inferenceRule{Tm-Succ}{
      \types \sym{n} : \sym{nat}
    }{
      \types \sym{succ}(\sym{n}) : \sym{nat}
    }

    \inferenceRule{Bool-Eq-Nat}{ }{ \types \sym{bool} \equiv \sym{nat} }
  \end{mathpar}
  the rule \rref{Succ-Congr-Typo} is not finitary over the first three rules, but is finitary over all four of them.
  As a last example, given the symbol $\sym{Id}$ with arity $(\Ty, [(\Ty,0), (\Tm,0), (\Tm,0)])$, the rules
  \begin{mathpar}
    \inferenceRule{Ty-Id}{
      \types \isType{\sym{A}} \\
      \types \sym{s} : \sym{A} \\
      \types \sym{t} : \sym{A}
    }{
      \types \isType{\sym{Id}(\sym{A}, \sym{s}, \sym{t})}
    }

    \inferenceRule{Ty-Id-Typo}{
      \types \isType{\sym{A}} \\
      \types \sym{s} : \sym{A} \\
      \types \sym{t} : \sym{A}
    }{
      \types \isType{\sym{Id}(\sym{A}, \sym{s}, \sym{s})}
    }

    \inferenceRule{Eq-Reflect}{
      \types \isType{\sym{A}} \\
      \types \sym{s} : \sym{A} \\
      \types \sym{t} : \sym{A} \\
      \types \sym{p} : \sym{Id}(\sym{A}, \sym{s}, \sym{t})
      }{
        \types \sym{s} \equiv \sym{t} : \sym{A}
      }
  \end{mathpar}
  are all raw, both \rref{Ty-Id} and \rref{Ty-Id-Typo} are finitary over an empty theory, while \rref{Eq-Reflect} is finitary over a theory containing \rref{Ty-Id}. The rule \rref{Ty-Id} is a symbol rule, but \rref{Ty-Id-Typo} does not.
\end{example}

Could we have folded \cref{def:raw-rule} of raw rules and \cref{def:finitary} of finitary rules into a single definition? Not easily, as that would generate a loop: finitary rules refer to theories and derivability, which refer to closure rules, which are generated from raw rules. Without a doubt something is to be learned by transforming the cyclic dependency to an inductive definition, but we do not attempt to do so here.

A finitary type theory is fairly well behaved from a type-theoretic point of view, but can still suffer from unusual finitary rules, such as \rref{Ty-Id-Typo} from \cref{ex:rule-examples}, which looks like a spelling mistake.
We thus impose a further restriction by requiring that every rule be either a symbol rule or an equality rule.

\begin{definition}
  \label{def:standard-type-theory}%
  A finitary type theory is \defemph{standard} if its specific object rules are symbol rules, and each symbol has precisely one associated rule.
\end{definition}

A standard type theory and its signature may be built iteratively as follows:
\begin{enumerate}

\item
  The empty theory is standard over the empty signature.

\item
  Given a standard type theory~$T$ over~$\Sigma$, and a rule-boundary
  \begin{equation*}
    \rawRule{\symM_1 \of \BB_1, \ldots, \symM_n \of \BB_n}{\bdry}
  \end{equation*}
  finitary for~$T$:
  \begin{itemize}
  \item
    If $\bdry$ is an object boundary, and $\symS \not\in \vert\Sigma\vert$, then~$T$ extended with the associated symbol rule
    \begin{equation*}
      \rawRule{\symM_1 \of \BB_1, \ldots, \symM_n \of \BB_n}{\plug{\bdry}{\symS(\genapp{\symM}_1, \ldots, \genapp{\symM}_n)}}
    \end{equation*}
    is standard over the extended signature $\extendMany{\Sigma}{\symS \mto \alpha}$,
    where $\alpha$ is the symbol arity associated with the rule-boundary.
  \item
    If $\bdry$ is an equation boundary, then $T$ extended with the equality rule
    \begin{equation*}
      \rawRule{\symM_1 \of \BB_1, \ldots, \symM_n \of \BB_n}{\plug{\bdry}{\dummy}}
    \end{equation*}
    is standard over~$\Sigma$.
  \end{itemize}
\end{enumerate}
A more elaborate well-founded induction may be employed when a theory features infinitely many rules, such as an infinite succession of universes.

\section{Meta-theorems}
\label{sec:meta-theorems}

We put our definitions to the test by proving meta-theorems which stipulate desirable structural properties of type theories. The theorems are all rather standard and expected. Nevertheless, we prove them to verify that our definition of type theories is sensible, and to provide general-purpose meta-theorems that apply in a wide range of situations.

Making the statements precise in full generality has not always been trivial. We therefore include them here, together with statements of auxiliary lemmas, to give the reader an overview of the technique, but mostly relegate the rather lengthy induction proofs to the appendix. We shall continue to do so in subsequent sections.

\subsection{Meta-theorems about raw theories}
\label{sec:meta-theorems-raw}

A \defemph{renaming} of an expression $e$ is an injective map $\rho$ with domain $\mv{e} \cup \fv{e}$ that takes metavariables to metavariables and free variables to free variables. The renaming acts on~$e$ to yield an expression $\act{\rho} e$ by replacing each occurrence of a metavariable~$\symM$ and a free variable~$\mkvar{a}$ with $\rho(\symM)$ and $\rho(\mkvar{a})$, respectively. We similarly define renamings of contexts, judgements, and boundaries.

\begin{proposition}[Renaming]
  \label{prop:tt-renaming}%
  If a raw type theory derives a judgement or a boundary, then it also derives its renaming.
\end{proposition}

\begin{proof}
  \label{proof:tt-renaming}
  Let $\rho$ be a renaming of a derivable judgement $\Theta; \Gamma \types \JJ$.
  We show that $\act{\rho} \Theta; \act{\rho} \Gamma \types \act{\rho} \JJ$ is derivable by induction on the derivation. The case of boundaries is similar.

  Most cases only require a direct application of the induction hypotheses to the premises.
  The only somewhat interesting case is \rref{TT-Abstr},
  \begin{equation*}
    \infer
    {
      \Theta; \Gamma \types \isType A \\
      \mkvar{a} \not\in \vert\Gamma\vert \\
      \Theta; \Gamma, \mkvar{a} \of A \types \JJ[\mkvar{a}/x]
    }{
      \Theta; \Gamma \types \abstr{x \of A} \; \JJ
    }
  \end{equation*}
  As $\mkvar a \not\in \vert\Gamma\vert$, and thus $\mkvar a \not\in \vert\rho\vert$, we may extend $\rho$ to a renaming $\rho' = \extend{\rho}{\mkvar{a}}{\mkvar{b}}$, where $\mkvar b$ is such that $\mkvar{b} \not\in \vert\act{\rho} \Gamma\vert$.
  By induction hypothesis for the first premise, $\act{\rho} \Theta; \act{\rho} \Gamma \types \isType{\act{\rho} A}$ is derivable. We apply the induction hypothesis for the second premise to~$\rho'$ and obtain $\act{\rho'} \Theta; \act{\rho'} (\Gamma, \mkvar{a} \of A) \types \act{\rho'} (\JJ[\mkvar{a}/x])$, which equals $\act{\rho} \Theta; \act{\rho} \Gamma, \mkvar{b} \of \act{\rho} A \types (\act{\rho} \JJ)[\mkvar{b}/x]$. Thus, we may conclude by \rref{TT-Abstr},
  \begin{equation*}
    \infer
    {
      \act{\rho} \Theta; \act{\rho} \Gamma \types \isType {(\act{\rho} A)} \\
      \mkvar{b} \not\in \vert\act{\rho} \Gamma\vert \\
      \act{\rho} \Theta; \act{\rho} \Gamma, \mkvar{b} \of \act{\rho} A \types (\act{\rho} \JJ)[\mkvar{b}/x]
    }{
      \act{\rho} \Theta; \act{\rho} \Gamma \types \abstr{x \of \act{\rho} A} \; \act{\rho} \JJ
    }
  \end{equation*} \qedhere
\end{proof}

\begin{proposition}[Weakening]
  \label{prop:tt-weakening}
  For a raw type theory:
  \begin{enumerate}
  \item If $\Theta; \Gamma_1, \Gamma_2 \types \JJ$ and $\mkvar{a} \not\in \vert\Gamma_1, \Gamma_2\vert$ then $\Theta; \Gamma_1, \mkvar{a} \of A, \Gamma_2 \types \JJ$.
  \item If $\Theta_1, \Theta_2; \Gamma \types \JJ$ and $\symM \not\in \vert\Theta_1, \Theta_2\vert$ then $\Theta_1, \symM \of \BB, \Theta_2; \Gamma \types \JJ$.
  \end{enumerate}
  An analogous statement holds for boundaries.
\end{proposition}

\begin{proof}
  Once again we proceed by induction on the derivation of the judgement in a straightforward manner, where
  the case \rref{TT-Abstr} relies on renaming (\cref{prop:tt-renaming}) to ensure that~$\mkvar{a}$ remains fresh in the subderivations.
\end{proof}

In several places we shall require well-formedness of contexts, a useful consequence of which we record first.

\begin{proposition}
  \label{prop:ctx-inversion}%
  If a raw type theory derives $\types \isMCtx{\Theta}$ then it derives $\Theta ; \emptyCtx \types \Theta(\symM)$ for every $\symM \in \vert\Theta\vert$;
  and if it derives $\Theta \types \isVCtx \Gamma$, then it derives $\Theta ; \Gamma \types \isType{\Gamma(\mkvar a)}$ for every $\mkvar a \in \vert\Gamma\vert$.
\end{proposition}
\begin{proof}
  By induction on the derivation of $\types \isMCtx{\Theta}$ and
  $\Theta \types \isVCtx{\Gamma}$, respectively, followed by weakening.
\end{proof}

\subsubsection{Admissibility of substitution}
\label{sec:admiss-subst}

In this section we prove that in a raw type theory substitution is admissible, and that substitution preserves judgemental equality.

\begin{restatable}{lemma}{restatepreparesubst}
  \label{lem:prepare-subst}%
  If a raw type theory derives $\Theta; \Gamma, \mkvar{a} \of A, \Delta \types \JJ$ and $\Theta; \Gamma \types t : A$ then it derives $\Theta; \Gamma, \Delta[t/\mkvar{a}] \types \JJ[t/\mkvar{a}]$.
\end{restatable}

\begin{proof}
  See the proof on \cpageref{proof:prepare-subst}.
\end{proof}

\begin{lemma}
  \label{lem:prepare-subst-bdry}%
  If a raw type theory derives $\Theta; \Gamma, \mkvar{a} \of A, \Delta \types \BB$ and $\Theta; \Gamma \types t : A$ then it derives $\Theta; \Gamma, \Delta[t/\mkvar{a}] \types \BB[t/\mkvar{a}]$.
\end{lemma}

\begin{proof}
  The base cases immediately reduce to the previous lemma. The case of \rref{TT-Bdry-Abstr} is similar to the case of \rref{TT-Abstr} in the previous lemma.
\end{proof}

\begin{restatable}{lemma}{restatettsubstandttconvabstr}
  \label{lem:tt-subst-and-tt-conv-abstr}%
  In a raw type theory the following rules are admissible:
  \begin{mathpar}
    \inferenceRuleT{TT-Subst}
    {
      \Theta; \Gamma \types \abstr{x \of A} \; \JJ
      \\
      \Theta; \Gamma \types t : A
    }{
      \Theta; \Gamma \types \JJ[t/x]
    }

    \inferenceRuleT{TT-Bdry-Subst}
    { \Theta; \Gamma \types \abstr{x \of A} \; \BB
      \\
      \Theta; \Gamma \types t : A
    }{
      \Theta; \Gamma \types \BB[t/x]
    }

    \inferenceRuleT{TT-Conv-Abstr}
    {
      \Theta; \Gamma \types \abstr{x \of A}\; \JJ
      \\
      \Theta; \Gamma \types \isType B
      \\
      \Theta; \Gamma \types A \equiv B
    }{
      \Theta; \Gamma \types \abstr{x \of B}\; \JJ
    }
  \end{mathpar}
\end{restatable}

\begin{proof}
  See the proof on \cpageref{proof:tt-subst-and-tt-conv-abstr}.
\end{proof}

The next lemma claims that substitution preserves equality, but is a bit finicky to state. Given terms $s$ and $t$, and an object judgement $\JJ$, define $\JJ[(s \equiv t)/\mkvar{a}]$ by
\begin{align*}
  (\isType{A})[(s \equiv t)/\mkvar{a}] &\ =\ (A[s/\mkvar{a}] \equiv A[t/\mkvar{a}]) \\
  (u : A)[(s \equiv t)/\mkvar{a}] &\ =\ (u[s/\mkvar{a}] \equiv u[t/\mkvar{a}] : A[s/\mkvar{a}]) \\
  (\abstr{x \of A}\; \JJ)[(s \equiv t)/\mkvar{a}] &\ =\ (\abstr{x \of A[s/\mkvar{a}]}\; \JJ[(s \equiv t)/\mkvar{a}]).
\end{align*}
That is, $\JJ[(s \equiv t)/\mkvar{a}]$ descends into abstractions by substituting~$s$ for~$\mkvar{a}$ in the types, and distributes types and terms over the equation $s \equiv t$.

\begin{restatable}{lemma}{restatepreparesubsteq}
  \label{lem:prepare-subst-eq}%
  If a raw type theory derives
  \begin{align}
  &\Theta; \Gamma \types s : A, \label{eq:prep-3} \\
  &\Theta; \Gamma \types t : A, \label{eq:prep-4} \\
  &\Theta; \Gamma \types s \equiv t : A. \label{eq:prep-5} \\
  &\Theta; \Gamma, \mkvar{a} \of A, \Delta \types \JJ, \label{eq:prep-2} \\
  &\Theta; \Gamma, \Delta[s/\mkvar{a}] \types B[s/\mkvar{a}] \equiv B[t/\mkvar{a}]
    \quad \text{for all $\mkvar{b} \in \vert\Delta\vert$ with $\Delta(\mkvar{b}) = B$}, \label{eq:prep-1}
  \end{align}
  then it derives
  \begin{enumerate}
  \item \label{it:prep-1} $\Theta; \Gamma, \Delta[s/\mkvar{a}] \types \JJ[s/\mkvar{a}]$,
  \item \label{it:prep-2} $\Theta; \Gamma, \Delta[s/\mkvar{a}] \types \JJ[t/\mkvar{a}]$, and
  \item \label{it:prep-3} $\Theta; \Gamma, \Delta[s/\mkvar{a}] \types \JJ[(s \equiv t)/\mkvar{a}]$
    if $\JJ$ is an object judgement.
  \end{enumerate}
\end{restatable}

\begin{proof}
  See the proof on \cpageref{proof:prepare-subst-eq}.
\end{proof}

\begin{figure}[pht]
  \centering
  \small
  \begin{ruleframe}
  \begin{mathpar}
    \inferenceRule{TT-Subst}
    { \Theta; \Gamma \types \abstr{x \of A} \; \JJ
      \\
      \Theta; \Gamma \types t : A
    }{
      \Theta; \Gamma \types \JJ[t/x]
    }

    \inferenceRule{TT-Bdry-Subst}
    { \Theta; \Gamma \types \abstr{x \of A} \; \BB
      \\
      \Theta; \Gamma \types t : A
    }{
      \Theta; \Gamma \types \BB[t/x]
    }

    \inferenceRule{TT-Subst-EqTy}
    { \Theta; \Gamma \types \abstr{x \of A} \abstr{\vec{y} \of \vec{B}} \; \isType{C} \\\\
      \Theta; \Gamma \types s : A \\
      \Theta; \Gamma \types t : A \\
      \Theta; \Gamma \types s \equiv t : A
    }{
      \Theta; \Gamma \types \abstr{\vec{y} \of \vec{B}[s/x]} \; C[s/x] \equiv C[t/x]
    }

    \inferenceRule{TT-Subst-EqTm}
    { \Theta; \Gamma \types \abstr{x \of A} \abstr{\vec{y} \of \vec{B}} \; u : C \\\\
      \Theta; \Gamma \types s : A \\
      \Theta; \Gamma \types t : A \\
      \Theta; \Gamma \types s \equiv t : A
    }{
      \Theta; \Gamma \types \abstr{\vec{y} \of \vec{B}[s/x]} \; u[s/x] \equiv u[t/x] : C[s/x]
    }

    \inferenceRule{TT-Conv-Abstr}
    { \Theta; \Gamma \types \abstr{x \of A}\; \JJ
      \\ \Theta; \Gamma \types \isType B \\ \Theta; \Gamma \types A \equiv B }
    { \Theta; \Gamma \types \abstr{x \of B}\; \JJ }
  \end{mathpar}
  \end{ruleframe}
  \caption{Admissible substitution rules}
  \label{fig:substitution-rules}
\end{figure}

\begin{theorem}[Admissibility of substitution]
  \label{thm:substitution-admissible}
  In a raw type theory, the closure rules from \cref{fig:substitution-rules} are admissible.
\end{theorem}

\begin{proof}
  We already established admissibility of \rref{TT-Subst}, \rref{TT-Bdry-Subst}, and \rref{TT-Conv-Abstr} in \cref{lem:tt-subst-and-tt-conv-abstr}.
  Both \rref{TT-Subst-EqTy} and \rref{TT-Subst-EqTm} are seen to be admissible the same way: invert the abstraction and apply \cref{lem:prepare-subst-eq} to derive the desired conclusion.
\end{proof}

We provide two more lemmas that allow us to combine substitutions and judgmental equalities more flexibly.

\begin{restatable}{lemma}{restatesubsteqsides}
  \label{lem:subst-eq-sides}%
  Suppose a raw type theory derives
  \begin{equation*}
    \Theta; \Gamma \types s : A, \qquad
    \Theta; \Gamma \types t : A, \quad\text{and}\quad
    \Theta; \Gamma \types s \equiv t : A.
  \end{equation*}
  \begin{enumerate}
  \item If it derives
    \begin{equation*}
      \Theta; \Gamma \types \abstr{x \of A} \abstr{\vec{y} \of \vec{B}} \; C \equiv D
      \quad\text{and}\quad
      \Theta; \Gamma \types \abstr{x \of A} \abstr{\vec{y} \of \vec{B}} \; \isType{D}
    \end{equation*}
    then it derives
    $
      \Theta; \Gamma \types \abstr{\vec{y} \of \vec{B}[s/x]} \; C[s/x] \equiv D[t/x].
    $

  \item If it derives
    \begin{equation*}
      \Theta; \Gamma \types
      \abstr{x \of A} \abstr{\vec{y} \of \vec{B}} \; u \equiv v : C
      \quad\text{and}\quad
      \Theta; \Gamma \types
      \abstr{x \of A} \abstr{\vec{y} \of \vec{B}} \; v : C
    \end{equation*}
    then it derives
    $
      \Theta; \Gamma \types
      \abstr{\vec{y} \of \vec{B}[s/x]} \; u[s/x] \equiv v[t/x] : C[s/x]
    $.
  \end{enumerate}
\end{restatable}

\begin{proof}
  See the proof on \cpageref{proof:subst-eq-sides}.
\end{proof}

\begin{restatable}{lemma}{restateeqsubstn}
  \label{lem:eq-subst-n}%
  Suppose a raw type theory derives, for $i = 1, \ldots, n$,
  \begin{align*}
    \Theta; \Gamma &\types s_i : A_i[\upto{\vec{s}}{i}/\upto{\vec{x}}{i}]
    \\
    \Theta; \Gamma &\types t_i : A_i[\upto{\vec{t}}{i}/\upto{\vec{x}}{i}]
    \\
    \Theta; \Gamma &\types s_i \equiv t_i : A_i[\upto{\vec{s}}{i}/\upto{\vec{x}}{i}].
  \end{align*}
  If it derives an object judgement
  $
  \Theta; \Gamma \types
  \abstr{\vec{x} \of \vec{A}}\;
  \plug{\BB}{e}
  $
  then it derives
  \begin{equation*}
    \Theta; \Gamma \types 
      \plug
      {(\BB[\vec{s}/\vec{x}])}
      {e[\vec{s}/\vec{x}] \equiv e[\vec{t}/\vec{x}]}.
  \end{equation*}
\end{restatable}

\begin{proof}
  See the proof on \cpageref{proof:eq-subst-n}.
\end{proof}

\subsubsection{Admissibility of instantiations}
\label{sec:admiss-inst}

We next turn to admissibility of instantiations, i.e.\ preservation of derivability under instantiation of metavariables by heads of derivable judgements.

\begin{definition}
  An instantiation $I = \finmap{\symM_1 \mto e_1, \ldots, \symM_n \mto e_n}$ of a metavariable context $\Xi = [\symM_1 \of \BB_1, \ldots, \symM_n \of \BB_n]$ over $\Theta; \Gamma$ is \defemph{derivable} when
  $\Theta; \Gamma \types \plug{(\upact{I}{k} \BB_k)}{e_k}$
  is derivable for $k = 1, \ldots, n$.
\end{definition}

\begin{restatable}{lemma}{restateinstantiationadmissible}
  \label{lem:instantiation-admissible}%
  In a raw type theory, let~$I$ be a derivable instantiation of~$\Xi$ over context~$\Theta; \Gamma$.
  If $\Xi; \Gamma, \Delta \types \JJ$ is derivable then so is $\Theta; \Gamma, \act{I} \Delta \types \act{I} \JJ$, and similarly for boundaries.
\end{restatable}

\begin{proof}
  See the proof on \cpageref{proof:instantiation-admissible}.
\end{proof}

\begin{theorem}[Admissibility of instantiation]
  \label{prop:instantiation-admissible}%
  In a raw type theory, let~$I$ be a derivable instantiation of~$\Xi$ over context~$\Theta; \Gamma$.
  If $\Xi; \Gamma \types \JJ$ is derivable then so is $\Theta; \Gamma \types \act{I} \JJ$, and similarly for boundaries.
\end{theorem}

\begin{proof}
  Apply \cref{lem:instantiation-admissible} with empty $\Delta$.
\end{proof}

We next show that, under favorable conditions, instantiating by judgementally equal instantiations leads to judgemental equality.
To make the claim precise, define the notation $\act{(I \equiv J)} \JJ$ by
\begin{align*}
  \act{(I \equiv J)} (\isType{A}) &= (\act{I} A \equiv \act{J} A \by \dummy),\\
  \act{(I \equiv J)} (t : A) &= (\act{I} t \equiv \act{J} t : \act{I} A \by \dummy),\\
  \act{(I \equiv J)} (\abstr{x \of A}\; \JJ) &= (\abstr{x \of \act{I} A}\; \act{(I \equiv J)} \JJ)
\end{align*}
and say that instantiations
\begin{equation*}
  I = \finmap{\symM_1 \mto e_1, \ldots, \symM_n \mto e_n}
  \qquad\text{and}\qquad
  J = \finmap{\symM_1 \mto f_1, \ldots, \symM_n \mto f_n}
\end{equation*}
of~$\Xi = [\symM_1 \of \BB_1, \ldots, \symM_n \of \BB_n]$ over~$\Theta; \Gamma$ are \defemph{judgementally equal} when, for $k = 1, \ldots, n$, if $\BB_k$ is an object boundary then
$
  \Theta; \Gamma \types \plug{(\upact{I}{k}\BB_k)}{e_k \equiv f_k}
$
is derivable.

\begin{restatable}{lemma}{restateeqinstadmit}
  \label{lem:eq-inst-admit}
  In a raw type theory, consider derivable instantiations $I$ and $J$
  of~$\Xi = [\symM_1 \of \BB_1, \ldots, \symM_n \of \BB_n]$ over~$\Theta; \Gamma$ which are judgementally equal.
  Suppose that $\types \isMCtx \Xi$ and $\Theta \types \isVCtx \Gamma$, and that $\Theta; \Gamma, \Delta \types \plug{(\upact I i \BB_i)}{J(\symM_i)}$ is derivable for $i = 1, \ldots, n$,
  and additionally that, for all $\mkvar{a} \in \vert\Delta\vert$ with $\Delta(\mkvar{a}) = A$, so are
  \begin{align*}
    \Theta; \Gamma, \act{I} \Delta &\types \isType{\act{I} A}, \\
    \Theta; \Gamma, \act{I} \Delta &\types \isType{\act{J} A}, \\
    \Theta; \Gamma, \act{I} \Delta &\types \act{I} A \equiv \act{J} A
  \end{align*}
  If $\Xi; \Gamma, \Delta \types \JJ$ is derivable then so are
  \begin{align}
    \label{eq:instEq-jdg-I}%
    \Theta; \Gamma, \act{I} \Delta &\types \act{I} \JJ, \\
    \label{eq:instEq-jdg-J}%
    \Theta; \Gamma, \act{I} \Delta &\types \act{J} \JJ, \\
    \label{eq:instEq-eq}%
    \Theta; \Gamma, \act{I} \Delta &\types \act{(I \equiv J)} \JJ
    \quad \text{if $\JJ$ is an object judgement.}
  \end{align}
\end{restatable}

\begin{proof}
  See the proof on \cpageref{proof:eq-inst-admit}.
\end{proof}

\Cref{lem:eq-inst-admit} imposes conditions on the instantiations and the context which can be reduced to the more familiar assumption of well-typedness of the context, using \Cref{lem:eq-inst-admit} itself, as follows.

\begin{restatable}{lemma}{restateinsteqbootstrapctx}
  \label{lem:inst-eq-bootstrap-ctx}
  In a raw type theory, consider
  $\Xi = [\symM_1 \of \BB_1, \ldots, \symM_n \of \BB_n]$ such that $\types \isMCtx{\Xi}$,
  and derivable instantiations
  \begin{equation*}
    I = \finmap{\symM_1 \mto e_1, \ldots, \symM_n \mto e_n}
    \qquad\text{and}\qquad
    J = \finmap{\symM_1 \mto f_1, \ldots, \symM_n \mto f_n}
  \end{equation*}
  of~$\Xi$ over~$\Theta; \Gamma$ which are judgementally equal.
  Suppose further that $\Theta \types \isVCtx{\Gamma}$ and $\Theta; \Gamma \types \plug{(\upact{I}{i} \BB_i)}{f_i}$ for $i = 1, \ldots, n$.
  If $\Theta \types \isVCtx{(\Gamma, \Delta)}$, then for all $\mkvar{a} \in \vert\Delta\vert$ with $\Delta(\mkvar{a}) = A$:
  \begin{align*}
    \Theta; \Gamma, \act{I} \Delta &\types \isType{\act{I} A}, \\
    \Theta; \Gamma, \act{I} \Delta &\types \isType{\act{J} A}, \\
    \Theta; \Gamma, \act{I} \Delta &\types \act{I} A \equiv \act{J} A.
  \end{align*}
\end{restatable}

\begin{proof}
  See the proof on \cpageref{proof:inst-eq-bootstrap-ctx}.
\end{proof}

\begin{restatable}{lemma}{restateinsteqbootstrapinst}
  \label{lem:inst-eq-bootstrap-inst}
  In a raw type theory, consider $\Xi = [\symM_1 \of \BB_1, \ldots, \symM_n \of \BB_n]$ such that $\types \isMCtx{\Xi}$,
  and derivable instantiations
  \begin{equation*}
    I = \finmap{\symM_1 \mto e_1, \ldots, \symM_n \mto e_n}
    \qquad\text{and}\qquad
    J = \finmap{\symM_1 \mto f_1, \ldots, \symM_n \mto f_n}
  \end{equation*}
  of~$\Xi$ over~$\Theta; \Gamma$ which are judgementally equal.
  Suppose that $\Theta \types \isVCtx{\Gamma}$.
  Then $\Theta; \Gamma \types \plug{(\upact{I}{i} \BB_i)}{f_i}$ is derivable for $i = 1, \ldots, n$.
\end{restatable}

\begin{proof}
  See the proof on \cpageref{proof:inst-eq-bootstrap-inst}.
\end{proof}

Finally, the lemmas can be assembled into an admissibility theorem about judgementally equal derivable instantiations.

\begin{theorem}[Admissibility of instantiation equality]
  \label{thm:admissibility-equality-instantiation}
  In a raw type theory, consider derivable instantiations $I$ and $J$ of~$\Xi$ over~$\Theta; \Gamma$ which are judgementally equal.
  Suppose that $\types \isMCtx{\Xi}$ and $\Theta \types \isVCtx{\Gamma}$.
  If an object judgement $\Xi; \Gamma \types \JJ$ is derivable then so is $\Theta; \Gamma \types \act{(I \equiv J)} \JJ$.
\end{theorem}

\begin{proof}
  \Cref{lem:eq-inst-admit} applies with empty~$\Delta$ because the additional precondition for~$I$ and~$J$ is guaranteed by \cref{lem:inst-eq-bootstrap-inst}.
\end{proof}

Our last meta-theorem about raw type theories shows that whenever a judgement is derivable, so are its presuppositions, i.e.\ its boundary is well-formed.

\begin{restatable}[Presuppositivity]{theorem}{restatepresuppositivity}
  \label{prop:presuppositivity}%
  If a raw type theory derives $\types \isMCtx{\Theta}$, $\Theta \types \isVCtx{\Gamma}$, and $\Theta; \Gamma \types \plug{\BB}{e}$ then it derives $\Theta; \Gamma \types \BB$.
\end{restatable}

\begin{proof}
  See the proof on \cpageref{proof:presuppositivity}.
\end{proof}

\subsection{Meta-theorems about finitary type theories}
\label{sec:meta-thm-finitary}

Several closure rules contain premises which at first sight seem extraneous, in particular the boundary premises in rule instantiations (\cref{def:raw-rule-instantiation}) and the object premises in a congruence rule (\cref{def:congruence-rule}). While these are needed for raw rules, they ought to be removable for finitary rules, which already have well-formed boundaries. We show that this is indeed the case by providing \emph{economic} versions of the rules, which are admissible in finitary type theories.
We also show that the metavariable rules (\cref{def:metavariable-rule}) have economic versions that are valid in well-formed metavariable contexts.

\begin{proposition}
  \label{prop:tt-specific-eco}[Economic version of \cref{def:raw-rule-instantiation}]
  Let $R$ be the raw rule $\rawRule{\Xi}{\plug{\bdry}{e}}$ with $\Xi = [\symM_1 \of \BB_1, \ldots, \symM_n \of \BB_n]$ such that $\Xi; \emptyCtx \types \bdry$ is derivable, in particular $R$ may be finitary. Then for any instantiation $I = [\symM_1 \mto e_1, \ldots, \symM_n \mto e_n]$ over $\Theta; \Gamma$, the following closure rule is admissible:
  \begin{equation*}
    \inferenceRule{TT-Specific-Eco}
    {\Theta; \Gamma \types \plug {(\upact{I}{i} \BB_i)}{e_i} \quad \text{for $i = 1, \ldots, n$}}
    {\Theta; \Gamma \types \act{I} (\plug{\bdry}{e})}
  \end{equation*}
\end{proposition}

\begin{proof}
  To apply $\act I R$, derive the missing premise $\Theta; \Gamma \types \act I \bdry$ via \cref{prop:instantiation-admissible}.
\end{proof}

\begin{restatable}[Economic version of \cref{def:metavariable-rule}]{proposition}{restatettmetaeco}
  \label{prop:tt-meta-eco}
  If a raw type theory derives $\types \isMCtx \Theta$ with $\Theta(\symM) = (\abstr{x_1 \of A_1} \cdots \abstr{x_m \of A_m} \; \bdry)$, the following closure rules are admissible:
  \begin{mathpar}
    \inferenceRule{TT-Meta-Eco}
    {
      \Theta; \Gamma \types t_j : A_j[\upto{\vec{t}}{j}/\upto{\vec{x}}{j}]
      \quad \text{for $j = 1, \ldots, m$}
    }{
      \Theta; \Gamma \types \plug{(\bdry[\vec{t}/\vec{x}])}{\symM(\vec{t})}
    }

    \inferenceRule{TT-Meta-Congr-Eco}
    { \Theta; \Gamma \types s_j \equiv t_j :
      A_j[\upto{\vec{s}}{j}/\upto{\vec{x}}{j}]
      \quad \text{for $j = 1, \ldots, m$}
    }{
      \Theta; \Gamma \types
      \plug
      {(\bdry[\vec{s}/\vec{x}])}
      {\symM_k(\vec{s}) \equiv \symM_k(\vec{t})}
    }
  \end{mathpar}
\end{restatable}

\begin{proof}
  See the proof on \cpageref{proof:tt-meta-eco}.
\end{proof}

\begin{restatable}{lemma}{restateequalboundaryconvert}
  \label{lem:equal-boundary-convert}%
  In a raw type theory, suppose $\Xi; \Gamma \types \BB$, and consider judgementally equal derivable instantiations $I, J$ of~$\Xi$ over $\Theta; \Gamma$.
  If $\Theta; \Gamma \types \plug{(\act{I} \BB)}{e}$ is derivable then so is $\Theta; \Gamma \types \plug{(\act{J} \BB)}{e}$.
\end{restatable}

\begin{proof}
  See the proof on \cpageref{proof:equal-boundary-convert}.
\end{proof}

\begin{restatable}[Economic version of \cref{def:congruence-rule}]{proposition}{restatecongruenceruleeco}
  \label{def:congruence-rule-eco}
  In a finitary type theory, consider one of its object rules~$R$
  \begin{equation*}
    \rawRule{\symM_1 \of \BB_1, \ldots, \symM_n \of \BB_n}{\plug{\bdry}{e}}.
  \end{equation*}
  Given instantiations of its premises,
  \begin{equation*}
    I = \finmap{\symM_1 \mto f_1, \ldots, \symM_n \mto f_n}
    \quad\text{and}\quad
    J = \finmap{\symM_1 \mto g_1, \ldots, \symM_n \mto g_n},
  \end{equation*}
  over $\Theta; \Gamma$ such that $\types \isMCtx{\Theta}$ and $\Theta \types \isVCtx{\Gamma}$,
  the following closure rule is admissible:
  \begin{equation*}
    \inferenceRule{TT-Congr-Eco}
    { {\begin{aligned}
      &\Theta; \Gamma \types \plug{(\upact{I}{i} \BB_i)}{f_i}  &&\text{for equation boundary $\BB_i$}\\
      &\Theta; \Gamma \types \plug{(\upact{I}{i} \BB_i)}{f_i \equiv g_i} &&\text{for object boundary $\BB_i$} \\
    \end{aligned}}
    }{
      \Theta; \Gamma \types \plug{(\act{I} \bdry)}{\act{I} e \equiv \act{J} e}
    }
  \end{equation*}
\end{restatable}

\begin{proof}
  See the proof on \cpageref{proof:congruence-rule-eco}.
\end{proof}

\subsection{Meta-theorems about standard type theories}
\label{sec:meta-theorems-about}

We next investigate to what extent a derivation of a derivable judgement can be reconstructed from the judgement itself. Firstly, a term expression holds enough information to recover a candidate for its type.

\begin{definition}
  Let $T$ be a standard type theory.
  The \defemph{natural type} $\natty{\Theta; \Gamma}{t}$ of a term expression~$t$ with respect to a context $\Theta; \Gamma$ is defined by:
  \begin{align*}
    \natty{\Theta; \Gamma}{\mkvar{a}} &= \Gamma(a), \\
    \natty{\Theta; \Gamma}{\symM(\vec{t})} &= B[\vec{t}/\vec{x}]
       &&\text{where $\Theta(\symM) = (\abstr{\vec{x} \of \vec{A}} \; \Box : B)$}
    \\
    \natty{\Theta; \Gamma}{\symS(\vec{e})} &= \act{\finmap{\vec{\symM} \mto \vec{e}}} B
       &&\text{where the rule for $\symS$ is $\rawRule{\vec{\symM} \of \vec{\BB}}{\Box : B}$}
  \end{align*}
\end{definition}

We prove an inversion principle that recovers the ``stump'' of a derivation of a derivable object judgement.

\begin{restatable}[Inversion]{theorem}{restateinversion}
  \label{thm:inversion}
  If a standard type theory derives an object judgement then it does so by a derivation which concludes with precisely one of the following rules:
  \begin{enumerate}
  \item the variable rule \rref{TT-Var},
  \item the metavariable rule \rref{TT-Meta},
  \item an instantiation of a symbol rule,
  \item the abstraction rule \rref{TT-Abstr},
  \item the term conversion rule \rref{TT-Conv-Tm} of the form
    \begin{equation*}
      \infer
      {\Theta; \Gamma \types t : \natty{\Theta;\Gamma}{t} \\
       \Theta; \Gamma \types \natty{\Theta; \Gamma}{t} \equiv A}
      {\Theta; \Gamma \types t : A}
    \end{equation*}
    where $\natty{\Theta;\Gamma}{t} \neq A$.
  \end{enumerate}
\end{restatable}

\begin{proof}
  See the proof on \cpageref{proof:inversion}.
\end{proof}

We may keep applying the theorem to all the object premises of a stump to recover the proof-relevant part of the derivation. The remaining proof-irrelevant parts are the equational premises. The inversion theorem yields further desirable meta-theoretic properties of standard type theories.

\begin{corollary}
  \label{cor:natural-type}%
  If a standard type theory derives $\Theta; \Gamma \types t : A$ then it derives $\Theta; \Gamma \types \natty{\Theta; \Gamma}{t} \equiv A$.
\end{corollary}

\begin{proof}
  By inversion, $\natty{\Theta; \Gamma}{t} = A$ or we obtain a derivation of
  $\types \natty{\Theta; \Gamma}{t} \equiv A$.
\end{proof}

\begin{theorem}[Uniqueness of typing]%
  \label{thm:uniqueness-of-typing}%
  For a standard type theory:
  \begin{enumerate}
  \item If $\Theta; \Gamma \types t : A$ and $\Theta; \Gamma \types t : B$ then $\Theta; \Gamma \types A \equiv B$.
  \item If $\types \isMCtx{\Theta}$ and $\Theta \types \isVCtx{\Gamma}$ and $\Theta; \Gamma \types s \equiv t : A$ and $\Theta; \Gamma \types s \equiv t : B$ then $\Theta; \Gamma \types A \equiv B$.
  \end{enumerate}
\end{theorem}

\begin{proof}
  The first statement holds because $A$ and $B$ are both judgmentally equal to the natural type of~$t$ by \cref{cor:natural-type}. The second statement reduces to the first one because the presuppositions $\Theta; \Gamma \types t : A$ and $\Theta; \Gamma \types t : B$ are derivable by \cref{prop:presuppositivity}.
\end{proof}

\section{Context-free finitary type theories}
\label{sec:context-free-tt}

In the forward-chaining style, characteristic of LCF-style theorem provers, which Andromeda~2 is designed to be, a judgement is not construed by reducing a goal to subgoals, but as a value of an abstract datatype, and built by applying an abstract datatype constructor to previously derived judgements. What should such a constructor do when its arguments have mismatching variable contexts? It can try to combine them if possible, or require that the user make sure ahead of time that they match. As was already noted by Geuvers et al.\ in the context of pure type systems \cite{geuvers10}, it is best to sidestep the whole issue by dispensing with contexts altogether. In the present section we give a second account of finitary type theories, this time without context and with free variables explicitly annotated with their types. These are actually implemented in the Andromeda~2 trusted nucleus.

Our formulation of \defemph{context-free finitary type theories} is akin to the $\Gamma_\infty$ formalism for pure type
systems~\cite{geuvers10}.
We would like to replace judgements of the form ``$\Theta; \Gamma \types \JJ$'' with just ``$\JJ$''.
In traditional accounts of logic, as well as in $\Gamma_\infty$, this is accomplished by explicit type annotations of free variables:
rather than having $\mkvar{a} : A$ in the variable context, each occurrence of~$\mkvar{a}$ is annotated with its type as~$\avar{a}{A}$.

We use the same idea, although we have to overcome several technical complications, of which the most challenging one is the lack of strengthening, which is the principle stating that
if $\Theta; \Gamma, \mkvar{a} \of A, \Delta \types \JJ$ is derivable and~$\mkvar{a}$ does not appear in~$\Delta$ and~$\JJ$, then~$\Theta; \Gamma, \Delta \types \JJ$ is derivable.
An example of a rule that breaks strengthening for finitary type theories is equality reflection from \cref{ex:tt-equality-reflection},
\begin{equation*}
  \infer{
    \types \isType{\sym{A}} \\
    \types \sym{s} : \sym{A} \\
    \types \sym{t} : \sym{A} \\
    \types \sym{p} : \sym{Id}(\sym{A}, \sym{s}, \sym{t})
  }{
    \types \sym{s} \equiv \sym{t} : \sym{A}
  }
\end{equation*}
Because the conclusion elides the metavariable~$\sym{p}$, it will not record the fact that a variable may have been used in the derivation of the fourth premise. Consequently, we cannot tell what variables ought to occur in the context just by looking at the judgement thesis.
As it turns out, variables elided by derivations of equations are the only culprit, and the situation can be rectified by modifying equality judgements so that they carry additional information about usage of variables.
In the present section we show how this is accomplished by revisiting the definition of type theories from \cref{sec:finitary-type-theories} and making the appropriate modifications.

\subsection{Raw syntax of context-free type theories}
\label{sec:context-free-raw-syntax}

Apart from removing the variable context and annotating free variables with type expressions, we make three further modifications to the raw syntax: we remove metavariable contexts, and instead annotate metavariables with boundaries; we introduce assumption sets that keep track of variables used in equality derivations; and we introduce explicit conversions.

\subsubsection{Free and bound variables}
\label{sec:vari-subst}

The \defemph{bound variables} $x, y, z, \ldots$ are as before, for example they could be de Bruijn indices, whereas the free variables are annotated explicitly with type expressions.
More precisely, given a set of names $\mkvar{a}, \sym{b}, \sym{c}, \ldots$ a \defemph{free variable} takes the form $\avar{a}{A}$ where $A$ is a type expression, cf.\ \cref{sec:context-free-raw-expressions}. Two such variables $\avar{a}{A}$ and $\sym{b}^B$ are considered syntactically equal when the symbols $\mkvar{a}$ and $\sym{b}$ are the same \emph{and} the type expressions~$A$ and $B$ are syntactically equal.
Thus it is quite possible to have variables $\avar{a}{A}$ and $\avar{a}{B}$ which are different even though $A$ and $B$ are judgmentally equal. In an implementation it may be a good idea to prevent such extravaganza by generating fresh symbols so that each one receives precisely one annotation.

Similarly, metavariables are tagged with boundaries, where again $\sym{M}^\BB$ and $\sym{N}^{\BB'}$ are considered equal when both the symbols $\sym{M}$ and $\sym{N}$ are equal and the boundaries $\BB$ and $\BB'$ are syntactically identical.

\subsubsection{Arities and signatures}
\label{sec:context-free-arities-signatures}

Arities of symbols and metavariables are as in \cref{sec:signatures}, and so are signatures.

\subsubsection{Raw expressions}
\label{sec:context-free-raw-expressions}

The raw expressions of a context-free type theory are built over a signature~$\Sigma$, as summarized in the top part of \cref{fig:syntax-context-free-type-theories}.

\begin{figure}[htbp]
  \centering
  \small
  \begin{ruleframe}[innerleftmargin=6pt,innerrightmargin=6pt]
  \begin{align*}
  \text{Type expression}\ A, B
  \bnfis& \symS(e_1, \ldots, e_n)       &&\text{type symbol application}\\
  \bnfor& \symM^\BB(t_1, \ldots, t_n)   &&\text{type metavariable application}
  \\
  \text{Term expression}\ s, t
  \bnfis& \avar{a}{A}                   &&\text{free variable}\\
  \bnfor& x                             &&\text{bound variable}\\
  \bnfor& \symS(e_1, \ldots, e_n)       &&\text{term symbol application}\\
  \bnfor& \symM^\BB(t_1, \ldots, t_n)   &&\text{term metavariable application} \\
  \bnfor& \convert{t}{\alpha}           &&\text{conversion}
  \\
  \text{Assumption set}\ \alpha, \beta
  \bnfis& \mathrlap{\asset{\ldots, a_i^A, \ldots, x_j, \ldots, \symM^\BB_k, \ldots}}
  \\
  \text{Argument}\ e
  \bnfis& A           &&\text{type argument} \\
  \bnfor& t           &&\text{term argument} \\
  \bnfor& \alpha      &&\text{assumption set} \\
  \bnfor& \abstr x e  &&\text{abstracted arg. ($x$ bound in $e$)}
  \\[1ex]
  \text{Judgement}\ \J
  \bnfis& \isType{A}                && \text{$A$ is a type} \\
  \bnfor& t : A                     && \text{$t$ has type $T$} \\
  \bnfor& A \equiv B \by \alpha     && \text{$A$ and $B$ are equal types} \\
  \bnfor& s \equiv t : A \by \alpha && \text{$s$ and $t$ are equal terms at $A$}
  \\
  \text{Abstracted judgement}\ \JJ
  \bnfis& \J                   &&\text{non-abstracted judgement} \\
  \bnfor& \abstr{x \of A}\; \JJ  &&\text{abstracted jdgt. ($x$ bound in $\JJ$)}
  \\[1ex]
  \text{Boundary}\ \bdry
  \bnfis& \isType{\Box}            &&\text{a type}\\
  \bnfor& \Box :  A                &&\text{a term of type $A$}\\
  \bnfor& A \equiv B \by \Box      &&\text{type equation boundary}\\
  \bnfor& s \equiv t : B \by \Box  &&\text{term equation boundary}
  \\
  \text{Abstracted boundary}\ \BB
  \bnfis& \bdry                   &&\text{non-abstracted boundary} \\
  \bnfor& \abstr{x \of A}\; \BB  &&\text{abstracted bdry. ($x$ bound in $\BB$)}
  \end{align*}
  \end{ruleframe}
\caption{The raw syntax of context-free finitary type theories}
\label{fig:syntax-context-free-type-theories}
\end{figure}

A type expression is either a type symbol $\symS$ applied to arguments $e_1, \ldots, e_n$, or a metavariable $\symM^\BB$ applied to term expressions $t_1, \ldots, t_n$.

The syntax of term expressions differs from the one in \cref{fig:syntax-general-type-theories} in two ways.
First, we annotate free variables with type expressions and metavariables with boundaries, as was already discussed, where it should be noted that in an annotation $A$ of $\avar{a}{A}$ or $\BB$ of $\symM^\BB$ there may be further free and metavariables, which are also annotated, and so on. We require that a boundary annotation $\BB$ be closed, and that a type annotation~$A$ contain no ``exposed'' bound variables, i.e.\ $A$ is syntactically valid on its own, without having to appear under an abstraction.
Second, we introduce the \defemph{conversion terms} ``$\convert{t}{\alpha}$'', which will serve to record the variables used to derive the equality along which~$t$ has been converted.

The expressions of syntactic classes $\EqTy$ and $\EqTm$ are the \defemph{assumption sets}, which are finite sets of free and bound variables, and metavariables. As we are already using the curly braces for abstraction, we write finite set comprehension as $\asset{\cdots}$.
Assumption sets record the variables and metavariables that are used in a derivation of an equality judgement but may not appear in the boundary of the conclusion.

We ought to be a bit careful about occurrences of variables, since the free variables may occur in variable annotations, and the metavariables in boundary annotations. \Cref{fig:cf-variable-occurrences}, the context-free analogue of \cref{fig:variable-occurrences}, shows the definitions of free, bound and metavariable occurrences.
\begin{figure}[htp]
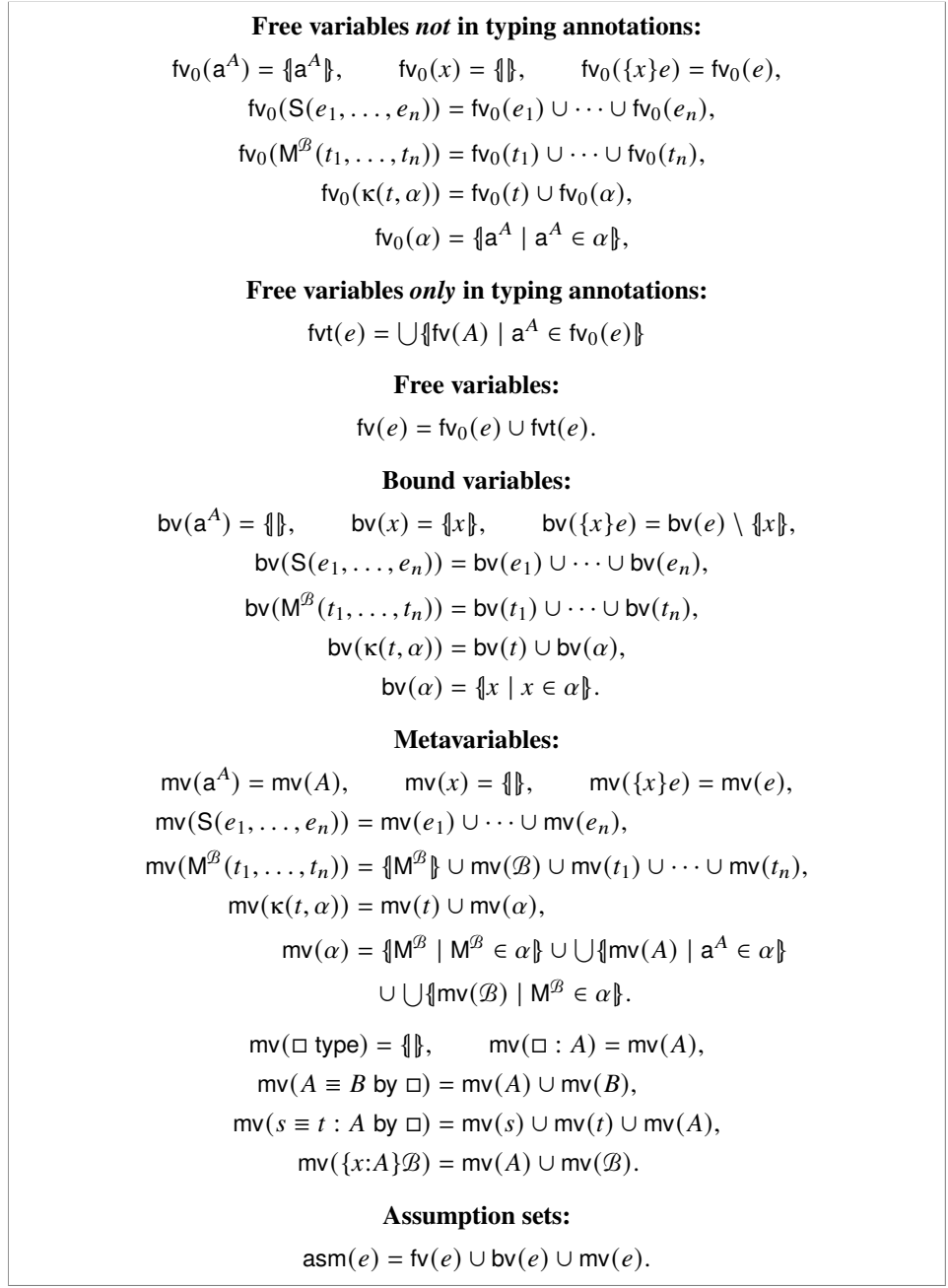

  \begin{ruleframe}
  \centering
  \small
  \textbf{Free variables \emph{not} in typing annotations:}
  \begin{gather*}
    \fvz{\avar{a}{A}} = \asset{\avar{a}{A}}, \qquad
    \fvz{x} = \asset{}, \qquad
    \fvz{\abstr{x} e} = \fvz{e}, \\
    \begin{aligned}
      \fvz{\symS(e_1, \ldots, e_n)} &= \fvz{e_1} \cup \cdots \cup \fvz{e_n}, \\
      \fvz{\symM^\BB(t_1, \ldots, t_n)} &= \fvz{t_1} \cup \cdots \cup \fvz{t_n}, \\
      \fvz{\convert{t}{\alpha}} &= \fvz{t} \cup \fvz{\alpha}, \\
      \fvz{\alpha} &= \asset{\avar{a}{A} \such \avar{a}{A} \in \alpha},
    \end{aligned}
  \end{gather*}

  \medskip

  \textbf{Free variables \emph{only} in typing annotations:}
  \begin{equation*}
    \textstyle
    \fvt{e} = \bigcup \asset{ \fv{A} \such \avar{a}{A} \in \fvz{e} }
  \end{equation*}

  \medskip

  \textbf{Free variables:}
  \begin{equation*}
    \textstyle
    \fv{e} = \fvz{e} \cup \fvt e.
  \end{equation*}

  \medskip

  \textbf{Bound variables:}
  \begin{gather*}
    \bv{\avar{a}{A}} = \asset{}, \qquad
    \bv{x} = \asset{x}, \qquad
    \bv{\abstr{x} e} = \bv{e} \setminus \asset{x}, \\
    \begin{aligned}
      \bv{\symS(e_1, \ldots, e_n)} &= \bv{e_1} \cup \cdots \cup \bv{e_n}, \\
      \bv{\symM^\BB(t_1, \ldots, t_n)} &= \bv{t_1} \cup \cdots \cup \bv{t_n}, \\
      \bv{\convert{t}{\alpha}} &= \bv{t} \cup \bv{\alpha}, \\
      \bv{\alpha} &= \asset{ x \such x \in \alpha}.
    \end{aligned}
  \end{gather*}

  \medskip

  \textbf{Metavariables:}
  \begin{gather*}
    \mv{\avar{a}{A}} = \mv{A}, \qquad
    \mv{x} = \asset{}, \qquad
    \mv{\abstr{x} e} = \mv{e}, \\
    \begin{aligned}
      \mv{\symS(e_1, \ldots, e_n)} &= \mv{e_1} \cup \cdots \cup \mv{e_n}, \\
      \mv{\symM^\BB(t_1, \ldots, t_n)} &= \asset{\symM^\BB} \cup \mv{\BB} \cup \mv{t_1} \cup \cdots \cup \mv{t_n}, \\
      \mv{\convert{t}{\alpha}} &= \mv{t} \cup \mv{\alpha}, \\
      \mv{\alpha} &= \textstyle
                      \asset{ \symM^\BB \such \symM^\BB \in \alpha}
                      \cup \bigcup \asset{\mv{A} \such {\avar a A} \in \alpha}\\
       &\phantom{=}\; \textstyle
                      {} \cup \bigcup \asset{\mv{\BB} \such {\symM^\BB} \in \alpha}.
    \end{aligned}
    \\[1ex]
    \mv{\isType{\Box}} = \asset{}, \qquad
    \mv{\Box : A} = \mv{A}, \\
    \begin{aligned}
      \mv{A \equiv B \by \Box} &= \mv{A} \cup \mv{B}, \\
      \mv{s \equiv t : A \by \Box} &= \mv{s} \cup \mv{t} \cup \mv{A}, \\
      \mv{\abstr{x \of A} \BB} &= \mv{A} \cup \mv{\BB}.
    \end{aligned}
  \end{gather*}

  \medskip

  \textbf{Assumption sets:}
  \begin{equation*}
    \asm{e} = \fv{e} \cup \bv{e} \cup \mv{e}.
  \end{equation*}
  \end{ruleframe}
  \caption{Context-free variable occurrences and assumption sets}
  \label{fig:cf-variable-occurrences}
\end{figure}
Note the difference between $\fvz{e}$, which collects \emph{only} the free variable occurrences not appearing in a type annotation, and $\fv{e}$ which collects them all.
Bound variables need not be collected from annotations, as they cannot appear there.

The collection of all free, bound and metavariables occurring in an expression is its
\defemph{assumption set $\asm{e}$}. Sometimes we write $\asm{e_1, \ldots, e_n}$ for the union $\bigcup_i \asm{e_i}$.

\subsubsection{Substitution and syntactic equality}
\label{sec:context-free-subst-equal}

We must review substitution and syntactic equality, because they are affected by annotations, assumption sets, and conversion terms.

There are two kinds of substitutions. An \defemph{abstraction} $e[x/\avar{a}{A}]$ transforms the free variable $\avar{a}{A}$ in~$e$ to a bound variable~$x$, whereas a \defemph{substitution} $e[s/x]$ replaces the bound variable $x$ with the term $s$. These are shown in \cref{fig:abstraction-substitution}.
Note that an abstraction $e[x/\avar{a}{A}]$ is only valid when $\avar{a}{A}$ does not appear in any type annotation in~$e$, $\avar{a}{A} \notin \fvt{e}$, because type annotations cannot refer to bound variables. Consequently, abstraction of several variables must be carried out in the reverse order of their dependencies.
We abbreviate a series of abstractions $((e[x_1/\avars{a}{1}{A_1}]) \cdots) [x_n/\avars{a}{n}{A_n}]$ as $e[x_1/\avars{a}{1}{A_1}, \ldots, x_n/\avars{a}{n}{A_n}]$ or just $e[\vec{x}/\vec{\sym{a}}_n^{A_n}]$.
Similarly, a series of substitutions $((e[s_1/x_1]) \cdots)[s_n/x_n]$ is written $e[s_1/x_1, \ldots, s_n/x_n]$ or just $e[\vec{s}/\vec{x}]$.

\begin{figure}[htp]
  \begin{ruleframe}
  \centering
  \small
  \textbf{Abstraction:}
  \begin{gather*}
    \avar{a}{A} [x/\avar{a}{A}] = x, \hspace{5em}
    y [x/\avar{a}{A}] = y, \\
    \begin{aligned}
      b^B [x/\avar{a}{A}] &= b^B \quad\text{if $\avar{a}{A} \neq b^B$ and $\avar{a}{A} \not\in \fvt{B}$} \\
      \symS(e_1, \ldots, e_n) [x/\avar{a}{A}] &= \symS(e_1[x/\avar{a}{A}], \ldots, e_n[x/\avar{a}{A}]), \\
      \symM^\BB(t_1, \ldots, t_n) [x/\avar{a}{A}] &= \symM^\BB(t_1[x/\avar{a}{A}], \ldots, t_n[x/\avar{a}{A}]), \\
      \convert{t}{\alpha} [x/\avar{a}{A}] &= \convert{t[x/\avar{a}{A}]}{A[x/\avar{a}{A}]}{\alpha[x/\avar{a}{A}]}, \\
      (\abstr{y} e)[x/\avar{a}{A}] &= \abstr{y} (e[x/\avar{a}{A}])
      \quad\text{if $x \neq y$},\\
      \alpha [x/\avar{a}{A}] &= \alpha
      \quad\text{if $\avar{a}{A} \not\in \alpha$}, \\
      \alpha [x/\avar{a}{A}] &= (\alpha \setminus \asset{\avar{a}{A}}) \cup \asset{x}
      \quad\text{if $\avar{a}{A} \in \alpha$ and $\avar{a}{A} \not\in \fvt{\alpha}$},
    \end{aligned}
  \end{gather*}

  \medskip

  \textbf{Substitution:}
  \begin{gather*}
    \avar{a}{A} [s/x] = \avar{a}{A}, \qquad
    x [s/x] = s, \qquad
    y [s/x] = y \quad\text{if $x \neq y$},\\
    \begin{aligned}
      \symS(e_1, \ldots, e_n) [s/x] &= \symS(e_1[s/x], \ldots, e_n[s/x]) \\
      \symM^\BB(t_1, \ldots, t_n) [s/x] &= \symM^\BB(t_1[s/x], \ldots, t_n[s/x]) \\
      \convert{t}{\alpha} [s/x] &= \convert{t[s/x]}{\alpha[s/x]} \\
      (\abstr{y} e)[s/x] &= \abstr{y} (e[s/x]) \\
      \alpha [s/x] &= \alpha
      \quad\text{if $x \not\in \alpha$} \\
      \alpha [s/x] &= (\alpha \setminus \asset{x}) \cup \asm{s}
      \quad\text{if $x \in \alpha$.}
    \end{aligned}
  \end{gather*}
  \end{ruleframe}
  \caption{Abstraction and substitution}
  \label{fig:abstraction-substitution}
\end{figure}

Syntactic equality is treated in a standard way, we only have to keep in mind the fact that symbols are considered syntactically equal if the bare symbols are equal \emph{and} their annotations are equal.
More interestingly, since conversion terms and assumption sets carry proof-irrelevant information, they should be ignored in certain situations. For this purpose, define the \defemph{erasure~$\erase{e}$} to be the raw expression~$e$ with the assumption sets and conversion terms removed:
\begin{gather*}
  \erase{\avar{a}{A}} = \avar{a}{A}, \quad
  \erase{x} = x, \quad
  \erase{\convert{t}{\alpha}} = \erase{t}, \quad
  \erase{\alpha} = \dummy, \quad
  \erase{\abstr{x} e} = \abstr{x} \erase{e}, \\
  \erase{\symS(e_1, \ldots, e_n)} = \symS(\erase{e_1}, \ldots, \erase{e_n}), \quad
  \erase{\symM^\BB(t_1, \ldots, t_n)} = \symM^\BB(\erase{t_1}, \ldots, \erase{t_n}).
\end{gather*}
The mapping $e \mapsto \erase{e}$ takes the context-free raw syntax of \cref{fig:syntax-context-free-type-theories} to the type-theoretic raw syntax of \cref{fig:syntax-general-type-theories} where the variables~$\avar{a}{A}$ and the metavariables~$\symM^\BB$ are construed as atomic symbols, i.e.\ their annotations are part of the symbol name.

\subsubsection{Judgements and boundaries}
\label{sec:context-free-judg-bound}

The lower part of \cref{fig:syntax-context-free-type-theories} summarizes the syntax of context-free judgements and boundaries. Apart from not having contexts, type judgements ``$\isType{A}$'' and term judgements ``$t : A$'' are as before.
Equality judgements are modified to carry assumption sets: a type equality takes the form ``$A \equiv B \by \alpha$'' and a term equality ``$s \equiv t : A \by \alpha$''.

Boundaries do not change, except of course that they have no contexts. The head of a boundary is filled like before, except that assumption sets are used instead of dummy values, see \cref{fig:cf-boundary-filling}.

\begin{figure}[htbp]
  \centering
  \begin{ruleframe}
    \small
    \centering

    \textbf{Filling the placeholder with a head:}
    \begin{align*}
      \plug{(\isType \Box)}{A} &= (\isType A) \\
      \plug{(\Box : A)}{t} &= (t : A) \\
      \plug{(A \equiv B \by \Box)}{e} &= (A \equiv B \by \asm{e}) \\
      \plug{(s \equiv t : A \by \Box)}{e} &= (s \equiv t : A \by \asm{e}) \\
      \plug{(\abstr{x \of A}\; \BB)}{\abstr{x} e} &= (\abstr{x \of A}\; \plug{\BB}{e}).
    \end{align*}

    \medskip

    \textbf{Filling the placeholder with an equality:}
    \begin{align*}
      \plug{(\isType{\Box})}{A_1 \equiv A_2 \by \alpha} &= (A_1 \equiv A_2 \by \alpha), \\
      \plug{(\Box : A)}{t_1 \equiv t_2 \by \alpha} &= (t_1 \equiv t_2 : A \by \alpha),\\
      \plug{(\abstr{x \of A}\; \BB)}{e_1 \equiv e_2 \by \alpha} &=
          (\abstr{x \of A}\; \plug{\BB}{e_1 \equiv e_2 \by \alpha}),
    \end{align*}
  \end{ruleframe}
  \caption{Context-free filling the head of a boundary}
  \label{fig:cf-boundary-filling}
\end{figure}

Free-variable occurrences in judgements and boundaries are defined as follows:
\begin{gather*}
  \fvz{\isType A} = \fvz{A}, \qquad
  \fvz{t : A} = \fvz{t} \cup \fvz{A}, \\
  \begin{aligned}
  \fvz{A \equiv B \by \alpha} &= \fvz{A} \cup \fvz{B} \cup \fvz{\alpha}, \\
  \fvz{s \equiv t : A \by \alpha} &= \fvz{s} \cup \fvz{t} \cup \fvz{A} \cup \fvz{\alpha},\\
  \fvz{\abstr{x : A} \JJ} &= \fvz{A} \cup \fvz{\JJ}, \\[1ex]
  \fv{\JJ} &= \fvz{\JJ} \cup \fvt \JJ.
\end{aligned}
\end{gather*}
We trust the reader can emulate the above definition to define the set $\mv{\JJ}$ of metavariable occurrences in a judgement~$\JJ$, as well as occurrences of free and metavariables in boundaries.

\subsubsection{Metavariable instantiations}
\label{sec:context-free-meta-vari-inst}

Next, let us rethink how metavariable instantiations work in the presence of the newly introduced syntactic constructs.
As before an \defemph{instantiation} is a sequence, representing a map,
\begin{equation*}
  I = \finmap{\symM^{\BB_1}_1 \mto e_1, \ldots, \symM^{\BB_n}_n \mto e_n}
\end{equation*}
such that $\mv{\BB_i} \subseteq \asset{\symM_1^{\BB_1}, \ldots, \symM_{i-1}^{\BB_{i-1}}}$ and $\arity{\BB_i} = \arity{e_i}$, for each $i = 1, \ldots, n$.
As in \cref{sec:meta-vari-inst}, $I$ acts on an expression $e$, provided that $\mv{e} \subseteq \vert{}I\vert$, by replacing metavariables with the corresponding expressions, see \cref{fig:metavariable-instantiation}.
\begin{figure}[htbp]
  \centering
  \begin{ruleframe}
  \begin{gather*}
    \act{I} x = x,\qquad
    \act{I} (\convert{t}{\alpha}) = \convert{\act{I} t}{\act{I} \alpha}, \\
    \act{I} \avar{a}{A} = \avar{a}{\act{I} A},\qquad
    \act{I} (\abstr{x} e) = \abstr{x} (\act{I} e),
    \\
    \begin{aligned}
      \act{I} (\symS(e_1, \ldots, e_k))
      &= \symS(\act{I} e_1, \ldots, \act{I} e_k),\\
      \act{I} (\symM_i^{\BB_i}(t_1, \ldots, t_{m_i}))
      &= I(\symM_i)[(\act{I} t_1)/x_1, \ldots, (\act{I} t_{m_i})/x_{m_i}], \\
      \act{I} \alpha &= \textstyle\bigcup \mkset{\asm{I(\symM_i)} \such \symM^{\BB_i}_i \in \alpha} \\
      &\quad \cup \asset{x \such x \in \alpha}
        \cup \asset{\avar{a}{\act{I} A} \such \avar{a}{A} \in \alpha}.
    \end{aligned}
    \\[1ex]
    \begin{aligned}[t]
      \act{I} (\isType A) &= (\isType{\act{I} A}), \\
      \act{I} (t : A) &= (\act{I} t : \act{I} A),\\
      \act{I} (A \equiv B \by \alpha) &= (\act{I} A \equiv \act{I} B \by \act{I} \alpha), \\
      \act{I} (\abstr{x \of A} \JJ) &= \abstr{x \of \act{I} A} \act{I} \JJ, \\
      \act{I} \Box &= \Box.
    \end{aligned}
  \end{gather*}
  \end{ruleframe}
  \caption{The action of a metavariable instantiation}
  \label{fig:metavariable-instantiation}
\end{figure}
Note that the action of~$I$ on a free variable changes the identity of the
variable by acting on its typing annotation.

\subsection{Context-free rules and type theories}
\label{sec:context-free-rules-type}

In this section we adapt the notions of raw and finitary rules and type theories to the
context-free setting. We shall be rather telegraphic about it, as the changes are
straightforward and require little discussion.

\begin{definition}
  \label{def:context-free-raw-rule}%
  A \defemph{context-free raw rule} $R$ over a signature $\Sigma$ has the form
  \begin{equation*}
    \rawRule{\symM_1^{\BB_1}, \ldots, \symM_n^{\BB_n}}{\J}
  \end{equation*}
  where the \defemph{premises} $\BB_i$ and the \defemph{conclusion} $\J$ are closed and syntactically valid over~$\Sigma$, $\mv{\BB_i} \subseteq \asset{\symM_1^{\BB_1}, \ldots, \symM_{i-1}^{\BB_{i-1}}}$ for every $i = 1, \ldots, n$, and $\mv{\J} = \asset{\symM_1^{\BB_1}, \ldots, \symM_n^{\BB_n}}$.
  We say that~$R$ is an \defemph{object rule} when~$\J$ is a type or a term judgement,
  and an \defemph{equality rule} when $\J$ is an equality judgement.
\end{definition}

The condition $\mv{\J} = \asset{\symM_1^{\BB_1}, \ldots, \symM_n^{\BB_n}}$ ensures that the conclusion of an instantiation of a raw rule records all uses of variables. We shall need it in the proof of \cref{thm:cf-to-tt-bdry-jdg}.

\begin{example}
    The context-free version of equality reflection from \cref{ex:tt-equality-reflection} is
    {
    \newcommand{\AAA}{\sym{A}^{\isType{\Box{}}}}
    \newcommand{\s}{\sym{s}^{\Box{}\, :\, \AAA{}}}
    \newcommand{\symt}{\sym{t}^{\Box{}\, :\, \AAA{}}}
    \newcommand{\ppp}{\sym{p}^{\sym{Id}(\AAA, \s, \symt)}}
    \begin{multline*}
      \rawRule{\AAA,\quad \s,\quad \symt,\quad \ppp\\}
      {\s\ \equiv\ \symt\ :\ \AAA\ \by\ \asset{\ppp}}
    \end{multline*}
    }
    which is quite unreadable. We indulge in eliding annotations on any variable that is already typed by a premise or a hypothesis, and write just
    \begin{equation*}
      \inferenceRule{CF-Eq-Reflect}{
        \types \isType{\sym{A}} \\
        \types \sym{s} : \sym{A} \\
        \types \sym{t} : \sym{A} \\
        \types \sym{p} : \sym{Id}(\sym{A}, \sym{s}, \sym{t})
      }{
        \types \sym{s} \equiv \sym{t} : \sym{A} \by \asset{\sym{p}}
      }
    \end{equation*}
    As there are no contexts, we could remove $\types$ too, but we leave it there out of habit. Note how the assumption
    set in the conclusion must record dependence on~$\sym{p}$, or else it would violate the assumption set condition of \cref{def:context-free-raw-rule}.
\end{example}

When formulating equality closure rules we face a choice of assumptions sets. For example, what should~$\gamma$ be in the transitivity rule
\begin{equation*}
  \infer
  {
    \types A \equiv B \by \alpha \\
    \types B \equiv C \by \beta
  }{
    \types A \equiv C \by \gamma}
  \; ?
\end{equation*}
Its intended purpose is to record any assumptions used in the premises but not already recorded by~$A$ and~$C$, which suggests the requirement
\begin{equation*}
  \asm{A} \cup \asm{B} \cup \asm{C} \cup \alpha \cup \beta \subseteq
  \asm{A} \cup \asm{C} \cup \gamma.
\end{equation*}
If we replace $\subseteq$ with $=$ we also avoid any extraneous asumptions, which leads to the following definition.

\begin{definition}
  In a closure rule $([p_1, \ldots, p_n], \plug{\bdry}{\alpha})$ whose conclusion is an equality judgement,
  $\alpha$ is \defemph{suitable} when $\asm{p_1, \ldots, p_n} = \asm{\plug{\bdry}{\alpha}}$.
\end{definition}

\noindent
Provided that $\asm{\bdry} \subseteq \asm{p_1, \ldots, p_n}$, we may always take the minimal suitable assumption set $\alpha = \asm{p_1, \ldots, p_n} \setminus \asm{\bdry}$. We do not insist on minimality, even though an implementation might make an effort to keep the assumption sets small, because minimality is not preserved by instantiations, whereas suitability is.
We shall indicate the suitability requirement in an equality closure rule by stating it as the side condition ``$\suitable{\alpha}$''.

\begin{definition}
  \label{def:context-free-rule-boundary}%
  A \defemph{context-free raw rule-boundary} over a signature $\Sigma$ has the form
  \begin{equation*}
    \rawRule{\symM_1^{\BB_1}, \ldots, \symM_n^{\BB_n}}{\bdry}
  \end{equation*}
  where the boundaries $\BB_i$ and $\bdry$ are closed and syntactically valid over~$\Sigma$, $\mv{\BB_i} \subseteq \asset{\symM_1^{\BB_1}, \ldots, \symM_{i-1}^{\BB_{i-1}}}$ for every $i = 1, \ldots, n$, and $\mv{\bdry} \subseteq \asset{\symM_1^{\BB_1}, \ldots, \symM_n^{\BB_n}}$.
  We say that~$R$ is an \defemph{object rule-boundary} when~$\bdry$ is an object boundary,
  and an \defemph{equality rule-boundary} when $\bdry$ is an equality boundary.
\end{definition}

\begin{definition}
  \label{def:context-free-symbol-rule}%
  Given an object rule-boundary
  \begin{equation*}
    \rawRule{\symM_1^{\BB_1}, \ldots, \symM_n^{\BB_n}}{\bdry}
  \end{equation*}
  over~$\Sigma$, the \defemph{associated symbol arity} is $(c, [\arity{\BB_1}, \ldots, \arity{\BB_n}])$, where $c \in \mkset{\Ty, \Tm}$ is the syntactic class of~$\bdry$.
  The \defemph{associated symbol rule} for $\symS \not\in \vert\Sigma\vert$ is the raw rule
  \begin{equation*}
    \rawRule
      {\symM_1^{\BB_1}, \ldots, \symM_n^{\BB_n}}
      {\plug{\bdry}{\symS(\genapp{\symM}_1^{\BB_1}, \ldots, \genapp{\symM}_n^{\BB_n})}}
  \end{equation*}
  over the extended signature $\extendMany{\Sigma}{\symS \mto (c, [\arity{\BB_1}, \ldots, \arity{\BB_n}])}$, where $\genapp{\symM}^\BB$ is the \defemph{generic application} of the metavariable~$\symM^\BB$, defined as:
  \begin{enumerate}
  \item
    $\genapp{\symM}^\BB = \abstr{x_1} \cdots \abstr{x_k} \symM^\BB(x_1, \ldots, x_k)$ if $\arity{\BB} = (c, k)$ and $c \in \mkset{\Ty, \Tm}$,
  \item
    $\genapp{\symM}^\BB = \abstr{x_1} \cdots \abstr{x_k} \asset{\symM^\BB, x_1, \ldots, x_k}$ if $\arity{\BB} = (c, k)$ and $c \in \mkset{\EqTy, \EqTm}$.
  \end{enumerate}
\end{definition}

\begin{definition}
  \label{def:context-free-equality-rule}%
  Given an equality rule-boundary
  \begin{equation*}
    \rawRule
      {\symM_1^{\BB_1}, \ldots, \symM_n^{\BB_n}}
      {\bdry},
  \end{equation*}
  the \defemph{associated equality rule} is
  \begin{equation*}
    \rawRule
      {\symM_1^{\BB_1}, \ldots, \symM_n^{\BB_n}}
      {\plug{\bdry}{\asset{\symM_1^{\BB_1}, \ldots, \symM_n^{\BB_n}} \setminus \asm{\bdry}}}.
  \end{equation*}
\end{definition}

\begin{definition}
  \label{def:context-free-rule-instantiation}
  An \defemph{instantiation} of a raw rule
  \begin{equation*}
    R = (\rawRule{\symM_1^{\BB_1}, \ldots, \symM_n^{\BB_n}}{\plug{\bdry}{e}})
  \end{equation*}
  over a signature~$\Sigma$ is an instantiation $I = \finmap{\symM_1^{\BB_1} \mto e_1, \ldots, \symM_n^{\BB_n} \mto e_n}$ of the metavariables of~$R$.
  The \defemph{closure rule $\act{I} R$} associated with $I$ and $R$ is $([p_1, \ldots, p_n, q], r)$ where $p_i$ is
  $
    \types
      \plug
         {\upact{I}{i} \BB_i)}
         {e_i}
  $,
  $q$ is $\types \act{I} \bdry$,
  and $r$ is $\types \act{I} (\plug{\bdry}{e})$.
\end{definition}

A minor complication arises when congruence rules (\Cref{def:congruence-rule}) are adapted to the context-free setting, because conversions must be inserted. Consider the congruence rule~\eqref{eq:pi-congruence-rule} for~$\Uppi$ from \cref{ex:pi-congruence-rule}. The premise $A_1 \equiv A_2$ ensures that the premise
$
\abstr{x \of A_1} \; B_1(x) \equiv B_2(x)
$
is well-formed by conversion of~$x$ on the right-hand side from~$A_1$ to~$A_2$, thus in the context-free version of the rule we should allow for the possibility of an explicit conversion. However, we should not enforce an unnecessary conversion in case $A_1 = A_2$, nor should we require particular conversions, as there may be many ways to convert a term. We therefore formulate flexible congruence rules as follows: if an occurrence of a term~$t$ possibly requires conversion, we allow in its place a term~$t'$ such that $\erase{t} = \erase{t'}$.

\begin{definition}
  \label{def:context-free-congruence-rule}
  The \defemph{context-free congruence rules} associated with a context-free raw type rule
  \begin{equation*}
    \rawRule{\symM_1^{\BB_1}, \ldots, \symM_n^{\BB_n}}{\isType{A}}
  \end{equation*}
  are closure rules, where
  \begin{equation*}
    I = \finmap{\symM_1^{\BB_1} \mto f_1, \ldots, \symM_n^{\BB_n} \mto f_n},
    \quad\text{and}\quad
    J = \finmap{\symM_1^{\BB_1} \mto g_1, \ldots, \symM_n^{\BB_n} \mto g_n},
  \end{equation*}
  of the following form:
  \begin{equation*}
    \infer{
      {\begin{aligned}
      &\types \plug{(\upact{I}{i} \BB_i)}{f_i}  &&\text{for $i = 1, \ldots, n$}\\
      &\types \plug{(\upact{J}{i} \BB_i)}{g_i}  &&\text{for $i = 1, \ldots, n$}\\
      &\erase{g'_i} = \erase{g_i} &&\text{for object boundary $\BB_i$} \\
      &\types \plug{(\upact{I}{i} \BB_i)}{f_i \equiv g'_i \by \alpha_i} &&\text{for object boundary $\BB_i$} \\
      & &&\suitable{\beta}
    \end{aligned}}
    }{
      \types \act{I} A \equiv \act{J} A \by \beta
    }
  \end{equation*}
  Similarly, the congruence rule associated with a raw term rule
  \begin{equation*}
    \rawRule{\symM_1^{\BB_1}, \ldots, \symM_n^{\BB_n}}{t : A}
  \end{equation*}
  are closure rules of the form
  \begin{equation*}
    \infer{
      {\begin{aligned}
      &\types \plug{(\upact{I}{i} \BB_i)}{f_i}  &&\text{for $i = 1, \ldots, n$}\\
      &\types \plug{(\upact{J}{i} \BB_i)}{g_i}  &&\text{for $i = 1, \ldots, n$}\\
      &\erase{g'_i} = \erase{g_i} &&\text{for object boundary $\BB_i$} \\
      &\types \plug{(\upact{I}{i} \BB_i)}{f_i \equiv g'_i \by \alpha_i} &&\text{for object boundary $\BB_i$} \\
      &\types t' : \act{I} A &&\erase{t'} = \erase{\act{J} t} \\
      & &&\suitable{\beta}
    \end{aligned}}
    }{
      \types \act{I} t \equiv t' : \act{I} A \by \beta
    }
  \end{equation*}
\end{definition}

\begin{example}
  The context-free congruence rules for~$\Uppi$ from \cref{ex:pi-congruence-rule} take the form
  \begin{equation*}
    \infer{
      \types \isType {A_1} \\
      \types \abstr{x \of A_1} \; \isType {B_1}
      \\\\
      \types \isType {A_2} \\
      \types \abstr{x \of A_2} \; \isType {B_2}
      \\\\
      \erase{A'_2} = \erase{A_2} \\
      \erase{\abstr{x} B'_2} = \erase{\abstr{x} B_2}
      \\\\
      \types A_1 \equiv A'_2 \by \alpha_1 \\
      \types \abstr{x \of A_1} \; B_1 \equiv B'_2 \by \alpha_2
    }{
      \types
      \Uppi(A_1, \abstr{x} B_1)
      \equiv
      \Uppi(A_2, \abstr{x} B_2)
      \by \beta
    }
  \end{equation*}
  where the minimal suitable $\beta$ is
  \begin{equation*}
    (\alpha_1 \cup \alpha_2 \cup \asm{A'_2, \abstr{x} B_2'}) \setminus
    (\asm{A_1, A_2, \abstr{x} B_1, \abstr{x} B_2}).
  \end{equation*}
  The type expressions $A'_2$ and $B_2'$ may be chosen in such a way that the equations       $\types A_1 \equiv A'_2 \by \alpha_1$ and $\types \abstr{x \of A_1} \; B_1 \equiv B'_2 \by \alpha_2$ are well-typed, so long as they match $A_2$ and $B_2$ up to erasure.
  In this case, we expect to be able to directly use $A_2$ for $A'_2$.
  The equation $\types \abstr{x \of A_1} \; B_1 \equiv B_2 \by \alpha_2$ where we use $B_2$ instead of $B_2'$ is not obviously well-typed, as $B_2$ is a family over $A_2$ rather than $A_1$.
  Intuitively, $B_2'$ should thus be $B_2$ where uses of $x$ have to first convert along the equation $A_1 \equiv A_2 \by \alpha_1$.
\end{example}

The context-free metavariable closure rules are in direct analogy with the usual ones from \cref{def:metavariable-rule}:

\begin{definition}
  \label{def:context-free-metavariable-rule}%
  The \defemph{context-free metavariable rules} associated with the metavariable $\symM^\BB$ where $\BB = (\abstr{x_1 \of A_1} \cdots \abstr{x_n \of A_n}\; \bdry)$ are the closure rules
  \begin{equation*}
    \inferenceRuleT{CF-Meta}{
      {\begin{aligned}
      &\types t_i : A_i[\upto{\vec{t}}{i}/\upto{\vec{x}}{i}] &&\text{for $i = 1, \ldots, n$} \\
      &\types \bdry[\vec{t}/\vec{x}]
      \end{aligned}}
    }{
      \types \plug{(\bdry[\vec{t}/\vec{x}])}{\symM^{\BB}(\vec{t})}
    }
  \end{equation*}
  where $\vec{x} = (x_1, \ldots, x_m)$, $\vec{t} = (t_1, \ldots, t_m)$.
  Furthermore, if $\bdry$ is an object boundary, then the \defemph{metavariable congruence rules} for $\symM^\BB$ are the closure rules \rref{CF-Meta-Congr-Ty} and \rref{CF-Meta-Congr-Tm} displayed in \cref{fig:context-free-struct-rules}.
\end{definition}

The following definition of context-free raw type theories is analogous to \cref{def:raw-type-theory}, except that we have to use the context-free versions of structural rules.

\begin{definition}
  \label{def:context-free-raw-type-theory}%
  A \defemph{context-free raw type theory~$T$} over a signature~$\Sigma$ is a family of context-free raw rules, called the \defemph{specific rules} of~$T$.
  The \defemph{associated deductive system} of~$T$ consists of:
  \begin{enumerate}
  \item the \defemph{structural rules} over~$\Sigma$:
    \begin{enumerate}
    \item the \emph{variable}, \emph{metavariable}, \emph{metavariable congruence}, and \emph{abstraction} closure rules (\cref{fig:context-free-struct-rules}),
    \item the \emph{equality} closure rules (\cref{fig:context-free-equality-rules}),
    \item the \emph{boundary} closure rules (\cref{fig:context-free-well-formed-boundaries});
    \end{enumerate}
  \item the instantiations of the specific rules of~$T$ (\cref{def:context-free-rule-instantiation});
  \item for each specific object rule of~$T$, the instantiations of the associated congruence rule (\cref{def:context-free-congruence-rule}).
  \end{enumerate}
  We write $\types_{T} \JJ$ when $\types \JJ$ is derivable with respect to the deductive system associated to~$T$, and similarly for $\types_{T} \BB$.
\end{definition}

The formulations of the abstraction rules \rref{CF-Abstr} and \rref{CF-Bdry-Abstr} are suitable for the backward-chaining style of proof, because their conclusions take a general form. For forward-chaining, we may derive  abstraction rules with premises in general form as follows:
\begin{mathpar}
  \inferenceRule{CF-Abstr-Fwd}
    {
      \types \isType A \\
      \types \JJ \\
      \avar{a}{A} \not\in \fvt{\JJ}
    }{
      \types \abstr{x \of A} \; \JJ[x/\avar{a}{A}]
    }

    \inferenceRule{CF-Bdry-Abstr-Fwd}
    {
      \types \isType A \\
      \types \BB \\
      \avar{a}{A} \not\in \fvt{\BB}
    }{
      \types \abstr{x \of A} \; \BB[x/\avar{a}{A}]
    }
\end{mathpar}
The side condition $\avar{a}{A} \not\in \fvt{\JJ}$ ensures that $\avar{a}{A} \not\in \fv{\JJ[x/\avar{a}{A}]}$, hence \rref{CF-Abstr-Fwd} can be derived as the instance of \rref{CF-Abstr}
\begin{equation*}
  \inferrule
  {
    \types \isType A \\
    \avar{a}{A} \not\in \fv{\JJ[x/\avar{a}{A}]} \\
    \types (\JJ[x/\avar{a}{A}])[\avar{a}{A}/x]
  }{
    \types \abstr{x \of A} \JJ[x/\avar{a}{A}]
  }
\end{equation*}
and similarly for boundary abstractions.

\begin{figure}[htp]
  \centering
  \begin{ruleframe}
  \small
  \begin{mathpar}
    \inferenceRule{CF-Var} { }
    { \types \avar{a}{A} : A }

    \inferenceRule{CF-Abstr}
    {
      \types \isType A \\
      \avar{a}{A} \not\in \fv{\JJ} \\
      \types \JJ[\avar{a}{A}/x]
    }{
      \types \abstr{x \of A} \; \JJ
    }

    \inferenceRule{CF-Meta}{
      {\begin{aligned}
      &\types t_i : A_i[\upto{\vec{t}}{i}/\upto{\vec{x}}{i}] &&\text{for $i = 1, \ldots, n$} \\
      &\types \bdry[\vec{t}/\vec{x}]
      \end{aligned}}
    }{
      \types \plug{(\bdry[\vec{t}/\vec{x}])}{\symM^{\abstr{\vec{x} \of \vec{A}} \bdry}(\vec{t})}
    }

    \\

    \inferenceRule{CF-Meta-Congr-Ty}
    { { \begin{aligned}
          &\mathrlap{\BB = \abstr{x_1 \of A_1} \cdots \abstr{x_m \of A_m}\; \isType{\Box}}
          \\
          &\types s_j : A[\upto{\vec{s}}{j}/\upto{\vec{x}}{j}]
            &&\text{for $j = 1, \ldots, m$} \\
          &\types t_j : A[\upto{\vec{t}}{j}/\upto{\vec{x}}{j}]
            &&\text{for $j = 1, \ldots, m$} \\
          &\erase{t_j} = \erase{t'_j}
            &&\text{for $j = 1, \ldots, m$} \\
          &\types s_j \equiv t'_j : A_j[\upto{\vec{s}}{j}/\upto{\vec{x}}{j}] \by \alpha_j
            &&\text{for $j = 1, \ldots, m$} \\
          &\suitable{\beta}
        \end{aligned} }
    }{
      \types
      \symM^\BB(\vec{s}) \equiv
      \symM^\BB(\vec{t})
      \by \beta
    }

    \inferenceRule{CF-Meta-Congr-Tm}
    { { \begin{aligned}
          &\mathrlap{\BB = \abstr{x_1 \of A_1} \cdots \abstr{x_m \of A_m}\; \Box : B}
          \\
          &\types s_j : A_j[\upto{\vec{s}}{j}/\upto{\vec{x}}{j}]
             &\text{for $j = 1, \ldots, m$}\\
          &\types t_j : A_j[\upto{\vec{t}}{j}/\upto{\vec{x}}{j}]
             &\text{for $j = 1, \ldots, m$}\\
          &\erase{t_j} = \erase{t'_j}
             &\text{for $j = 1, \ldots, m$}\\
          &\types s_j \equiv t'_j : A_j[\upto{\vec{s}}{j}/\upto{\vec{x}}{j}] \by \alpha_j
             &\text{for $j = 1, \ldots, m$}\\
        \end{aligned} }
      \\\\
      \types v : B[\vec{s}/\vec{x}] \\
      \erase{\symM^\BB(\vec{t})} = \erase{v} \\
      \suitable{\beta}
    }{
      \types
      \symM^\BB(\vec{s}) \equiv
      v
      : B[\vec{s}/\vec{x}]
      \by \beta
    }

  \end{mathpar}
  \end{ruleframe}
  \caption{Context-free free variable, metavariable, and abstraction closure rules}
  \label{fig:context-free-struct-rules}
\end{figure}

\begin{figure}[htp]
  \centering
  \begin{ruleframe}
  \small
  \begin{mathpar}
  \inferenceRule{CF-EqTy-Refl}{
    \types \isType{A_1} \\
    \types \isType{A_2} \\
    \erase{A_1} = \erase{A_2}
  }{
    \types A_1 \equiv A_2 \by \asset{}
  }

  \inferenceRule{CF-EqTy-Sym}
  {
    \types A \equiv B \by \alpha
  }{
    \types B \equiv A \by \alpha
  }

  \inferenceRule{CF-EqTy-Trans}
  {
    \types A \equiv B \by \alpha \\
    \erase{B} = \erase{B'}
    \types B' \equiv C \by \beta \\
    \suitable{\gamma}
  }{
    \types A \equiv C \by \gamma
  }

  \\

  \inferenceRule{CF-EqTm-Refl}
  {
    \types t_1 : A \\
    \types t_2 : A \\
    \erase{t_1} = \erase{t_2}
  }{
    \types t_1 \equiv t_2 : A \by \asset{}
  }

  \inferenceRule{CF-EqTm-Sym}
  {
    \types s \equiv t : A \by \alpha
  }{
    \types t \equiv s : A \by \alpha
  }

  \inferenceRule{CF-EqTm-Trans}
  { \types s \equiv t : A \by \alpha \\
    \erase{t} \equiv \erase{t'} \\
    \types t' \equiv u : A \by \beta \\
    \suitable{\gamma}
  }{
    \types s \equiv u : A \by \gamma
  }

  \inferenceRule{CF-Conv-Tm}
  { \types t : A \\
    \types A \equiv B \by \alpha
    \\\\
    \asm{t, A, B, \alpha} = \asm{t, B, \beta}
  }{
    \types \convert{t}{\beta} : B
  }

  \inferenceRule{CF-Conv-EqTm}
  {
    \types s \equiv t : A \by \alpha \\
    \types A \equiv B \by \beta \\\\
    \asm{s, A, B, \beta} = \asm{s, B, \gamma} \\\\
    \asm{t, A, B, \beta} = \asm{t, B, \delta}
  }{
    \types
    \convert{s}{\gamma} \equiv \convert{t}{\delta} : B \by \alpha
  }

  \end{mathpar}
  \end{ruleframe}
  \caption{Context-free closure rules for equality}
  \label{fig:context-free-equality-rules}
\end{figure}

\begin{figure}[htp]
  \centering
  \begin{ruleframe}
  \small
  \begin{mathpar}
    \inferenceRule{CF-Bdry-Ty}
    {
    }{
     \types \isType \Box
    }

    \inferenceRule{CF-Bdry-Tm}
    {
      \types \isType{A}
    }{
      \types \Box : A
    }

    \inferenceRule{CF-Bdry-EqTy}
    {
      \types \isType{A} \\
      \types \isType{B}
    }{
      \types A \equiv B \by \Box
    }

    \inferenceRule{CF-Bdry-EqTm}
    {
      \types \isType{A} \\
      \types s : A \\
      \types t : A
    }{
      \types s \equiv t : A \by \Box
    }

    \inferenceRule{CF-Bdry-Abstr}
    {
      \types \isType A \\
      \avar{a}{A} \not\in \fv{\BB} \\
      \types \BB[\avar{a}{A}/x]
    }{
      \types \abstr{x \of A} \; \BB
    }
  \end{mathpar}
  \end{ruleframe}
  \caption{Well-formed context-free abstracted boundaries}
  \label{fig:context-free-well-formed-boundaries}
\end{figure}

The context-free analogues of the auxiliary judgements $\types \isMCtx{\Theta}$ and $\Theta \types \isVCtx{\Gamma}$ are as follows. For simplicity we define a single notion that encompasses the well-formedness of all annotations.

\begin{definition}\label{def:well-typed-annotations}
  An expression $e$ has \defemph{well-typed annotations} when
  $\types \BB$ for every $\symM^\BB \in \asm{e}$ and $\types \isType{A}$ for every $\avar{a}{A} \in \asm{e}$.
  The notion evidently extends to judgements and boundaries.
\end{definition}

The context-free version of finitary rules and type theories is quite similar to the original one.

\begin{definition}
  \label{def:context-free-finitary}%
  Given a raw theory~$T$ over a signature~$\Sigma$, a context-free raw rule
  $\rawRule{\symM_1^{\BB_1}, \ldots, \symM_n^{\BB_n}}{\plug{\bdry}{e}}$
  over $\Sigma$ is \defemph{finitary} over~$T$ when $\types_T \BB_i$ for $k = 1, \ldots, n$, and $\types_T \bdry$,
  Similarly, a raw rule-boundary $\rawRule{\symM_1^{\BB_1}, \ldots, \symM_n^{\BB_n}}{\bdry}$
  is finitary over $T$ when $\types_T \BB_i$ for $k = 1, \ldots, n$, and $\types_T \bdry$.

  A \defemph{context-free finitary type theory} is a context-free raw type theory $(R_i)_{i \in I}$ for which there exists a well-founded order $(I, \prec)$ such that each $R_i$ is finitary over $(R_j)_{j \prec i}$.
\end{definition}

\begin{definition}
  \label{def:context-free-standard-type-theory}%
  A context-free finitary type theory is \defemph{standard} if its specific object rules are symbol rules, and each symbol has precisely one associated rule.
\end{definition}

\section{Meta-theorems about context-free theories}
\label{sec:context-free-meta-theorems}

The meta-theorems from \cref{sec:meta-theorems} carry over to the context-free setting. Unfortunately, there seems to be no wholesale method for transferring the proofs, and one simply has to adapt them manually to the context-free setting. The process is quite straightforward, so we indulge in omitting the details.

\subsection{Meta-theorems about context-free raw theories}
\label{sec:meta-theorems-cf-raw}

In the context-free setting, a renaming is still an injective map~$\rho$ taking unannotated symbols to unannotated symbols. Its action~$\act{\rho} e$ on an expression~$e$ recursively descends into~$e$, including into variable annotations, i.e.\ $\act{\rho} (\avar{a}{A}) = \rho(\mkvar{a})^{\act{\rho} A}$ and $\act{\rho} (\symM^\BB) = \rho(\symM)^{\act{\rho} \BB}$. The action is extended to judgements and boundaries in a straightforward manner. Renaming preserves the size of an expression, as long as all symbols are deemed to have the same size.

\begin{proposition}[Context-free renaming]
  \label{prop:context-free-renaming}%
  If a context-free raw type theory derives a judgement or a boundary, then it also derives its renaming.
\end{proposition}

\begin{proof}
  Straightforward induction on the derivation.
\end{proof}

Weakening (\cref{prop:tt-weakening}) is not applicable, as there is no context that could be weakened, and no variable ever occurs in the conclusion of a judgement without it being used in the derivation.

We next prove admissibility of substitution rules. We take a slightly different route than in \cref{sec:meta-theorems-raw} in order to avoid substituting a term for a free variable, as that changes type annotations and therefore the identity of variables.
\Cref{lem:context-free-prepare-subst,lem:context-free-prepare-subst-bdry} are proved by mutual structural induction, with a further structural induction within each lemma.

\begin{restatable}{lemma}{restatecontextfreepreparesubst}
  \label{lem:context-free-prepare-subst}%
  If a context-free raw type theory derives
  \begin{align*}
    &\vdash \abstr{x_1 \of A_1} \cdots \abstr{x_n \of A_n} \; \JJ &&\text{and} \\
    &\vdash t_i : A_i[\upto{\vec{t}}{i}/\upto{\vec{x}}{i}] &&\text{for $i = 1, \ldots, n$}
  \end{align*}
  then it derives $\vdash \JJ[\vec{t}/\vec{x}]$.
\end{restatable}

\begin{proof}
  See the proof on \cpageref{proof:context-free-prepare-subst}.
\end{proof}

\begin{lemma}
  \label{lem:context-free-prepare-subst-bdry}%
  If a context-free raw type theory derives
  \begin{align*}
    &\vdash \abstr{x_1 \of A_1} \cdots \abstr{x_n \of A_n} \; \BB &&\text{and} \\
    &\vdash t_i : A_i[\upto{\vec{t}}{i}/\upto{\vec{x}}{i}] &&\text{for $i = 1, \ldots, n$}
  \end{align*}
  then it derives $\vdash \BB[\vec{t}/\vec{x}]$.
\end{lemma}

\begin{proof}
  We proceed as in the proof of \cref{lem:context-free-prepare-subst},
  where \rref{CF-Bdry-Abstr} is treated like \rref{CF-Abstr}, and the remaining ones invert to
  \cref{lem:context-free-prepare-subst}.
\end{proof}

\begin{theorem}[Context-free admissibility of substitution]
  \label{thm:context-free-substitution-admissible}
  In a context-free raw type theory, the following substitution rules are admissible:
  \begin{mathpar}
      \inferenceRule{CF-Subst}
      {
        \types \abstr{x \of A} \; \JJ \\
        \types t : A
      }{
        \types \JJ[t/x]
      }

      \inferenceRule{CF-Bdry-Subst}
      {
        \types \abstr{x \of A} \; \BB
        \\
        \types t : A
      }{
        \types \BB[t/x]
      }
  \end{mathpar}
\end{theorem}

\begin{proof}
  The admissibility of \rref{CF-Subst} and \rref{CF-Bdry-Subst} corresponds to the case $n = 1$ of \cref{lem:context-free-prepare-subst} and \cref{lem:context-free-prepare-subst-bdry}, respectively.
\end{proof}

Before addressing the context-free versions of \rref{TT-Subst-EqTy} and \rref{TT-Subst-EqTm}, we prove the context-free presuppositivity theorem.

Of course, presuppositivity holds in the context-free setting as well.

\begin{restatable}[Context-free presuppositivity]{theorem}{restatecontextfreepresuppositivity}
  \label{prop:context-free-presuppositivity}
  If a context-free raw type theory derives $\types \plug{\BB}{e}$ and $\plug{\BB}{e}$ has well-typed annotations, then it derives $\types \BB$.
\end{restatable}

\begin{proof}
  See the proof on \cpageref{proof:context-free-presuppositivity}.
\end{proof}

Let us now turn to meta-theorems stating that equal substitutions act equally. Once again we need to account for insertion of conversions. In congruence rules such conversions appeared in premises: equations associated to object premises of the shape $\plug{(\upact I i \BB_i)}{f_i \equiv g_i' \by \alpha_i}$ referred to a primed version of $g_i$ to allow the use of conversions in $g_i$.
In the following lemma, conversions appear in the result of a substitution. Therefore, rather than being permissive about insertions of conversions, we are faced with showing that it is possible to insert them.
Similarly to \cref{lem:prepare-subst-eq}, we prove that equal terms can be substituted into a judgement to yield equal results, but the right hand side of these results is only prescribed up to erasure, namely as $C'$ and $u'$.

\begin{restatable}{lemma}{restatecontextfreepreparesubsteq}
  \label{lem:context-free-prepare-subst-eq}%
  If a context-free raw type theory derives
  \begin{align}
    &\types \abstr{x_1 \of A_1} \cdots \abstr{x_n \of A_n} \; \JJ  \notag \\
    \intertext{
      where $\abstr{\vec{x} \of \vec{A}} \; \JJ$ has well-typed annotations,
      and for $i = 1, \ldots, n$}
    &\types s_i : A_i[\upto{\vec{s}}{i}/\upto{\vec{x}}{i}]  \notag \\
    &\types t_i : A_i[\upto{\vec{t}}{i}/\upto{\vec{x}}{i}] \notag \\
    &\types s_i \equiv t'_i : A_i[\upto{\vec{s}}{i}/\upto{\vec{x}}{i}] \by \alpha_i
     \qquad\text{and $\erase{t'_i} = \erase{t_i}$}.
    \label{eq:cf-prep-subst-eq-5}
  \end{align}
  then:
  \begin{enumerate}
  \item \label{it:cf-prep-subst-eq-1}%
    if $\JJ = (\abstr{\vec{y} \of \vec{B}} \; \isType{C})$ then there are $\gamma$ and $C'$ such that
    $\erase{C[\vec{t}/\vec{x}]} = \erase{C'}$,
    \begin{equation*}
      \types \abstr{\vec{y} \of \vec{B}[\vec{s}/\vec{x}]} \;
      C[\vec{s}/\vec{x}] \equiv C' \by \gamma,
    \end{equation*}

  \item \label{it:cf-prep-subst-eq-2}%
    if $\JJ = (\abstr{\vec{y} \of \vec{B}} \; u : C)$ then there are $\delta$ and $u'$ such that $\erase{u[\vec{t}/\vec{x}]} = \erase{u'}$ and
    \begin{equation*}
      \types \abstr{\vec{y} \of \vec{B}[\vec{s}/\vec{x}]} \;
        u[\vec{s}/\vec{x}] \equiv u' : C[\vec{s}/\vec{x}] \by \delta.
    \end{equation*}
  \end{enumerate}
  Furthermore, no extraneous assumptions are introduced by $\gamma$, $C'$, $\delta$ and $u'$:
  \begin{equation*}
    \asm{\abstr{\vec{y}} \gamma,
         \abstr{\vec{y}} C',
         \abstr{\vec{y}} \delta,
         \abstr{\vec{y}} u'}
    \subseteq
    \asm{\vec{s}, \vec{t}, \vec{t}', \vec{\alpha}, \abstr{\vec{x} \of \vec{A}} \; \JJ}.
  \end{equation*}
\end{restatable}

\begin{proof}
  See the proof on \cpageref{proof:context-free-prepare-subst-eq}.
\end{proof}

\begin{restatable}{theorem}{restatecontextfreesubsteq}
  \label{thm:context-free-subst-eq}%
  In a context-free raw type theory, the following rules are admissible:
  \begin{mathpar}
    \inferenceRule{CF-Subst-EqTy}
    {{\begin{aligned}
      &\types \abstr{\vec{x} \of \vec{A}} \abstr{\vec{y} \of \vec{B}} \; \isType{C} \\
      &\types s_i : A_i[\upto{\vec{s}}{i}/\upto{\vec{x}}{i}] &&\text{for $i = 1, \ldots, n$} \\
      &\types t_i : A_i[\upto{\vec{t}}{i}/\upto{\vec{x}}{i}] &&\text{for $i = 1, \ldots, n$} \\
      &\erase{t_i} = \erase{t'_i}                            &&\text{for $i = 1, \ldots, n$} \\
      &\types s_i \equiv t'_i : A_i[\upto{\vec{s}}{i}/\upto{\vec{x}}{i}] \by \alpha_i
          &&\text{for $i = 1, \ldots, n$} \\
      & &&\suitable{\beta}
     \end{aligned}}
    }{
      \types \abstr{\vec{y} \of \vec{B}[\vec{s}/\vec{x}]} \;
        C[\vec{s}/x] \equiv C[\vec{t}/x] \by \beta
    }

    \inferenceRule{CF-Subst-EqTm}
    {{\begin{aligned}
      &\types \abstr{\vec{x} \of \vec{A}} \abstr{\vec{y} \of \vec{B}} \; u : C \\
      &\types s_i : A_i[\upto{\vec{s}}{i}/\upto{\vec{x}}{i}] &&\text{for $i = 1, \ldots, n$} \\
      &\types t_i : A_i[\upto{\vec{t}}{i}/\upto{\vec{x}}{i}] &&\text{for $i = 1, \ldots, n$} \\
      &\erase{t_i} = \erase{t'_i}                            &&\text{for $i = 1, \ldots, n$} \\
      &\types s_i \equiv t'_i : A_i[\upto{\vec{s}}{i}/\upto{\vec{x}}{i}] \by \alpha_i
          &&\text{for $i = 1, \ldots, n$} \\
      & &&\suitable{\beta}
      \end{aligned}}
    }{
      \types \abstr{\vec{y} \of \vec{B}[\vec{s}/\vec{x}]} \;
                u[\vec{s}/\vec{x}] \equiv \convert{u[\vec{t}/\vec{x}]}{\beta} : C[\vec{s}/x] \by \beta
    }
  \end{mathpar}
\end{restatable}

\begin{proof}
  See the proof on \cpageref{proof:context-free-subst-eq}.
\end{proof}

Lastly, we prove the context-free counterpart of instantiation admissibility \cref{prop:instantiation-admissible}.
The notion of a derivable instantiation carries over easily to the context-free setting: $I = \finmap{\symM^{\BB_1}_1 \mto e_1, \ldots, \symM^{\BB_n}_n \mto e_n}$ is \defemph{derivable} when $\types \plug{(\upact{I}{i} \BB_i)}{e_i}$ for every $i = 1, \ldots, n$.

\begin{restatable}[Context-free admissibility of instantiation]{theorem}{restatecontextfreeinstantiationadmissible}
  \label{prop:context-free-instantiation-admissible}%
  In a raw type theory,
  if $\types \JJ$ is derivable, it has well-typed annotations, and~$I$ is a derivable instantiation such that $\mv{\JJ} \subseteq |I|$, then $\types \act{I} \JJ$ is derivable, and similarly for boundaries.
\end{restatable}

\begin{proof}
  See the proof on \cpageref{proof:context-free-instantiation-admissible}.
\end{proof}

\subsection{Meta-theorems about context-free finitary theories}
\label{sec:meta-theorems-cf-finitary}

The context-free economic rules for finitary theories carry over to the context-free setting. The proofs are analogous to those of \cref{sec:meta-thm-finitary} so we omit them.

\begin{proposition}
  \label{prop:cf-specific-eco}[Economic version of \cref{def:context-free-rule-instantiation}]
  Let $R$ be the context-free raw rule $\rawRule{\Xi}{\plug{\bdry}{e}}$ with $\Xi = [\symM_1^{\BB_1}, \ldots, \symM_n^{\BB_n}]$ such that $\types \bdry$ is derivable, in particular~$R$ may be finitary. Then for any instantiation $I = [\symM_1^{\BB_1} \mto e_1, \ldots, \symM_n^{\BB_n} \mto e_n]$, the following closure rule is admissible:
  \begin{equation*}
    \inferenceRule{CF-Specific-Eco}
    {\types \plug {(\upact{I}{i} \BB_i)}{e_i} \quad \text{for $i = 1, \ldots, n$}}
    {\types \act{I} (\plug{\bdry}{e})}
  \end{equation*}
\end{proposition}

\begin{proposition}[Economic version of \cref{def:context-free-metavariable-rule}]
  \label{prop:cf-meta-eco}
  In a context-free raw type theory, if $\BB = \abstr{x_1 \of A_1} \cdots \abstr{x_m \of A_m} \; \bdry$, $\vec{s}$, and $\vec{t}$ have well-typed annotations, then the following closure rules are admissible:
  \begin{mathpar}
    \inferenceRule{CF-Meta-Eco}
    {
      \types t_j : A_j[\upto{\vec{t}}{j}/\upto{\vec{x}}{j}]
      \quad \text{for $j = 1, \ldots, m$}
    }{
      \types \plug{(\bdry[\vec{t}/\vec{x}])}{\symM(\vec{t})}
    }

    \inferenceRule{CF-Meta-Congr-Eco}
    { \types s_j \equiv t_j :
      A_j[\upto{\vec{s}}{j}/\upto{\vec{x}}{j}]
      \quad \text{for $j = 1, \ldots, m$}
    }{
      \types
      \plug
      {(\bdry[\vec{s}/\vec{x}])}
      {\symM_k(\vec{s}) \equiv \symM_k(\vec{t})}
    }
  \end{mathpar}
\end{proposition}

\subsection{Meta-theorems about context-free standard theories}
\label{sec:meta-theorems-cf-standard}

Inversion and uniqueness of typing (\cref{thm:inversion,thm:uniqueness-of-typing}) carry over to context-free finitary theories. First, the notion of natural type is simpler, as it does not depend on the context anymore.

\begin{definition}
  \label{def:context-free-natural-type}%
  Let $T$ be a finitary type theory.
  The \emph{natural type} $\natty{}{t}$ of a term expression~$t$ is defined by:
  \begin{align*}
    \natty{}{\avar{a}{A}}
    &= A, \\
    \natty{}{\symM^{\BB}(t_1, \ldots, t_m)}
    &=
      \begin{aligned}[t]
      &A[t_1/x_1, \ldots, t_m/x_m]\\
      &\qquad
        \text{where $\BB = (\abstr{x_1 \of A_1} \cdots \abstr{x_m \of A_m} \; \Box : A)$}
      \end{aligned}
    \\
    \natty{}{\symS(e_1, \ldots, e_n)}
    &=
      \begin{aligned}[t]
        &\act{\finmap{\symM_1 \mto e_1, \ldots, \symM_n \mto e_n}} B\\
        &\qquad
          \begin{aligned}[t]
            &\text{where the symbol rule for $\symS$ is}\\
            &\text{$\rawRule{\symM_1^{\BB_1}, \ldots, \symM_n^{\BB_n}}{\Box : B}$}
          \end{aligned}
      \end{aligned}
    \\
    \natty{}{\convert{t}{\alpha}}
    &= \natty{}{t}
  \end{align*}
\end{definition}

Next, we define an operation which peels conversions off a term, and another one that collects the peeled assumptions sets. We shall use these in the formulation of the context-free inversion theorem.

\begin{definition}
  \label{def:context-free-stripping}%
  The \defemph{conversion-stripping} $\strip t$ of a term expression~$t$ is defined by:
  \begin{equation*}
    \strip{t} =
    \begin{cases}
      \strip{t'} &\text{if $t = \convert{t'}{\alpha}$,} \\
      t &\text{otherwise.}
    \end{cases}
  \end{equation*}
  The \defemph{conversion-residue} $\residue{t}$ is defined by
  \begin{equation*}
    \residue{t} =
    \begin{cases}
      \alpha \cup \residue{t'} &\text{if $t = \convert{t'}{\alpha}$,} \\
      \asset{} &\text{otherwise.}
    \end{cases}
  \end{equation*}
\end{definition}

Note that $\erase{t} = \erase{\strip{t}}$ and that $\asm{t} = \asm{\strip{t}, \residue{t}}$.

\begin{restatable}{lemma}{restatecontextfreenaturaltype}
  \label{lem:context-free-natural-type}%
  If a context-free standard type theory derives $\types t : A$ then
  \begin{enumerate}
  \item it derives $\types \strip{t} : \natty{}{t}$ by an application of \rref{CF-Var},
    \rref{CF-Meta}, or an instantiation of a term symbol rule, and
  \item it derives $\types \natty{}{t} \equiv A \by \residue{t}$.
  \end{enumerate}
\end{restatable}

\begin{proof}
  See the proof on \cpageref{proof:context-free-natural-type}.
\end{proof}

\begin{theorem}[Context-free inversion]
  \label{thm:context-free-inversion}%
  If a context-free standard type theory derives $\types t : A$, then:
  \begin{itemize}
  \item %
    if $A = \natty{}{t}$, it derives $\types \strip t : \natty{}{t}$ by a derivation which concludes with
    \rref{CF-Var},
    \rref{CF-Meta},
    or an instantiation of a term symbol rule;
  \item %
    if $A \neq \natty{}{t}$, it derives $\types \convert{\strip{t}}{\residue{t}} : A$ by \rref{CF-Conv-Tm}.
  \end{itemize}
\end{theorem}

\begin{proof}
  Apply \cref{lem:context-free-natural-type} and, depending on whether $A = \natty{}{t}$, either use $\types \strip{t} : \natty{}{t}$ so obtained directly or convert it along $\types \natty{}{t} \equiv A \by \residue{t}$,
  observing that the side condition
  $
  \asm{\strip{t}, \natty{}{t}, A, \residue{t}} = \asm{\strip{t}, \residue{t}, A}
  $
  holds because $\asm{\natty{}{t}} \subseteq \asm{t} = \asm{\strip{t}, \residue{t}}$.
\end{proof}

\begin{theorem}[Context-free uniqueness of typing]%
  \label{thm:context-free-uniqueness-of-typing}%
  For a context-free standard type theory:
  \begin{enumerate}
  \item If $\types t : A$ and $\types t : B$, then $\types A \equiv B \by \alpha$ for some assumption set $\alpha$.
  \item If $\types s \equiv t : A \by \beta_1$ and $\types s \equiv t : B \by \beta_2$, with well-typed variables, then $\types A \equiv B \by \alpha$ for some assumption set $\alpha$.
  \end{enumerate}
   In both cases, $\alpha \subseteq \asm{t}$ can be computed from the judgements involved, without recourse to their derivations.
\end{theorem}

\begin{proof}
  The first statement holds because $A$ and $B$ are both judgmentally equal to the natural type of~$t$ by \cref{lem:context-free-natural-type}. The second statement reduces to the first one because the presuppositions $\types t : A$ and $\types t : B$ are derivable by~\cref{prop:context-free-presuppositivity}.
\end{proof}

\subsection{Special meta-theorems about context-free theories}
\label{sec:meta-theorems-cf-specific}

We prove several meta-theorems which are specific to context-free type theories.
The example of the equality reflection rule in the beginning of \cref{sec:context-free-tt} showcased that finitary type theories do not enjoy strengthening. Context-free type theories, however, do satisfy this meta-property.

\begin{theorem}[Strengthening]
  \label{thm:context-free-strengthening}%
  If a context-free raw type theory derives
  \begin{equation*}
    \types \abstr{\vec{y} \of \vec{B}} \abstr{x \of A}\; \JJ
  \end{equation*}
  and $x \not\in \bv{\JJ}$ then it also derives $\types \abstr{\vec{y} \of \vec{B}}\; \JJ$.
\end{theorem}

\begin{proof}
  We proceed by induction on the derivation of $\abstr{\vec{y} \of \vec{B}} \abstr{x \of A}\; \JJ$.
  The only case to consider is \rref{CF-Abstr}.
  If the outer abstraction is empty, then the derivation ends with the abstraction
  \begin{equation}
    \label{eq:strengthening}
    \infer
    {
      \types \isType A \\
      \avar{a}{A} \not\in \fv{\JJ} \\
      \types \JJ[\avar{a}{A}/x]
    }{
      \types \abstr{x \of A} \; \JJ
    }
  \end{equation}
  Because $x \not\in \bv{\JJ}$, it follows that $\avar{a}{A} \not\in \fvz{\JJ[\avar{a}{A}/x]}$ and that $\JJ[\avar{a}{A}/x] = \JJ$, which is the second premise, hence derivable.
  The other possibility is that the derivation ends with
  \begin{equation*}
    \infer
    {
      \types \isType A \\
      \avar{c}{{}C} \not\in \fv{\abstr{\vec{y} \of \vec{B}} \abstr{x \of A}\; \JJ} \\
      \types \abstr{\vec{y} \of \vec{B}[\avar{c}{C}/z]} \abstr{x \of A[\avar{c}{C}/z]}\; \JJ[\avar{c}{C}/z]
      }{
      \types \abstr{z \of C} \abstr{\vec{y} \of \vec{B}} \abstr{x \of A}\; \JJ
    }
  \end{equation*}
  From $x \not\in \bv{\JJ}$ it follows that $x \not\in \bv{\JJ[\avar{c}{C}/z]}$, hence we may apply the induction hypothesis to the second premise and conclude by abstracting~$\avar{c}{C}$.
\end{proof}

Why cannot we adapt the above proof to type theories with contexts? In the derivation \eqref{eq:strengthening}, the second premise turns out to be precisely the desired conclusion, whereas \rref{TT-Abstr} would yield $\Theta; \Gamma, \mkvar{a} \of A \types \JJ$ where $\Theta; \Gamma \types \JJ$ is needed. Indeed, strengthening is not generally valid for type theories with contexts.

The next lemma can be used to modify the head of a judgement so that it fits another boundary, as long as there is agreement up to erasure.

\begin{restatable}[Boundary conversion]{lemma}{restateboundaryconvert}
  \label{lem:boundary-convert}%
  In a context-free raw theory, if $\types \BB_1$, $\types \BB_2$, $\types \plug{\BB_1}{e_1}$ and $\erase{\BB_1} = \erase{\BB_2}$ then there is $e_2$ such that $\types \plug{\BB_2}{e_2}$, $\asm{e_2} \subseteq \asm{\plug{\BB_1}{e_1}}$ and $\erase{e_1} = \erase{e_2}$.
\end{restatable}

\begin{proof}
  See the proof on \cpageref{proof:boundary-convert}.
\end{proof}

\section{A correspondence between theories with and without contexts}
\label{sec:context-free-cons-comp}

We now establish a correpondence between finitary type theories with and without contexts. We use the prefixes ``tt`` (for ``traditional types``) and ``cf`` (for ``context-free``) to disambiguate between the two versions of type theory. Thus the raw tt-syntax is the one from \cref{fig:syntax-general-type-theories}, and the raw cf-syntax the one from
\cref{fig:syntax-context-free-type-theories}.

To ease the translation between the two versions of type theory, we shall use annotated free variables $\avar{a}{A}$ and annotated metavariables $\symM^\BB$ in both version of raw syntax, where the annotations $A$ and $\BB$ are those of the cf-syntax. In the tt-syntax these annotations are considered part of the symbol names, and do not carry any type-theoretic significance.

\subsection{Translation from cf-theories to tt-theories}
\label{sec:cf-to-tt-translation}

We first show how to translate constituents of cf-theories to corresponding constituents of tt-theories. The plan is simple enough: move the annotations to contexts, elide the conversion terms, and replace the assumption sets with the dummy value.

The first step towards the translation was taken in \cref{sec:context-free-subst-equal}, where we defined the erasure operation taking a cf-expression~$e$ to a tt-expression $\erase{e}$ by removing conversions and replacing assumptions sets with the dummy value. Note that erasure and substitution commute, $\erase{e[t/x]} = \erase{e}[\erase{t}/x]$, by an induction on the syntactic structure of~$e$.

Next, in order to translate cf-judgements to tt-judgements, we need to specify when a context correctly encodes the information provided by cf-annotations.

\begin{definition}
  \label{def:suitable-context}
  We say that $\Theta$ is a \defemph{suitable metavariable context} for a set of cf-metavariables~$S$ when $S \subseteq \vert{}\Theta\vert{}$ and $\Theta(\symM^\BB) = \erase{\BB}$ for all $\symM^\BB \in S$.
  Similarly,~$\Gamma$ is a \defemph{suitable variable context} for a set of free cf-variables~$V$ when~$V \subseteq \vert{}\Gamma\vert{}$ and $\Gamma(\avar{a}{A}) = \erase{A}$ for all $\avar{a}{A} \in V$.
  We say that $\Theta; \Gamma$ is a \defemph{suitable context} for $S$ and $V$ when $\Theta$ is suitable for~$S$ an~$\Gamma$ for~$V$.
\end{definition}

As a shorthand, we say that $\Theta; \Gamma$ is \emph{suitable} for a syntactic entity~$e$ when it is suitable for~$\mv{e}$ and~$\fv{e}$. As suitability only depends on the assumption set, it follows from suitability of $\Theta; \Gamma$ for~$e$ and $\asm{e'} \subseteq \asm{e}$ that $\Theta; \Gamma$ is also suitable for~$e'$.

Next, say that a free cf-variable $\avar{a}{A}$ \defemph{depends} on a free cf-variable $\avar{b}{B}$, written $\avar{b}{B} \prec \avar{a}{A}$, when $\avar{b}{B} \in \fv{A}$, and that a set $S$ of free cf-variables is \defemph{closed under dependence} when $\avar{b}{B} \prec \avar{a}{A} \in S$ implies $\avar{b}{B} \in S$. Every set $S$ of cf-variables is contained in the least closed set, which is $\bigcup \asset{\fv{\avar{a}{A}} \such \avar{a}{A} \in S}$. We similarly define dependence for cf-metavariables.

The following lemma shows how to construct suitable contexts.

\begin{lemma}
  \label{lem:suitable-ctx}
  For every finite set of cf-metavariables~$S$ there exists a suitable metavariable context $\Theta$, such that $\vert{}\Theta\vert{}$ is the closure of $S$ with respect to dependence.
  For every finite set of free cf-variables~$V$ there exists a suitable variable context~$\Gamma$, such that $\vert{}\Gamma\vert{}$ is the closure of $V$ with respect to dependence.
\end{lemma}

\begin{proof}
  Given a finite set of free cf-variables~$S$, the well-founded order $\prec$ on
  $\bigcup \asset{\fv{\avar{a}{A}} \such \avar{a}{A} \in S}$ may be extended to a total one, say
  $\mkvar{a}_1^{A_1}, \ldots, \mkvar{a}_n^{A_n}$.
  Now take~$\Gamma$ to be the variable context
  $\mkvar{a}_1^{A_1} : \erase{A_1}, \ldots, \mkvar{a}_n^{A_n} : \erase{A_n}$.
  The argument for metavariables is analogous.
\end{proof}

A totally ordered extension of~$\prec$ can be given explicitly, so the preceding proof yields an explicit
construction of a suitable contexts.
Notice that the construction does not introduce any spurious assumptions, in the sense that for a variable context~$\Gamma$ the constructed suitable set~$V$ contains only the variables appearing in~$\Gamma$ and the annotations of types appearing in~$\Gamma$.

\begin{proposition}
  If $\Theta; \Gamma$ is suitable for a cf-judgement~$\JJ$ then $\Theta; \Gamma \types \erase{\JJ}$ is a syntactically valid tt-judgement, and similarly for boundaries.
\end{proposition}

\begin{proof}
  A straightforward induction on the structrure of the judgement~$\JJ$.
\end{proof}

Next we translate rules, theories, and derivations.

\begin{proposition}
  \label{prop:cf-to-tt-raw-rule}
  A cf-rule and a cf-rule-boundary
  \begin{equation*}
    \rawRule{\symM_1^{\BB_1}, \ldots, \symM_n^{\BB_n}}{\J}
    \qquad\text{and}\qquad
    \rawRule{\symM_1^{\BB_1}, \ldots, \symM_n^{\BB_n}}{\bdry}
  \end{equation*}
  respectively translate to the raw tt-rule and the tt-rule-boundary
  \begin{equation*}
    \rawRule{\symM_1^{\BB_1} \of \erase{\BB_1}, \ldots, \symM_n^{\BB_n} \of \erase{\BB_n}}{\erase{\J}}
  \end{equation*}
  and
  \begin{equation*}
    \rawRule{\symM_1^{\BB_1} \of \erase{\BB_1}, \ldots, \symM_n^{\BB_n} \of \erase{\BB_n}}{\erase{\bdry}}.
  \end{equation*}
  A raw-cf theory $T = \finmap{R_i}_{i \in I}$ over a signature~$\Sigma$ is thus translated rule-wise to the raw tt-theory $\trantt{T} = \finmap{\trantt{(R_i)}}_{i \in I}$ over the same signature.
\end{proposition}

\begin{proof}
  The conditions in \cref{def:context-free-raw-rule} guarantee that
  $
  \symM_1^{\BB_1} \of \erase{\BB_1}, \ldots, \symM_n^{\BB_n} \of \erase{\BB_n}
  $
  is a metavariable context and that it is suitable for $\erase{\J}$ and $\erase{\bdry}$.
\end{proof}

\begin{restatable}[Translation from finitary cf- to tt-theories]{theorem}{restatecftottbdryjdg}
  \label{thm:cf-to-tt-bdry-jdg}%
  \parbox{0pt}{} %
  \begin{enumerate}
  \item\label{it:cf-to-tt-1}
    The translation of a finitary cf-theory is finitary.
  \item\label{it:cf-to-tt-2}
    Suppose $T$ is a finitary cf-theory whose translation $\trantt{T}$ is also finitary.
    Let $\Theta; \Gamma$ be tt-context such that $\types_{\trantt{T}} \isMCtx{\Theta}$ and $\Theta \types_{\trantt{T}} \isVCtx{\Gamma}$. If $\types_T \JJ$ and $\Theta; \Gamma$ is suitable for $\JJ$, then $\Theta; \Gamma \types_{\trantt{T}} \erase{\JJ}$.
  \item\label{it:cf-to-tt-3} With $T$, $\Theta; \Gamma$ as in~\eqref{it:cf-to-tt-2}, if $\types_T \BB$ and $\Theta; \Gamma$ is suitable for~$\BB$ then $\Theta; \Gamma \types_{\trantt{T}} \erase{\BB}$.
  \end{enumerate}
\end{restatable}

\begin{proof}
  See the proof on \cpageref{proof:cf-to-tt-bdry-jdg}.
\end{proof}

With the theorem in hand, the loose ends are easily tied up.

\begin{corollary}
  \label{thm:cf-to-tt-thy}%
  The translation of a standard cf-theory is a standard tt-theory.
\end{corollary}

\begin{proof}
  The translation takes symbol rules to symbol rules, and equality rules to equality rules.
\end{proof}

\begin{corollary}
  If a finitary cf-theory $T$ derives $\types_T \JJ$ and~$\JJ$ has well-typed annotations then there exists a context $\Theta; \Gamma$ which is suitable for~$\JJ$ such that $\types_{\trantt{T}} \isMCtx{\Theta}$ and $\Theta \types_{\trantt{T}} \isVCtx{\Gamma}$.
\end{corollary}

\begin{proof}
  We may use the suitable context $\Theta; \Gamma$ with~$\Theta$ and~$\Gamma$ constructed respectively from~$\mv{\JJ}$ and~$\fv{\JJ}$ as in \cref{lem:suitable-ctx}.
\end{proof}

\subsection{Translation from tt-theories to cf-theories}
\label{sec:tt-to-cf-translation}

Transformation from tt-theories to cf-theories requires annotation of variables with typing information, insertion of conversions, and reconstruction of assumption sets.
Unlike in the previous section, we cannot directly translate judgements, but must look at derivations in order to tell where conversions should be inserted and what assumption sets used.
We begin by defining auxiliary notions that help organize the translation.

Given a cf-expression~$e$, let $\erasex{e}$ be the \defemph{double erasure} of~$e$, which is like erasure $\erase{e}$, except that we also remove annotations: $\erasex{\symM^\BB} = \symM$ and $\erasex{\avar{a}{A}} = \mkvar{a}$.
The following definition specifies when an assignment of annotations to variables, which we call a \emph{labeling}, meets the syntactic criteria that makes it eligible for a translation.

\begin{definition}
  \parbox{0pt}{}
  \begin{enumerate}
  \item Consider a metavariable context
    \begin{equation*}
      \Theta = [\symM_1 \of \BB_1, \ldots, \symM_m \of \BB_m].
    \end{equation*}
    An \defemph{eligible labeling for $\Theta$} is a map
    \begin{equation*}
      \theta = \finmap{\symM_1 \mto \BB'_1, \ldots, \symM_m \mto \BB'_m}
    \end{equation*}
    which assigns to each $\symM_i$ a cf-boundary $\BB'_i$ such that $\erasex{\BB'_i} = \BB_i$, and if $\symM_j^{\BB} \in \mv{\BB'_i}$ then $\BB = \BB'_j$.

  \item
    With $\Theta$ and $\theta$ as above, consider a variable context
    \begin{equation*}
      \Gamma = [\mkvar{a}_1 \of A_1, \ldots, \mkvar{a}_n \of A_n],
    \end{equation*}
    over~$\Theta$. An \defemph{eligible labeling for $\Gamma$} with respect to $\theta$ is a map
    \begin{equation*}
      \gamma = \finmap{\mkvar{a}_1 \mto A'_1, \ldots, \mkvar{a}_n \mto A'_n}
    \end{equation*}
    which assigns to each $\mkvar{a}_i$ a cf-type $A'_i$ such that $\erasex{A'_i} = A_i$, if $\symM_j^{\BB} \in \mv{A_i}$ then $\BB = \theta(\symM_j)$, and if $\mkvar{a}_k^A \in \fv{A_i}$ then $A = \gamma(\mkvar{a}_k)$.

  \item A pair $(\theta, \gamma)$ is an \defemph{eligible labeling for $\Gamma; \Theta$} when~$\theta$ is eligible for~$\Theta$ and $\gamma$ is eligible for~$\Gamma$ with respect to $\theta$.

  \item
    With $(\theta, \gamma)$ eligible for $\Theta; \Gamma$,
    an \defemph{eligible cf-judgement} $\JJ'$ for a tt-judgement $\JJ$ over $\Theta; \Gamma$ is one that
    satisfies $\erasex{\JJ'} = \JJ$, if $\symM_i^\BB \in \mv{\JJ'}$ then $\BB = \theta(\symM_i)$, and if $\mkvar{a}_k^A \in \fv{\JJ'}$ then $A = \gamma(\mkvar{a}_k)$.

  \item
    With $(\theta, \gamma)$ eligible for $\Theta; \Gamma$,
    an \defemph{eligible cf-boundary} $\BB'$ for a tt-boundary $\BB$ over $\Theta; \Gamma$ is one that satisfies $\erasex{\BB'} = \BB$, if $\symM_i^{\BB''} \in \mv{\BB'}$ then $\BB'' = \theta(\symM_i)$, and if $\mkvar{a}_k^A \in \fv{\BB'}$ then $A = \gamma(\mkvar{a}_k)$.
  \end{enumerate}
\end{definition}

\noindent
We also postulate eligibility requirements for raw rules and theories.

\begin{definition}
  Consider a raw tt-rule
  \begin{equation*}
    R = (\rawRule{\symM_1 \of \BB_1, \ldots, \symM_n \of \BB_n}{\J}).
  \end{equation*}
  An \defemph{eligible raw cf-rule} for~$R$ is a raw cf-rule
  \begin{equation*}
    R' = (\rawRule{\symM_1^{\BB'_1}, \ldots, \symM_n^{\BB'_n}}{\J'})
  \end{equation*}
  such that $\theta = \finmap{\symM_1 \mto \BB'_1, \ldots, \symM_n \mto \BB'_n}$ is
  eligible for $[\symM_1 \of \BB_1, \ldots, \symM_n \of \BB_n]$, and $\J'$ is eligible for $\J$ with respect to~$\theta$ (and the empty labeling for $\emptyCtx$).

  Let $T = \finmap{R_i}_{i \in I}$ be a raw tt-theory over $\Sigma$.
  An \defemph{eligible raw cf-theory} for~$T$ is a raw cf-theory $T' = \finmap{R'_i}_{i \in I}$ over $\Sigma$ such that each $R'_i$ is eligible for $R_i$.
\end{definition}

\begin{restatable}[Translation of standard tt- to cf-theories]{theorem}{restatetttocf}
  \label{thm:tt-to-cf}
  \parbox{0pt}{}
  \begin{enumerate}
  \item
    \label{item:tt-cf-theory}%
    For any standard tt-theory $T$ there exists a standard cf-theory $T'$ eligible for~$T$.

  \item
    \label{item:tt-cf-extension}%
    For any $T$, $T'$ as above,
    if $\types_T \isMCtx{\Theta}$ then there exists an eligible labeling~$\theta$ for $\Theta$ such that $\types_{T'} \theta(\symM)$ for every $\symM \in \vert{}\Theta\vert{}$.

  \item
    \label{item:tt-cf-context}%
    For any $T$, $T'$, $\Theta$, $\theta$ as above,
    if $\Theta; \emptyCtx \types_T \isVCtx \Gamma$ then there exists an eligible labeling~$\gamma$ for~$\Gamma$ with respect to $\theta$ such that $\types_{T'} \isType{\gamma(\mkvar{a})}$ for every $\mkvar{a} \in \vert{}\Gamma\vert{}$.

  \item
    \label{item:tt-cf-boundary}%
    For any $T$, $T'$, $\Theta$, $\theta$, $\Gamma$, $\gamma$ as above,
    if $\Theta; \Gamma \types_T \BB$ then there exists an eligible cf-boundary $\BB'$ for~$\BB$ with respect to $\theta$, $\gamma$ such that $\types_{T'} \BB'$.

  \item
    \label{item:tt-cf-judgement}%
    For any $T$, $T'$, $\Theta$, $\theta$, $\Gamma$, $\gamma$, as above,
    if $\Theta; \Gamma \types_T \JJ$ then there exists an eligible cf-judgement $\JJ'$ for~$\JJ$ with respect to $\theta$, $\gamma$ such that $\types_{T'} \JJ'$.

  \end{enumerate}
\end{restatable}

\begin{proof}
  See the proof on \cpageref{proof:tt-to-cf}.
\end{proof}

\subsection{Transporting meta-theorems across the correspondence}
\label{sec:lift-thms}

In \cref{sec:context-free-meta-theorems} we proved enough meta-theorems about cf-theories to secure the translations between cf- and tt-theories. We may now take advantage of the translations by transporting meta-theorems about tt-theories to their cf-counterparts.
We illustrate the technique by proving the cf-counterpart of \cref{thm:admissibility-equality-instantiation}, which states that judgementally equal derivations act equally on judgements, and by formulating the economic congrurence cf-rules.

\begin{proposition}
  \label{prop:cf-admissibility-equality-instantiation-forward}
  In a standard cf-theory, consider derivable instantiations
  \begin{equation*}
    I = \finmap{\symM^{\BB_1}_1 \mto f_1, \ldots, \symM^{\BB_n}_n \mto f_n}
    \quad\text{and}\quad
    J = \finmap{\symM^{\BB_1}_1 \mto g_1, \ldots, \symM^{\BB_n}_n \mto g_n}
  \end{equation*}
  such that $\types \BB_i$ for each $i = 1, \ldots, n$, as well as
  \begin{equation}
    \label{eq:cf-eq-inst-I-J}%
    \types \plug{(\upact{I}{i} \BB_i)}{f_i \equiv g'_i \by \alpha_i}
    \qquad\text{and}\qquad
    \erase{g'_i} = \erase{g_i}.
  \end{equation}
  If an object cf-judgement~$\plug{\BB}{e}$ has well-typed annotations and is derivable then there is a derivable equality $\plug{\BB'}{e_I \equiv e_J \by \beta}$ such that $\beta \subseteq \asm{\JJ, \vec{f}, \vec{g}, \vec{g}', \vec{\alpha}}$, $\erase{\BB'} = \erase{\act{I} \BB}$, $\erase{e_I} = \erase{\act{I} e}$ and $\erase{e_J} = \erase{\act{J} e}$.
\end{proposition}

\begin{proof}
  Let $\Theta; \Gamma$ be a context which is suitable for both~\eqref{eq:cf-eq-inst-I-J} and $\plug{\BB}{e}$, and is minimal in the sense that any variable appearing in it also appears in \eqref{eq:cf-eq-inst-I-J} or~$\plug{\BB}{e}$.
  Let $\Xi = \finmap{\symM^{\BB_1}_1 : \erase{\BB_1}, \ldots, \symM^{\BB_n}_n : \erase{\BB_n}}$.
  By \cref{thm:tt-to-cf}, erasure yields judgementally equal derivable tt-instaniations $\erase{I}$ and $\erase{J}$
  of $\Xi$ over $\Theta; \Gamma$, and a derivable judgement $\Theta; \Gamma \vdash \erase{\plug{\BB}{e}}$.
  By \Cref{thm:admissibility-equality-instantiation}, the tt-equality
  \begin{equation*}
    \Theta ; \Gamma \types
       \plug{\erase{\act{I} \BB}}
            {\erase{\act{I} e} \equiv \erase{\act{J} e}}
  \end{equation*}
  is derivable. We apply the renaming $\symM_i^{\BB_i} \mapsto \symM_i$ and $\avar{a}{A_i}_i \mto \mkvar{a}_i$ to it and
  obtain
  \begin{equation*}
    \Theta ; \Gamma \types
       \plug{\erasex{\act{I} \BB}}
            {\erasex{\act{I} e} \equiv \erasex{\act{J} e}}.
  \end{equation*}
  Next, we apply \cref{thm:tt-to-cf} to the above equation with labelings
  $\theta(\symM_i) = \BB_i$ and
  $\gamma(\mkvar{a}_i) = A_i$,
  which results in a derivable cf-equality
  \begin{equation}
    \label{eq:cf-inst-eq-cf}%
    \types \plug{\BB'}{e_I \equiv e_J \by \beta}.
  \end{equation}
  such that $\erase{\BB'} = \erase{\act{I} \BB}$, $\erase{e_I} = \erase{\act{I} e}$ and $\erase{e_J} = \erase{\act{J} e}$.
  Because we required $\Theta; \Gamma$ to be minimal, $\beta$ satisfies the desired constraint.
\end{proof}

The previous proposition gives us a forward-chaining style of congruence rule, because the conclusion is calculated from the premises via the translation theorems. There is also a backward-chaining version in which we proceed from a given (well-formed) cf-equality that we wish to establish.

\begin{corollary}
  \label{cor:cf-admissibility-equality-instantiation-backward}
  In a standard cf-theory, consider derivable instantiation
  \begin{equation*}
    I = \finmap{\symM^{\BB_1}_1 \mto f_1, \ldots, \symM^{\BB_n}_n \mto f_n}
    \quad\text{and}\quad
    J = \finmap{\symM^{\BB_1}_1 \mto g_1, \ldots, \symM^{\BB_n}_n \mto g_n}
  \end{equation*}
  such that $\types \BB_i$ for each $i = 1, \ldots, n$, as well as
  \begin{equation}
    \label{eq:cf-eq-inst-I-J-backward}%
    \types \plug{(\upact{I}{i} \BB_i)}{f_i \equiv g'_i \by \alpha_i}
    \qquad\text{and}\qquad
    \erase{g'_i} = \erase{g_i}.
  \end{equation}
  Suppose $\types \plug{\BB'}{e_I \equiv e_J \by \Box}$ is derivable, where $\erase{\BB'} = \erase{\act{I} \BB}$, $\erase{e_I} = \erase{\act{I} e}$ and $\erase{e_J} = \erase{\act{J} e}$. Then there is $\beta \subseteq \asm{\JJ, \vec{f}, \vec{g}, \vec{g}', \vec{\alpha}}$ such that $\types \plug{\BB'}{e_I \equiv e_J \by \beta}$ is derivable.
\end{corollary}

\begin{proof}
  By \cref{prop:cf-admissibility-equality-instantiation-forward} there is a derivable judgement
  \begin{equation*}
    \types \plug{\BB'}{e'_I \equiv e'_J \by \beta'}
  \end{equation*}
  such that $\erase{\BB''} = \erase{\act{I} \BB}$, $\erase{e'_I} = \erase{\act{I} e}$, $\erase{e'_J} = \erase{\act{J} e}$, and $\beta$ satisfies that required condition. Apply \cref{lem:boundary-convert} to rectify the boundary to the given one.
\end{proof}

The method works on other meta-theorems, too. For example, the backward-chaining cf-variant of economic congruence tt-rules (\cref{def:congruence-rule-eco}) goes as follows.

\begin{proposition}
  \label{prop:context-free-congruence-rule-eco-forward}
  In a standard cf-theory, consider a derivable finitary object rule
  \begin{equation*}
    \rawRule{\symM_1^{\BB_1}, \ldots, \symM_n^{\BB_n}}{\plug{\bdry}{e}}
  \end{equation*}
  and instantiations of its premises
  \begin{equation*}
    I = \finmap{\symM_1^{\BB_1} \mto f_1, \ldots, \symM_n^{\BB_n} \mto f_n},
    \quad\text{and}\quad
    J = \finmap{\symM_1^{\BB_1} \mto g_1, \ldots, \symM_n^{\BB_n} \mto g_n}.
  \end{equation*}
  Suppose the following are derivable:
  \begin{enumerate}
  \item $\types \plug{(\upact{I}{i} \BB_i)}{f_i}$ for each equality boundary $\BB_i$,
  \item $\types \plug{(\upact{I}{i} \BB_i)}{f_i \equiv g'_i \by \alpha_i}$ with $\erase{g'_i} = \erase{g_i}$ for each object boundary $\BB_i$.
  \end{enumerate}
  Suppose $\types \plug{\bdry'}{e_I \equiv e_J \by \Box}$ is derivable, where $\erase{\bdry'} = \erase{\act{I} \bdry}$, $\erase{e_I} = \erase{\act{I} e}$, $\erase{e_j} = \erase{\act{J} e}$. Then there is $\beta \subseteq \asm{\plug{\bdry}{e}, \vec{f}, \vec{g}, \vec{g'}, \vec{\alpha}}$ such that $\types \plug{\bdry'}{e_I \equiv e_J \by \beta}$ is derivable.
\end{proposition}

\begin{proof}
  We proceed much as in the proof of \cref{prop:context-free-congruence-rule-eco-forward,cor:cf-admissibility-equality-instantiation-backward}, except that we apply \cref{def:congruence-rule-eco} on the tt- side.
\end{proof}

\section{Related and future work}
\label{sec:conclusion}

Our investigation into an general metatheory for type theory has lead us to present and study two languages.
In \Cref{sec:finitary-type-theories}, we gave a general definition of a broad class of finitary type theories and proven that it satisfies the expected desirable type theoretic metatheorems.
In \Cref{sec:context-free-tt}, we introduced a context-free formulation of type theories and demonstrated that this definition satisfies further metatheorems that were previously lost, notably strengthening and good inversion principles.
Context-free type theories serve as the theoretical foundation of Andromeda~2, as the annotation discipline for variables and metavariables turned out to be better suited for an effectful meta-language~\cite{pgh:thesis}.
The generality of finitary type theories has been put to work in~\cite{eqchk}, where a general equality checking algorithm is shown to be sound for all standard type theories.

Our work was developed concurrently with several other general frameworks for type theory.
There are different approaches to the study of formal systems such as logics and type theories, ranging from syntactic \cite{Cartmell:Generalised:1978,harper-honsell-plotkin:framework} to semantic \cite{Fiore:Functorial:2014,Isaev:Algebraic:2016,Capriotti:Models:2017} characterisations.
To reasonably delimit the scope of this discussion we shall focus on those that (i) are sufficiently expressive to faithfully represent a wide family of dependent type theories, but (ii) are sufficiently restrictive to prove general meta-theorems that are comparable to ours.

\subsubsection*{General dependent type theories}
The closest relative are general dependent type theories~\cite{gtt}, which we proposed together with Lumsdaine.
Finitary and general dependent type theories (GDTT) have more in common than divides them.
FTT can be seen as a bridge from GDTT to context-free type theories (CFTT).
As context-free type theories in turn are intended as the theoretical underpinning of Andromeda~2, the choice was made to restrict arities of rules and symbols to be finite, which allows for a direct representation as concrete syntax.
This restriction is somewhat coincidental, and we expect that it should be possible to generalise much of the treatment of FTT and possibly CFTT to arbitrary arities.

The treatment of variables and metavariables in FFT differs from that of GDTT in an inessential way: the former usees a locally-nameless discipline and metavariable contexts, while the latter uses shape systems and metavariables as theory extensions. Once again the difference is motivated by implementation details and the rôle metavariables play in proof assistants.

Finally, the levels of well-formedness of the two formalisms differs slighly. GDTT places fewer restrictions on the rules of raw type theories, while a raw FTT already satisfies presuppositivity.

We expect that translations between the finitary fragment of GDTT and FTT can be defined under mild assumptions, and leave their formal comparison as future work.

\subsubsection*{Logical frameworks}
Perhaps the most prominent family of systems for representing logics are logical frameworks \cite{harper-honsell-plotkin:framework,pfenning:logical-frameworks}.
Logical frameworks have spawned a remarkably fruitful line of work \cite{Cervesato:Linear:2002,Watkins:Concurrent:2003,Cousineau:Embedding:2007} and several implementations exist \cite{Pfenning:System:1999,Pientka:Beluga:2010}.
In concurrent work to the development of GDTTs and FTTs, Uemura~\cite{Uemura:General:2019} and Harper~\cite{Harper:Equational:2021} recently proposed frameworks with the purpose of representing type theories.

Both Uemura's LF (\emph{ULF} for short), and Harper's Equational LF (henceforth \emph{EqLF}) extend previous frameworks by the addition of an equality type satisfying reflection to judgemental equality at the framework level, and Uemura includes a substantial development of a general categorical semantics.
Harper's Equational LF \emph{almost} forms a standard finitary type theory. In fact, only inessential modifications are needed to put it in standard form, as is confirmed by a formalisation of EqLF in Andromeda~2~\cite{pgh:thesis}.
We compare both accounts of type theory to FTT along several axes. As they are quite similar, we focus on Uemura's variant.

In one way, ULF is more expressive than FTT.
While FTT allows only one judgement form for types, terms, and their equalities, ULF can also capture theories with other judgement forms, such as the fibrancy judgement of the homotopy type system or two-level type theory~\cite{Voevodsky:HTS:2013,Annenkov:TwoLevel:2019}, or the face formulas of cubical type theory~\cite{Cohen:Cubical:2016}.
While it may be possible to reconstruct some type theories expressible in ULF via the use of universes in FTT, a careful analysis would be required to show that the account is faithful, for instance by showing that it is sound and complete for derivability.
Conversely, every standard finitary type theory is expressible in ULF. The translation is straightforward, and we take this as a sign that both ULF and FTT achieve their goal of giving a ``natural'' account of type theory.

Finitary type theories on the other hand are not directly expressible in ULF or in EqLF.
Frequently, accounts of type theory present rules that are not standard, most often because a symbol does not record all of the metavariables introduced by its premises as arguments.
But it is also standard practice to have only one notation for say dependent products which may occur at more than one sort, as is done in \cite{Martin-Lof:Constructive:1982,Harper:Equational:2021}, or give a general cumulativity rule allowing the silent inclusion of types from one sort into another~\cite{Luo:Extended:1990,Uemura:General:2019}.
One may of course take the view that such presentations are not \emph{really} type theories and should be read with full annotations inserted.
It is usually understood that such an annotated presentation can be given, and by including the right set of equations the original calculus can be recovered~\cite{Harper:Type:1991}.
Proofs that an unannotated theory is equivalent to a fully annotated one are hard labour~\cite[Theorem~4.13]{Streicher:Semantics:1991}.
Finitary type theories can thus serve to study the elaboration of such unannotated to a standard FTT or ULF presentation.
One such useful general result can already be found in~\cite{gtt}, where it is shown that every raw type theory, possibly containing cyclic dependencies between rules, is equivalent to a well-founded one.
The assumption of well-founded stratification is hardwired in ULF through the definition of a signature and in EqLF trough the inductive construction of a context serving as signature, so that such a theorem could not even be stated in ULF or EqLF.
In ongoing research, Petković Komel is employing finitary type theories to investigate a general elaboration theorem, stating that all finitary type theories can be elaborated to standard ones~\cite{PetkovicKomel:Elaboration:2021}.

It would be useful to prove a general adequacy theorem of Uemura's or Harper's~\cite{Harper:Equational:2021} logical framework for finitary type theories.
Conversely, the extension of finitary and context-free type theories to other judgement forms in the style of Uemura's LF seems within reach and would allow the expression of exciting new type theories such as those based on cubical sets~\cite{Cohen:Cubical:2016,Angiuli:Cartesian:2018,Cavallo:Unifying:2020}.
Another active domain of current research are modal type theories~\cite{Schreiber:Quantum:2014,Birkedal:Multimodal:2021}.
Multimodal type theory does not readily fit into our setup or the framework of Uemura~\cite{Gratzer:Normalization:2021}, and the development of modal finitary type theories is an exciting possibility for further work.

\subsubsection*{Context-free type theories}
Geuvers et al.~\cite{geuvers10} investigated the $\Gamma_\infty$ system, a context-free formulation of pure type systems.
They prove similar metatheorems, including translations from and to traditional pure type systems.
Pure type systems disallow proof-irrelevant rules such as equality reflection. Consequently, the results of \cite{geuvers10} are obtained more straightforwardly and without complications arising from the use of conversion terms and assumption sets.
Like the authors of \cite{geuvers10}, our motivation for avoiding explicit contexts came from implementation considerations.
A previous version of Andromeda implemented a form of extensional type theory with assumption sets \cite{andromeda1}.
The results of \cite{geuvers10} have been formalised in the Coq proof assistant.
A formalisation of context-free type theories could serve as trusted nucleus of a future version of Andromeda.
Generalisations of finitary type theories to more general judgement forms in the style of~\cite{Uemura:General:2019} should be mirrored by the development of the corresponding context-free notions and eventually implemented in Andromeda.

\bibliographystyle{plain}
\bibliography{finitary}

\appendix

\section{Proofs of statements}
\label{sec:proofs-statements}

We provide here without further comment the rather technical detailed proofs that were elided in the main text.

\subsection{Proofs of meta-theorems about type theories}
\label{sec:proofs-from-meta-theorems}

This section provides missing proofs from \cref{sec:meta-theorems}.

\restatepreparesubst*

\begin{proof}
  \label{proof:prepare-subst}%
  We proceed by induction on the derivation of the judgement. The induction is mutual with the corresponding statement for boundaries, \cref{lem:prepare-subst-bdry}.

  \inCase{TT-Var}
  If the derivation ends with the variable rule for~$\mkvar{a}$ then we apply weakening to $\Theta; \Gamma \types t : A$ to get $\Theta; \Gamma, \Delta[t/\mkvar{a}] \types t : A$.
  For other variables, we apply the variable rule for the same variable.

  \inCase{TT-Abstr}
  Consider a derivation which ends with an abstraction
  \begin{equation*}
    \infer
    {
      \Theta; \Gamma, \mkvar{a} \of A, \Delta \types \isType B \\
      \mkvar{b} \not\in \vert\Gamma, \mkvar{a} \of A, \Delta\vert \\
      \Theta; \Gamma, \mkvar{a} \of A, \Delta, \mkvar{b} \of B \types \JJ[\mkvar{b}/x]
    }{
      \Theta; \Gamma, \mkvar{a} \of A, \Delta \types \abstr{x \of B} \; \JJ
    }
  \end{equation*}
  The induction hypotheses for the premises yield
  \begin{equation*}
    \Theta; \Gamma, \Delta[t/\mkvar{a}] \types \isType {B[t/\mkvar{a}]}
    \quad\text{and}\quad
      \Theta; \Gamma, \Delta[t/\mkvar{a}], \mkvar{b} \of B[t/\mkvar{a}] \types (\JJ[\mkvar{b}/x])[t/\mkvar{a}].
  \end{equation*}
  Note that $(\JJ[\mkvar{b}/x])[t/\mkvar{a}] = (\JJ[t/\mkvar{a}])[\mkvar{b}/x]$, because $x \notin \bv{t}$ as $t$ is closed, and $\mkvar a \not= \mkvar b$ by assumption. Hence abstracting $\mkvar{b}$ in the second premise yields
  $$
    \Theta; \Gamma, \Delta[t/\mkvar{a}] \types \abstr{x \of B[t/\mkvar{a}]} \; \JJ[t/\mkvar{a}],
  $$ as desired.

  \inCaseText{\rref{TT-Meta} and \rref{TT-Meta-Congr}}
  We only consider the congruence rules, as the metavariable rule is treated similarly.
  Consider a derivation which ends with the congruence rule for a metavariable $\symM$ whose boundary is $\Theta(\symM) = \abstr{\vec{x} \of \vec{B}}\; \bdry$:
  \begin{equation*}
    \infer{
     \Theta; \Gamma, \mkvar{a} \of A, \Delta \types s_j : B_j[\upto{\vec{s}}{j}/\upto{\vec{x}}{j}]
      \quad \text{for $j = 1, \ldots, m$}
      \\
     \Theta; \Gamma, \mkvar{a} \of A, \Delta \types t_j : B_j[\upto{\vec{t}}{j}/\upto{\vec{x}}{j}]
      \quad \text{for $j = 1, \ldots, m$}
      \\
     \Theta; \Gamma, \mkvar{a} \of A, \Delta \types s_j \equiv t_j : B_j[\upto{\vec{s}}{j}/\upto{\vec{x}}{j}]
      \quad \text{for $j = 1, \ldots, m$}
      \\
      \Theta; \Gamma, \mkvar{a} \of A, \Delta \types
      C[\vec s/\vec x] \equiv C[\vec t/\vec x]
      \quad \text{if $\bdry = \Box : C$}
    }{
      \Theta; \Gamma, \mkvar{a} \of A, \Delta \types
      \plug
        {(\bdry[\vec{s}/\vec{x}])}
        {\symM(\vec{s}) \equiv \symM(\vec{t})}
    }
  \end{equation*}
  We apply the induction hypotheses to the premises, and conclude by
  \rref{TT-Meta-Congr} for~$\symM$, applied to $\vec{s}[\mkvar{a}/x]$ and $\vec{t}[\mkvar{a}/x]$, taking into account that in general $(e[u/x])[v/\mkvar{a}] = (e[v/\mkvar{a}])[u[v/\mkvar{a}]/x]$.

  \inCaseText{of a specific rule}
  Consider a derivation ending with the application of a raw rule
  $R = (\rawRule{\symM_1 \of \BB_1, \ldots, \symM_n \of \BB_n}{\J})$ with $\J = \plug \bdry e$, instantiated
  by $I = \finmap{\symM_1 \mto e_1, \ldots, \symM_n \mto e_n}$,
  \begin{equation*}
    \infer
    {
      \Theta; \Gamma, \mkvar{a} \of A, \Delta \types
      \plug{(\upact{I}{i} \BB_i)}{e_i}
      \quad \text{for $i = 1, \ldots, n$}
      \\\\
      \Theta; \Gamma, \mkvar a \of A, \Delta \types
      \act I \bdry
    }{
      \Theta; \Gamma, \mkvar{a} \of A, \Delta \types \act{I} \J
    }
  \end{equation*}
  The induction hypotheses for the premises yield, for $i = 1, \ldots, n$,
  \begin{equation*}
      \Theta; \Gamma, \Delta[t/\mkvar{a}] \types
      (\plug{(\upact{I}{i} \BB_i)}{e_i})[t/\mkvar{a}],
  \end{equation*}
  which equals
  \begin{equation*}
      \Theta; \Gamma, \Delta[t/\mkvar{a}] \types
      \plug{(\upact{I[t/\mkvar{a}]}{i} \BB_i)}{e_i[t/\mkvar{a}]}.
  \end{equation*}
  By \cref{lem:prepare-subst-bdry}, we further obtain $\Theta; \Gamma, \Delta \types \act {(I[t/\mkvar a])} \bdry$.
  Now apply $R$ instantiated at
  $I[t/a] = \finmap{\symM_1 \mto e_1[t/\mkvar{a}], \ldots, \symM_n \mto e_n[t/\mkvar{a}]}$
  to derive
  $
    \Theta; \Gamma, \Delta[t/\mkvar{a}] \types
    \act{I[t/a]} \J
  $,
  which equals
  $
    \Theta; \Gamma, \Delta[t/\mkvar{a}] \types
    (\act{I} \J)[t/\mkvar{a}]
  $.

  \inCaseText{of a congruence rule}
  Apply the induction hypotheses to the premises and conclude by the same rule.

  \inCasesText{
  \rref{TT-EqTy-Refl},
  \rref{TT-EqTy-Sym},
  \rref{TT-EqTy-Trans},
  \rref{TT-EqTm-Refl},
  \rref{TT-EqTm-Sym},
  \rref{TT-EqTm-Trans},
  \rref{TT-Conv-Tm},
  \rref{TT-Conv-EqTm}}
  These cases are dispensed with, once again, by straightforward applications of the induction hypotheses.
\end{proof}

\restatettsubstandttconvabstr*

\begin{proof}
  \label{proof:tt-subst-and-tt-conv-abstr}%
  Suppose the premises of \rref{TT-Subst} are derivable. By inversion the first premise is derived by an application of \rref{TT-Abstr}, therefore for some $\mkvar{a} \not\in \vert{}\Gamma\vert{}$, we can derive
  $
    \Theta; \Gamma, \mkvar{a} \of A \types \JJ[\mkvar{a}/x]
  $.
  \Cref{lem:prepare-subst} yields $\Theta; \Gamma \types (\JJ[\mkvar{a}/x])[t/\mkvar{a}]$, which is equal to the conclusion of \rref{TT-Subst}.

  The rule \rref{TT-Bdry-Subst} follows from \cref{lem:prepare-subst-bdry}.

  Next, assuming the premises of \rref{TT-Conv-Abstr} are derivable, its conclusion is derived as
  {\footnotesize
    \begin{mathpar}\mprset{sep=1em}
      \infRR{}
      { \Theta; \Gamma \types \isType B \\
        \infRR{TT-Subst}
        { \infRL{}
          {\Theta; \Gamma \types \abstr{x \of A}\; \JJ}
          {\Theta; \Gamma, \mkvar a \of B \types \abstr{x \of A}\; \JJ} \\
          \infRR{}
          { \infRL{}{ }{\Theta; \Gamma, \mkvar a \of B \types \mkvar a \of B} \\
            \infRR{}
            { \infRR{}
              { \Theta; \Gamma \types A \equiv B }
              { \Theta; \Gamma \types B \equiv A } }
            { \Theta; \Gamma, \mkvar a \of B \types B \equiv A } }
          { \Theta; \Gamma, \mkvar a \of B \types \mkvar a \of A } }
        { \Theta; \Gamma, \mkvar a \of B \types \JJ[\mkvar a/x] }
      }
      { \Theta; \Gamma \types \abstr{x \of B}\; \JJ }
\ifjar
\phantom{TT-Subst}
    \end{mathpar}
    {\normalsize \qedhere}
\else
      {\normalsize \qedhere}
    \end{mathpar}
\fi
  }
\end{proof}

\restatepreparesubsteq*

\begin{proof}
  \label{proof:prepare-subst-eq}%
  We proceed by induction on the derivation of~\eqref{eq:prep-2}.

  \inCase{TT-Var}
  For a variable $\mkvar{b} \in \vert{}\Gamma\vert{}$, \eqref{it:prep-1} and~\eqref{it:prep-2} follow by the same variable rule, while \eqref{it:prep-3} follows by reflexivity for~$\mkvar{b}$ and the same variable rule.

  For the variable $\mkvar{a}$, the desired judgements are precisely the assumptions \eqref{eq:prep-3}, \eqref{eq:prep-4}, and \eqref{eq:prep-5} weakened to $\Gamma, \Delta[s/\mkvar{a}]$.

  For a variable $\mkvar b \in \vert{}\Delta\vert{}$ with $B = \Delta(\mkvar b)$, the same variable rule derives $\Theta; \Delta[s/\mkvar{a}] \types \mkvar b : B[s/\mkvar a]$ to satisfy \eqref{it:prep-1}, while \eqref{it:prep-2} requires an additional conversion along
  \begin{equation}
    \label{eq:subst-eq-var-convert}%
    \Theta; \Gamma, \Delta[s/\mkvar{a}] \types B[s/\mkvar a] \equiv B[t/\mkvar a]
  \end{equation}
  which is just \eqref{eq:prep-1}.
  To show \eqref{it:prep-3}, namely $\Theta; \Gamma, \Delta[s/\mkvar{a}] \types \mkvar b \equiv \mkvar b : B[s/\mkvar a]$, we use \rref{TT-EqTm-Refl} and the variable rule.

  \inCase{TT-Abstr}
  Consider a derivation ending with an abstraction
  \begin{equation*}
    \infer{
      \Theta; \Gamma, \mkvar a \of A, \Delta \types \isType B \\
      \mkvar b \not\in \vert{}\Gamma, \mkvar a \of A, \Delta\vert{} \\
      \Theta; \Gamma, \mkvar a \of A, \Delta, \mkvar b \of B \types \JJ[\mkvar b/x]
    }{
      \Theta; \Gamma, \mkvar a \of A, \Delta \types \abstr{x \of B} \; \JJ
    }
  \end{equation*}
  The induction hypothesis \eqref{it:prep-1} applied to the first premise yields
  \begin{align}
    \label{eq:subst-eq-abstr-first-1}%
    &\Theta; \Gamma, \Delta[s/\mkvar{a}] \types \isType{B[s/\mkvar a]}, \\
    \label{eq:subst-eq-abstr-first-2}%
    &\Theta; \Gamma, \Delta[s/\mkvar{a}] \types B[s/\mkvar a] \equiv B[t/\mkvar{a}].
  \end{align}
  Equation~\eqref{eq:subst-eq-abstr-first-2} ensures that the extended variable context $\Delta, \mkvar b \of B$ satisfies~\eqref{eq:prep-1}, hence we may use
  the induction hypothesis \eqref{it:prep-1} for the last premise to show
  \begin{equation*}
    \Theta; \Gamma, \Delta[s/\mkvar{a}], \mkvar b \of B[s/\mkvar a] \types \JJ[\mkvar b/x][s/\mkvar a],
  \end{equation*}
  which equals
  \begin{equation}
    \label{eq:subst-eq-abstr-last-1}%
    \Theta; \Gamma, \Delta[s/\mkvar{a}], \mkvar b \of B[s/\mkvar a] \types \JJ[s/\mkvar a][\mkvar b/x].
  \end{equation}
  We can thus use the abstraction rule with \eqref{eq:subst-eq-abstr-first-1} and \eqref{eq:subst-eq-abstr-last-1} to derive $\Theta; \Gamma, \Delta[s/\mkvar{a}] \types \abstr{x \of B[s/\mkvar{a}]}\; \JJ[s/\mkvar a]$, as required.

  The derivation of $\Theta; \Gamma, \Delta[s/\mkvar{a}] \types \abstr{x \of B[t/\mkvar{a}]}\; \JJ[t/\mkvar a]$ is more interesting. We first apply induction hypothesis~\eqref{it:prep-2} to the last premise and get
  \begin{equation*}
    \Theta; \Gamma, \Delta[s/\mkvar{a}], \mkvar{b} \of B[s/\mkvar{a}] \types
    \JJ[\mkvar{b}/x][t/\mkvar{a}].
  \end{equation*}
  Abstraction now gets us to $\Theta; \Gamma, \Delta[s/\mkvar{a}] \types \abstr{x \of B[s/\mkvar{a}]} \; \JJ[t/\mkvar a]$, after which we apply \rref{TT-Conv-Abstr} from \cref{lem:tt-subst-and-tt-conv-abstr} to replace $B[s/\mkvar{a}]$ with $B[t/\mkvar{a}]$
  using~\eqref{eq:subst-eq-abstr-first-2}.

  Lastly, we use the induction hypothesis \eqref{it:prep-3} for the last premise to derive
  \begin{equation*}
    \Theta; \Gamma, \Delta[s/\mkvar{a}], \mkvar b \of B[s/\mkvar a] \types
    (\JJ[\mkvar b/x])[(s \equiv t)/\mkvar a],
  \end{equation*}
  which equals
  \begin{equation}
    \label{eq:subst-eq-abstr-last-3}%
    \Theta; \Gamma, \Delta[s/\mkvar{a}], \mkvar b \of B[s/\mkvar a] \types
    (\JJ[(s \equiv t)/\mkvar a])[\mkvar b/x].
  \end{equation}
  We may thus apply abstraction to \eqref{eq:subst-eq-abstr-first-1} and \eqref{eq:subst-eq-abstr-last-3} to derive
  \begin{equation*}
    \Theta; \Gamma, \Delta[s/\mkvar{a}] \types
    \abstr{x \of B[s/\mkvar{a}]} \; \JJ[(s \equiv t)/\mkvar{a}],
  \end{equation*}
  as desired.

  \inCase{TT-Meta}
  Suppose \eqref{eq:prep-2} concludes with the metavariable rule for $\symM$,
  where $\Theta(\symM) = \BB = (\abstr{x_1 \of A_1} \cdots \abstr{x_n \of A_n}\; \bdry)$:
  \begin{equation}
    \label{eq:prep-mv}%
    \infer
    {{\begin{aligned}
      &\Theta; \Gamma, \mkvar a \of A, \Delta \types u_i : A_i[\upto{\vec{u}}{i}/\upto{\vec{x}}{i}]
      &&\text{for $i = 1, \ldots, n$}
      \\
      &\Theta; \Gamma, \mkvar{a} \of A, \Delta \types \bdry[\vec{u}/\vec{x}]
     \end{aligned}}
    }{
      \Theta; \Gamma, \mkvar a \of A, \Delta \types (\plug{(\bdry[\vec{u}/\vec{x}])}{\symM(\vec{u})})
    }
  \end{equation}
  Judgements \eqref{it:prep-1} and \eqref{it:prep-2} are derived by the metavariable rule for~$\symM$, applied to the corresponding induction hypotheses for the premises of~\eqref{eq:prep-mv}.
  We address \eqref{it:prep-3} in case $\bdry = (\Box : B)$, and leave the simpler case $\bdry = (\isType{\Box})$ to the reader. We thus seek a derivation of
  \begin{equation*}
    \Theta; \Gamma, \Delta[s/\mkvar{a}] \types
    (\symM(\vec{u}) : B[\vec{u}/\vec x])[(s \equiv t)/\mkvar a]
  \end{equation*}
  which equals
  \begin{equation*}
    \Theta; \Gamma, \Delta[s/\mkvar{a}] \types
    \symM(\vec{u}[s/\mkvar a])
      \equiv \symM(\vec{u}[t/\mkvar a])
      : B[\vec{u}[s/\mkvar a]/\vec x].
  \end{equation*}
  This is just the conclusion of the congruence rule \rref{TT-Meta-Congr} for~$\symM$, suitably applied so that its term and term equation premises are precisely the induction hypotheses (\ref{it:prep-1},\ref{it:prep-2},\ref{it:prep-3}) for the term premises of~\eqref{eq:prep-mv}, and its type equation premise is obtained by application of the induction hypothesis~\eqref{it:prep-3} to the last premise of~\eqref{eq:prep-mv}.

  \inCase{TT-Meta-Congr}
  If \eqref{eq:prep-2} ends with a congruence rule for an object metavariable~$\symM$
  then both \eqref{it:prep-1} and \eqref{it:prep-2} follow by the same congruence rule, applied to the respective induction hypotheses for the premises.

  \inCaseText{of a specific rule}
  Suppose \eqref{eq:prep-2} ends with an application of the raw rule
  $R = (\rawRule{\symM_1 \of \BB_1, \ldots, \symM_n \of \BB_n}{\J})$
  instantiated with $I = \finmap{\symM_1 \mto e_1, \ldots, \symM_n \mto e_n}$:
  \begin{equation}
    \label{eq:prep-subst-eq-specific}
    \infer
    {{\begin{aligned}
      &\Theta; \Gamma, \mkvar{a} \of A, \Delta \types \plug{(\upact{I}{i} \BB_i)}{e_i}
      &&\text{for $i = 1, \ldots, n$} \\
      &\Theta; \Gamma, \mkvar{a} \of A, \Delta \types \act{I} \bdry
      &&\text{where $\J = \plug{\bdry}{e}$}
     \end{aligned}}
    }{
      \Theta; \Gamma, \mkvar{a} \of A, \Delta \types
      \act{I} \J
    }
  \end{equation}
  We would like to derive
  \begin{align}
    \Theta; \Gamma, \Delta[s/\mkvar{a}] &\types (\act{I} \J)[s/\mkvar a],
    \label{eq:subst-eq-R1}
    \\
    \Theta; \Gamma, \Delta[s/\mkvar{a}] &\types (\act{I} \J)[t/\mkvar a],
    \label{eq:subst-eq-R2}
    \\
    \intertext{and in case $\J$ is an object judgement, also}
    \Theta; \Gamma, \Delta[s/\mkvar{a}] &\types (\act{I} \J)[(s \equiv t)/\mkvar{a}].
    \label{eq:subst-eq-R3}
  \end{align}
  We derive~\eqref{eq:subst-eq-R1} by $\act{(I[s/\mkvar{a}])} R$ where $I[s/\mkvar a] = \langle \symM_1 \mto e_1[s/\mkvar a], \ldots, \symM_n \mto e_n[s/\mkvar a] \rangle$, as its premises are induction hypotheses. Similarly, \eqref{eq:subst-eq-R2} is derived by $\act{(I[t/\mkvar{a}])} R$.
  We consider \eqref{eq:subst-eq-R3} in case $\J = (u : B)$ and leave the simpler case $\J = (\isType B)$ to the reader. We thus need to derive
  \begin{equation}
    \label{eq:subst-eq-R3-tm}%
    \Theta; \Gamma, \Delta[s/\mkvar{a}] \types
    (\act{I} u)[s/\mkvar a] \equiv (\act{I} u)[t/\mkvar a] : (\act{I} B)[s/\mkvar a],
  \end{equation}
  which we do by applying the congruence rule, where $J = I [s/\mkvar{a}]$ and $K = I [t/\mkvar{a}]$,
  \begin{equation*}
    \infer{
      {\begin{aligned}
      &\Theta; \Gamma, \Delta[s/\mkvar{a}] \types \plug{(\upact{J}{i} \BB_i)}{e_i[s/\mkvar{a}]}
        &&\text{for $i = 1, \ldots, n$}\\
      &\Theta; \Gamma, \Delta[s/\mkvar{a}] \types \plug{(\upact{K}{i} \BB_i)}{e_i[t/\mkvar{a}]}
        &&\text{for $i = 1, \ldots, n$}\\
      &\Theta; \Gamma, \Delta[s/\mkvar{a}] \types \plug{(\upact{J}{i} \BB_i)}{e_i[s/\mkvar{a}] \equiv e_i[t/\mkvar{a}]}
        &&\text{for object boundary $\BB_i$} \\
      &\Theta; \Gamma, \Delta[s/\mkvar{a}] \types \act{J} B \equiv \act{K} B
    \end{aligned}}
    }{
      \Theta; \Gamma, \Delta[s/\mkvar{a}] \types \act{J} u \equiv \act{K} u : \act{J} B
    }
  \end{equation*}
  The first three rows of premises are just the induction hypotheses for the first row of premises
  of~\eqref{eq:prep-subst-eq-specific}, and the last one is~\eqref{it:prep-3} for the last premise
  of~\eqref{eq:prep-subst-eq-specific}.

  \inCaseText{of a congruence rule}
  Both \eqref{it:prep-1} and \eqref{it:prep-2} are derived by applying the induction hypotheses to the premises and using the congruence rule.

  \inCase{TT-Conv-Tm}
  Consider a derivation ending with a conversion
  \begin{equation*}
    \infer
    { \Theta; \Gamma, \mkvar{a} \of A, \Delta \types u : B \\
      \Theta; \Gamma, \mkvar{a} \of A, \Delta \types B \equiv C }
    { \Theta; \Gamma, \mkvar{a} \of A, \Delta \types u : C }
  \end{equation*}
  The judgements $\Theta; \Gamma, \Delta[s/\mkvar{a}] \types u[s/\mkvar a] : C[s/\mkvar a]$ and $\Theta; \Gamma, \Delta[s/\mkvar{a}] \types u[t/\mkvar a] : C[t/\mkvar a]$ immediately follow from the induction hypothesis and conversion.
  To derive $\Theta; \Gamma, \Delta[s/\mkvar{a}] \types (u : C)[(s \equiv t)/\mkvar a]$, note that the induction hypothesis \eqref{it:prep-3} for the first premise yields
  \begin{equation*}
    \Theta; \Gamma, \Delta[s/\mkvar{a}] \types u[s/\mkvar{a}] \equiv u[t/\mkvar{a}] : B[s/\mkvar{a}],
  \end{equation*}
  and \eqref{it:prep-1} applied to the second premise
  \begin{equation*}
    \Theta; \Gamma, \Delta[s/\mkvar{a}] \types B[s/\mkvar{a}] \equiv C[s/\mkvar{a}].
  \end{equation*}
  Thus by equality conversion we conclude $\Theta; \Gamma, \Delta[s/\mkvar{a}] \types u[s/\mkvar{a}] \equiv u[t/\mkvar{a}] : C[s/\mkvar{a}]$.

  \inCasesText{
  \rref{TT-EqTy-Refl},
  \rref{TT-EqTy-Sym},
  \rref{TT-EqTy-Trans},
  \rref{TT-EqTm-Refl},
  \rref{TT-EqTm-Sym},
  \rref{TT-EqTm-Trans},
  \rref{TT-Conv-EqTm}}
  These cases are dispensed with by straightforward applications of the induction hypotheses.
\end{proof}

\restatesubsteqsides*

\begin{proof}
  \label{proof:subst-eq-sides}%
  We spell out the proof of the first claim only.
  By substituting $s$ for $x$ in the first assumption we obtain
  \begin{equation*}
    \Theta; \Gamma \types \abstr{\vec{y} \of \vec{B}[s/x]} \; C[s/x] \equiv D[s/x],
  \end{equation*}
  and by applying \rref{TT-Subst-EqTy} to the second assumption
  \begin{equation*}
    \Theta; \Gamma \types \abstr{\vec{y} \of \vec{B}[s/x]} \; D[s/x] \equiv D[t/x].
  \end{equation*}
  These two may be combined to give the desired judgement by unpacking the abstraction,
  applying transitivity, and packing up the abstraction.
\end{proof}

\restateeqsubstn*

\begin{proof}
  \label{proof:eq-subst-n}%
  First, by inversion on the derivation of
  $
  \Theta; \Gamma \types
  \abstr{\vec{x} \of \vec{A}}\;
  \plug{\BB}{e}
  $
  we see that, for $i = 1, \ldots, n$,
  \begin{equation*}
    \Theta; \Gamma \types \abstr{\upto{\vec{x}}{i} \of \upto{\vec{A}}{i}}\; \isType{A_i}.
  \end{equation*}
  Next, we claim that, for all $j = 1, \ldots, i-1$,
  \begin{equation*}
    \Theta; \Gamma \types
    \begin{aligned}[t]
      \abstr{x_j \of A_j[\upto{\vec{s}}{j}/\upto{\vec{x}}{j}]} \cdots \abstr{x_{i-1} \of A_{i-1}[\upto{\vec{s}}{j}/\upto{\vec{x}}{j}]} \; \qquad\qquad& \\
      A_i[\upto{\vec{s}}{j}/\upto{\vec{x}}{j}] \equiv A_i[\upto{\vec{t}}{j}/\upto{\vec{x}}{j}]. &
    \end{aligned}
  \end{equation*}
  Indeed, when $j = 1$ the statement reduces to reflexivity, while an application of \cref{lem:subst-eq-sides} lets us pass from~$j$ to $j+1$. When $j = i$ we obtain
  \begin{equation*}
    \Theta; \Gamma \types
    A_i[\upto{\vec{s}}{i}/\upto{\vec{x}}{i}] \equiv
    A_i[\upto{\vec{t}}{i}/\upto{\vec{x}}{i}],
  \end{equation*}
  and this can be used to show by conversion that
  $
    \Theta; \Gamma \types t_i : A_i[\upto{\vec{s}}{i}/\upto{\vec{x}}{i}]
  $.
  Now the goal can be derived by repeated applications of \cref{lem:subst-eq-sides}.
\end{proof}

\restateinstantiationadmissible*

\begin{proof}
  \label{proof:instantiation-admissible}%
  We proceed by structural induction on the derivation of $\Xi; \Gamma, \Delta \types \JJ$,
  only devoting attention to the metavariable and abstraction rules, as all the other cases are straightforward.

  \inCase{TT-Meta}
  Consider an application of a metavariable rule for~$\symM$ with
  $\Xi(\symM) = (\abstr{x_1 \of A_1} \cdots \abstr{x_m \of A_m}\; \bdry)$ and
  $I(\symM) = \abstr{\vec{x}} e$:
  \begin{equation*}
    \infer
    { { \begin{aligned}
          &\Xi; \Gamma, \Delta \types t_j : A_j[\upto{\vec{t}}{j}/\upto{\vec{x}}{j}]
            \quad \text{for $j = 1, \ldots, m$}\\
          &\Xi; \Gamma, \Delta \types \bdry[\vec t/ \vec x]
        \end{aligned} } }
    {
      \Xi; \Gamma, \Delta \types \plug{(\bdry[\vec{t}/\vec{x}])}{\symM(\vec{t})}
    }
  \end{equation*}
  We need to derive
  \begin{equation}
    \label{eq:instantiate-1}%
    \Theta; \Gamma, \act{I} \Delta \types \plug{((\act{I} \bdry)[\act{I} \vec{t}/\vec{x}])}{e[\act{I} \vec{t}/\vec{x}]}.
  \end{equation}
  By induction hypothesis, for each $j = 1, \ldots, m$,
  \begin{equation*}
    \Theta; \Gamma, \act{I} \Delta \types \act{I} t_j : (\act{I} A_j)[\upto{\act{I} \vec{t}}{j}/\upto{\vec{x}}{j}],
  \end{equation*}
  while derivability of~$I$ at $\symM$ and weakening by $\act{I} \Delta$ yield
  \begin{equation}
    \label{eq:instantiate-2}%
    \Theta; \Gamma, \act{I} \Delta \types \abstr{\vec{x} \of \act{I} \vec{A}} \; \plug{(\act{I} \bdry)}{e}.
  \end{equation}
  We now derive~\eqref{eq:instantiate-1} by repeatedly using \rref{TT-Subst} to substitute $\act{I} t_i$'s for $x_i$'s in~\eqref{eq:instantiate-2}.

  \inCase{TT-Meta-Congr}
  Consider an application of a metavariable congruence rule for~$\symM$ with
  $\Xi(\symM) = (\abstr{x_1 \of A_1} \cdots \abstr{x_m \of A_m}\ \bdry)$ and
  $I(\symM) = \abstr{\vec{x}} e$:
  \begin{equation*}
    \infer
    { { \begin{aligned}
          &\Xi; \Gamma, \Delta \types s_j :
              A_j[\upto{\vec{s}}{j}/\upto{\vec{x}}{j}]
          &\text{for $j = 1, \ldots, m$}
          \\
          &\Xi; \Gamma, \Delta \types t_j :
              A_j[\upto{\vec{t}}{j}/\upto{\vec{x}}{j}]
          &\text{for $j = 1, \ldots, m$}
          \\
          &\Xi; \Gamma, \Delta \types s_j \equiv t_j :
              A_j[\upto{\vec{s}}{j}/\upto{\vec{x}}{j}]
          &\text{for $j = 1, \ldots, m$}
          \\
          &\Xi; \Gamma, \Delta \types C[\vec s/\vec x] \equiv C[\vec t/\vec x]
          &\text{if $\bdry = (\Box : C)$}
        \end{aligned} }
    }{
      \Xi; \Gamma, \Delta \types
      \plug
      {(\bdry[\vec{s}/\vec{x}])}
      {\symM(\vec{s}) \equiv \symM(\vec{t})}
    }
  \end{equation*}
  We need to derive
  \begin{equation*}
    \Theta; \Gamma, \act{I} \Delta \types
      \plug
      {((\act{I} \bdry)[\act{I}\vec{s}/x])}
      {e[\act{I} \vec{s}/\vec{x}] \equiv e[\act{I} \vec{t}/\vec{x}]}.
  \end{equation*}
  Derivability of~$I$ yields
  \begin{equation}
    \label{eq:instantiate-3}%
    \Theta; \Gamma, \act{I} \Delta \types \abstr{\vec{x} \of \act{I} \vec{A}} \; \plug{(\act{I} \bdry)}{e}.
  \end{equation}
  We may apply \cref{lem:eq-subst-n} to \eqref{eq:instantiate-3} with terms $\act{I} \vec{s}$ and $\act{I} \vec{t}$. The preconditions of the lemma are met by the induction hypotheses for the premises.

  \inCase{TT-Abstr}
  Suppose the derivation ends with an abstraction
  \begin{equation*}
    \inferrule
    {
      \Xi; \Gamma, \Delta \types \isType A \\
      \mkvar{a} \not\in \vert{}\Gamma, \Delta\vert{} \\
      \Xi; \Gamma, \Delta, \mkvar{a} \of A \types \JJ[\mkvar{a}/x]
    }{
      \Xi; \Gamma, \Delta \types \abstr{x \of A} \; \JJ
    }
  \end{equation*}
  The induction hypotheses for the premises state
  \begin{equation*}
    \Theta; \Gamma, \act{I} \Delta \types \isType{\act{I} A}
    \qquad\text{and}\qquad
    \Theta; \Gamma, \act{I} \Delta, \mkvar{a} \of \act{I} A \types \act{I} (\JJ[\mkvar{a}/x]).
  \end{equation*}
  Because $\act{I} (\JJ[\mkvar{a}/x]) = (\act{I} \JJ)[\mkvar{a}/x]$ we may abstract~$\mkvar{a}$ to derive
  \begin{equation*}
    \Theta; \Gamma, \act{I} \Delta \types \abstr{x \of \act{I} A} \; \act{I} \JJ.
\ifjar
  \end{equation*}
  \qedhere
\else
    \qedhere
  \end{equation*}
\fi
\end{proof}

\restateeqinstadmit*

\begin{proof}
  \label{proof:eq-inst-admit}%
  Note that \eqref{eq:instEq-jdg-I} already follows from \cref{prop:instantiation-admissible}, so we do not bother to reprove it, but we include the statement because we use it repeatedly.
  We proceed by structural induction on the derivations of $\types \isMCtx{\Xi}$ and $\Xi; \Gamma, \Delta \types \JJ$.

  \inCase{TT-Var}
  Consider a derivation ending with the variable rule
  \begin{equation*}
    \infer{
    }{
      \Xi; \Gamma, \Delta \types \mkvar{a}_i : A_i.
    }
  \end{equation*}
  We derive
  \eqref{eq:instEq-jdg-J}
  by the variable rule, and when $\mkvar{a}_i \in \vert{}\Delta\vert{}$ a subsequent conversion along
  $\Theta; \Gamma, \act{I} \Delta \types \act{I} A_i \equiv \act{J} A_i$.
  The judgement \eqref{eq:instEq-eq} holds by \rref{TT-EqTm-Refl}.

  \inCase{TT-Abstr}
  Consider a derivation ending with an abstraction
  \begin{equation*}
    \infer{
      \Xi; \Gamma, \Delta \types \isType{B} \\
      \mkvar{b} \notin \vert{}\Gamma, \Delta\vert{} \\
      \Xi; \Gamma, \Delta, \mkvar{b} \of B \types \JJ[\mkvar{b} / y]\\
    }{
      \Xi; \Gamma, \Delta \types \abstr{y \of B}\; \JJ
    }
  \end{equation*}
  The induction hypothesis for the first premise yields
  \begin{align}
    \label{eq:instEq-1}
    \Theta; \Gamma, \act{I} \Delta &\types \isType{\act{I} B},\\
    \label{eq:instEq-1.5}
    \Theta; \Gamma, \act{I} \Delta &\types \isType{\act{J} B},\\
    \label{eq:instEq-2}
    \Theta; \Gamma, \act{I} \Delta &\types \act{I} B \equiv \act{J} B.
  \end{align}
  The extended variable context $\Gamma, \Delta, \mkvar{b} \of B$ satisfies the preconditions of the induction hypotheses for the second premise, therefore
  \begin{align}
    \label{eq:instEq-abstr-I}
    \Theta; \Gamma, \act{I} \Delta, \mkvar{b} \of \act{I} B &\types (\act{I} \JJ)[\mkvar{b}/y], \\
    \label{eq:instEq-abstr-J}
    \Theta; \Gamma, \act{I} \Delta, \mkvar{b} \of \act{I} B &\types (\act{J} \JJ)[\mkvar{b}/y], \\
    \label{eq:instEq-abstr-eq}
    \Theta; \Gamma, \act{I} \Delta, \mkvar{b} \of \act{I} B &\types (\act{(I \equiv J)} \JJ)[\mkvar{b}/y],
  \end{align}
  where \eqref{eq:instEq-abstr-eq} is present only when $\JJ$ is an object judgement.
  Now \eqref{eq:instEq-eq} follows by abstraction from \eqref{eq:instEq-1} and \eqref{eq:instEq-abstr-eq}.
  To derive \eqref{eq:instEq-jdg-J}, we first abstract \eqref{eq:instEq-abstr-J} to get
  \begin{equation*}
    \Theta; \Gamma, \act{I} \Delta \types \abstr{y \of \act{I} B} \; \act{J} \JJ
  \end{equation*}
  and then apply \rref{TT-Conv-Abstr} to convert it along \eqref{eq:instEq-2} to derive the desired
  \begin{equation*}
    \Theta; \Gamma, \act{I} \Delta \types \abstr{y \of \act{J} B} \; \act{J} \JJ.
  \end{equation*}

  \inCaseText{of a specific rule}
  Consider a specific rule
  \begin{equation*}
    R =
    (\rawRule
    {\sym{N}_1 \of \BB'_1, \ldots, \sym{N}_m \of \BB'_m}
    {\plug{\bdry}{e}})
  \end{equation*}
  and an instantiation
  $K = \finmap{\sym{N}_1 \mto g_1, \ldots, \sym{N}_m \mto g_m}$.
  Suppose the derivation ends with the instantiation $\act{K} R$:
  \begin{equation}
    \label{eq:instEq-5}
    \infer{
      \Xi; \Gamma, \Delta \types \plug{(\upact{K}{i} \BB'_i)}{g_i}
      \quad \text{for $i = 1, \ldots, m$}
      \\\\
      \Xi; \Gamma, \Delta \types \act K \bdry
    }{
      \Xi; \Gamma, \Delta \types \act{K} (\plug{\bdry}{e})
    }
  \end{equation}
  We derive \eqref{eq:instEq-jdg-J} by~$\act{(\act{J} K)} R$ where
  $\act{J}K = \finmap{\sym{N}_1 \mto \act{J}g_1, \ldots, \sym{N}_m \mto \act{J}g_m}$.
  The resulting premises for $i = 1, \ldots, m$ are precisely the induction hypotheses \eqref{eq:instEq-jdg-J} for the premises of~\eqref{eq:instEq-5}.
  The last premise, $\Theta; \Gamma, \act I \Delta \types \act {(\act J K)} \bdry$, follows by case analysis of $\bdry$ and the same induction hypothesis \eqref{eq:instEq-jdg-J}.
  To establish \eqref{eq:instEq-eq}, we must derive
  \begin{equation*}
    \Theta; \Gamma, \act{I} \Delta \types
    \plug{(\act{(\act{I} K)} \bdry)}{\act{(\act{I} K)} e \equiv \act{(\act{J} K)} e}.
  \end{equation*}
  We do so by an application of the congruence rule associated with~$R$,
  instantiated with $\act I K$ and $\act J K$.
  The resulting closure rule has four sets of premises, all of which are derivable:
  \begin{itemize}
  \item both copies of premises of~$R$ are derivable because they are the induction hypotheses \eqref{eq:instEq-jdg-I}
    and \eqref{eq:instEq-jdg-J} for the premises of~\eqref{eq:instEq-5},
  \item the additional equational premises are derivable because they are the induction hypotheses \eqref{eq:instEq-eq} for the premises of~\eqref{eq:instEq-5}.
  \end{itemize}

  \inCaseText{of a congruence rule}
  Similar to the case of a specific rule. Given a congruence rule with instantiations $L$ and $K$, \eqref{eq:instEq-jdg-J} follows from the same congruence rule with instantiations $\act J L$ and $\act J K$. The premises hold by induction hypothesis~\eqref{eq:instEq-jdg-J}.

  \inCase{TT-Meta}
  Consider a derivation ending with an application of the metavariable rule for~$\symM_i$, where
  $\vec{x} = (x_1, \ldots, x_m)$,
  $\vec{t} = (t_1, \ldots, t_m)$,
  $J(\symM_i) = \abstr{\vec{x}} e$, and
  $\BB_i = \abstr{\vec{x} \of \vec{A}}\; \bdry$,
  \begin{equation}
    \label{eq:instEq-metavar}%
    \infer
    {
      \Xi; \Gamma, \Delta \types t_j : A_j[\upto{\vec{t}}{j}/\upto{\vec{x}}{j}]
      \quad \text{for $j = 1, \ldots, m$}
      \\\\
      \Xi; \Gamma, \Delta \types \bdry[\vec t/\vec x]
    }{
      \Xi; \Gamma, \Delta \types \plug{(\bdry[\vec{t}/\vec{x}])}{\symM_i(\vec{t})}
    }
  \end{equation}
  Because $J$ is derivable we know that
  $
    \Theta; \Gamma \types \abstr{\vec{x} \of \act{J} \vec{A}}\; \plug{(\act{J} \bdry)}{e}
  $.
  For \eqref{eq:instEq-jdg-J}, we derive
  \begin{equation*}
    \Theta; \Gamma, \act{I} \Delta \types
    \plug
    {((\act{J}\bdry)[\act{J} \vec{t}/\vec{x}])}
    {e[\act{J} \vec{t}/\vec{x}]}
  \end{equation*}
  by substituting $\act{J} \vec{t}$ for $\vec{x}$ by repeated applications of \rref{TT-Subst}, which generate premises, for $j = 1, \ldots, m$,
  \begin{equation*}
    \Theta; \Gamma, \act{I} \Delta \types \act{J}t_j : (\act{J}A_j)[\act{J} \upto{\vec{t}}{j}/\upto{\vec{x}}{j}].
  \end{equation*}
  These are precisely the induction hypotheses for the premises of~\eqref{eq:instEq-metavar}.

  It remains to show \eqref{eq:instEq-eq}. Writing $I(\symM_i)$ as $\abstr{x} e'$, we must establish
  \begin{equation*}
    \Theta; \Gamma, \act{I} \Delta \types
    \plug
    {((\act{I}\bdry)[\act{I} \vec{t}/\vec{x}])}
    {e'[\act{I} \vec{t}/\vec{x}] \equiv e[\act{J} \vec{t}/\vec{x}]}.
  \end{equation*}
  Because $I$ and $J$ are judgementally equal, we know that
  \begin{equation*}
    \Theta; \Gamma \types
       \abstr{\vec{x} : \act{I} \vec A}\;
        \plug{(\act{I} \bdry)}
             {e' \equiv e}.
  \end{equation*}
  By substituting $\act I {\vec t}$ for $\vec{x}$ by repeated use of \rref{TT-Subst}, we derive
  \begin{equation}
    \Theta; \Gamma, \act{I} \Delta \types
    \plug
    {((\act{I}\bdry)[\act{I} \vec{t}/\vec{x}])}
    {e'[\act{I} \vec{t}/\vec{x}] \equiv e[\act{I} \vec{t}/\vec{x}]},
  \end{equation}
  where the substitutions generate obligations, for $j = 1, \ldots, m$,
  \begin{equation*}
    \Theta; \Gamma, \act{I} \Delta \types \act{I}t_j : (\act{I}A_j)[\act{I} \upto{\vec{t}}{j}/\upto{\vec{x}}{j}].
  \end{equation*}
  These are precisely the induction hypotheses for the term premises of~\eqref{eq:instEq-metavar}.
  By transitivity it suffices to derive
  \begin{equation}
    \label{eq:instEq-meta-1}%
    \Theta; \Gamma, \act{I} \Delta \types
    \plug
    {((\act{I}\bdry)[\act{I} \vec{t}/\vec{x}])}
    {e[\act{I} \vec{t}/\vec{x}] \equiv e[\act{J} \vec{t}/\vec{x}]}.
  \end{equation}
  The induction hypotheses for the premises of \eqref{eq:instEq-metavar} for $j = 1, \ldots, m$ are
  \begin{align}
    \Theta; \Gamma, \act{I} \Delta &\types \act{I} t_j : (\act{I} A_j)[\upto{\act{I} \vec{t}}{j}/\upto{\vec{x}}{j}] \label{eq:instEq-meta-It-IA}\\
    \Theta; \Gamma, \act{I} \Delta &\types \act{J} t_j : (\act{J} A_j)[\upto{\act{J} \vec{t}}{j}/\upto{\vec{x}}{j}] \label{eq:instEq-meta-Jt-JA}\\
    \Theta; \Gamma, \act{I} \Delta &\types \act{I} t_j \equiv \act{J} t_j : (\act{I} A_j)[\upto{\act{I} \vec{t}}{j}/\upto{\vec{x}}{j}]. \label{eq:instEq-meta-eq-It-Jt}
  \end{align}
  We would like to apply \cref{lem:eq-subst-n} to these judgements to derive \eqref{eq:instEq-meta-1}, but the type of the terms~$\act J t_j$ in \eqref{eq:instEq-meta-Jt-JA} does not match the type of the corresponding terms~$\act I t_j$.
  We rectify the situation by successively deriving the equality of the types involved and converting, as follows.

  By assumption $\types \isMCtx \Xi$ holds and hence $\upto \Xi i; \emptyCtx \types \abstr{x_1 \of A_1} \cdots \abstr{x_{j-1} \of A_{j-1}}\; \isType {A_j}$ for $j = 1, \ldots, m$.
  Note that the preceding judgement is derivable in a smaller metavariable context, and we can thus appeal to the induction hypothesis to derive
  \begin{equation*}
    \Theta; \Gamma, \act I \Delta \types \abstr{x_1 \of \act I A_1} \cdots \abstr{x_{j-1} \of \act I A_{j-1}}\; \act I A_j \equiv \act J A_j.
  \end{equation*}
  We apply \cref{lem:eq-subst-n} together with~(\ref{eq:instEq-meta-It-IA},\ref{eq:instEq-meta-Jt-JA},\ref{eq:instEq-meta-eq-It-Jt}) to obtain
  \begin{equation*}
    \Theta; \Gamma, \act I \Delta \types
    (\act I A_j)[\act I \upto {\vec t} j/\upto {\vec x} j] \equiv
    (\act J A_j)[\act J \upto {\vec t} j/\upto {\vec x} j].
  \end{equation*}
  We now appeal to \rref{TT-Conv-Tm} to derive
  \begin{equation}
    \Theta; \Gamma, \act{I} \Delta
    \types \act J t_j : (\act I A_j)[\upto {\act I \vec t} j/\upto {\vec x} j].\label{eq:instEq-meta-Jt-IA}
  \end{equation}
  Finally we derive \eqref{eq:instEq-meta-1} by applying \cref{lem:eq-subst-n} to~(\ref{eq:instEq-meta-It-IA},\ref{eq:instEq-meta-Jt-IA},\ref{eq:instEq-meta-eq-It-Jt}) and to the judgement $\Theta; \Gamma \types \abstr{\vec{x} \of \act{I} \vec{A}} \plug{(\act{I} \bdry)}{e}$,
  which equals $\Theta; \Gamma \types \plug{(\act{I} \BB_i)}{J(\symM_i)}$ and so is derivable by assumption.

  \inCase{TT-Meta-Congr}
  Consider a derivation ending with an application of the congruence rule for~$\symM_i$,
  where
  $\vec{x} = (x_1, \ldots, x_m)$,
  $\vec{s} = (s_1, \ldots, s_m)$,
  $\vec{t} = (t_1, \ldots, t_m)$,
  $J(\symM_i) = \abstr{\vec{x}} e$, and
  $\BB_i = \abstr{\vec{x} \of \vec{A}}\; \bdry$,
  \begin{equation}
    \label{eq:instEq-6}
    \infer{
      {\begin{aligned}
         \Xi; \Gamma, \Delta &\types s_j : A_j[\upto{\vec{s}}{j}/\upto{\vec{x}}{j}] \
         &\text{for $j = 1, \ldots, m$} \\
         \Xi; \Gamma, \Delta &\types t_j : A_j[\upto{\vec{t}}{j}/\upto{\vec{x}}{j}] \
         &\text{for $j = 1, \ldots, m$} \\
         \Xi; \Gamma, \Delta &\types s_j \equiv t_j : A_j[\upto{\vec{s}}{j}/\upto{\vec{x}}{j}] \
         &\text{for $j = 1, \ldots, m$} \\
         \Xi; \Gamma, \Delta &\types C[\vec s/\vec x] \equiv C[\vec t/\vec x] \
         &\text{if $\bdry = (\Box : C)$}
      \end{aligned}}
    }{
      \Xi; \Gamma, \Delta \types \plug{(\bdry[\vec{s}/\vec{x}])}{\symM_i(\vec{s}) \equiv \symM_i(\vec{t})}
    }
  \end{equation}
  Because $J$ is derivable we know that $\Theta; \Gamma \types \abstr{\vec{x} \of \act{J} \vec{A}}\; \plug{(\act{J} \bdry)}{e}$, therefore by weakening also
  \begin{equation*}
    \Theta; \Gamma, \act{I} \Delta \types \abstr{\vec{x} \of \act{J} \vec{A}}\; \plug{(\act{J}\bdry)}{e}.
  \end{equation*}
  The desired judgement
  \begin{equation*}
    \Theta; \Gamma, \act{I} \Delta \types
    \plug
    {((\act{J}\bdry)[\act{J} \vec{s}/\vec{x}])}
    {e[\act{J} \vec{s}/\vec{x}] \equiv e[\act{J} \vec{t}/\vec{x}]}
  \end{equation*}
  may be derived by repeated applications of \rref{TT-Subst-EqTm},
  provided that, for $j = 1, \ldots, m$,
  \begin{align*}
    \Theta; \Gamma, \act{I} \Delta &\types \act{J}s_j : (\act{J}A_j)[\act{J} \upto{\vec{s}}{j}/\upto{\vec{x}}{j}], \\
    \Theta; \Gamma, \act{I} \Delta &\types \act{J}t_j : (\act{J}A_j)[\act{J} \upto{\vec{t}}{j}/\upto{\vec{x}}{j}], \\
    \Theta; \Gamma, \act{I} \Delta &\types \act{J}s_j \equiv \act{J}t_j : (\act{J}A_j)[\act{J} \upto{\vec{s}}{j}/\upto{\vec{x}}{j}].
  \end{align*}
  These are precisely induction hypotheses for~\eqref{eq:instEq-6}.

  \inCasesText{
    \rref{TT-EqTy-Refl},
    \rref{TT-EqTy-Sym},
    \rref{TT-EqTy-Trans},
    \rref{TT-EqTm-Refl},
    \rref{TT-EqTm-Sym},
    \rref{TT-EqTm-Trans},
    \rref{TT-Conv-Tm}, and
    \rref{TT-Conv-EqTm}
  }
  The remaining cases are all equality rules. Each is established by an appeal to the induction hypotheses for the premises, followed by an application of the same rule.
\end{proof}

\restateinsteqbootstrapctx*

\begin{proof}
  \label{proof:inst-eq-bootstrap-ctx}%
  We proceed by induction on the length of~$\Delta$. The base case is trivial.
  For the induction step, suppose $\Theta \types \isVCtx{(\Gamma, \Delta, \mkvar{b} \of B)}$.
  For $\mkvar{a} \in \vert{}\Delta\vert{}$ we apply the induction hypothesis to~$\Delta$ and weaken by $\mkvar{b} \of \act{I} B$.
  To deal with $\mkvar{b}$,  we apply \Cref{lem:eq-inst-admit} to $\Theta; \Gamma, \Delta \types \isType{B}$, which
  holds by inversion, and weaken by $\mkvar{b} \of \act{I} B$ to derive the desired
  \begin{align*}
    \Theta; \Gamma, \act{I} \Delta, \mkvar{b} \of \act{I} B &\types \isType{\act{I} B}, \\
    \Theta; \Gamma, \act{I} \Delta, \mkvar{b} \of \act{I} B &\types \isType{\act{J} B}, \\
    \Theta; \Gamma, \act{I} \Delta, \mkvar{b} \of \act{I} B &\types \act{I} B \equiv \act{J} B.
\ifjar \else \qedhere \fi
  \end{align*}
\end{proof}

\restateinsteqbootstrapinst*

\begin{proof}
  \label{proof:inst-eq-bootstrap-inst}.
  We proceed by induction on~$n$. The base case is trivial.
  To prove the induction step for $n > 0$, suppose the statement holds for $\upto{\Xi}{n}$, $\upto{I}{n}$ and $\upto{J}{n}$, and that $\BB_n = \abstr{x_1 \of A_1} \cdots \abstr{x_m \of A_m} \; \bdry$.
  By inversion on $\types \isMCtx{\Xi}$ and weakening we derive $\upto{\Xi}{n}; \Gamma \types \BB_n$. Then by inverting the abstractions of~$\BB_n$ we obtain variables $\vec{a} = (\mkvar{a}_1, \ldots, \mkvar{a}_m)$ such that, with $A'_i = A_i[\upto{\vec{a}}{i}/\upto{\vec{x}}{i}]$ and $\Delta = [\mkvar{a}_1 \of A'_1, \ldots, \mkvar{a}_m \of A'_m]$,
  \begin{equation*}
    \upto{\Xi}{n} \types \isVCtx{(\Gamma, \Delta)},
    \qquad\text{and}\qquad
    \upto{\Xi}{n}; \Gamma, \Delta \types \bdry[\vec{a}/\vec{x}].
  \end{equation*}
  We apply \cref{lem:inst-eq-bootstrap-ctx} to $\upto{\Xi}{n}$, $\upto{I}{n}$, $\upto{J}{n}$, and $\Delta$ to derive, for $i = 1, \ldots, m$,
  \begin{align}
    \notag
    \Theta; \Gamma, \act{I} \Delta &\types \isType{\act{I} A'_i}, \\
    \notag
    \Theta; \Gamma, \act{I} \Delta &\types \isType{\act{J} A'_i}, \\
    \notag
    \Theta; \Gamma, \act{I} \Delta &\types \act{I} A'_i \equiv \act{J} A'_i, \\
    \label{eq:bootstrap-1}
    \Theta; \Gamma, \act{I} \Delta &\types \mkvar{a}_i : \act{J} A'_i.
  \end{align}
  where \eqref{eq:bootstrap-1} follows by conversion from the judgement above it.
  Next, we use \eqref{eq:bootstrap-1} to substitute $\mkvar{a}_i$ for $x_i$ in
  $
  \Theta; \Gamma, \act{I} \Delta \types \abstr{\vec{x} \of \act{J} \vec{A}} \; \plug{(\act{J} \bdry)}{f_n}
  $,
  which results in
  \begin{equation}
    \label{eq:bootstrap-2}%
    \Theta; \Gamma, \act{I} \Delta \types
       \plug
       {((\act{J} \bdry)[\vec{a}/\vec{x}])}
       {f_n[\vec{a}/\vec{x}]}.
  \end{equation}
  If we can reduce \eqref{eq:bootstrap-2} to
  \begin{equation}
    \label{eq:bootstrap-3}%
    \Theta; \Gamma, \act{I} \Delta \types
       \plug
       {((\act{I} \bdry)[\vec{a}/\vec{x}])}
       {f_n[\vec{a}/\vec{x}]},
  \end{equation}
  we will be able to derive the desired judgement
  \begin{equation*}
    \Theta; \Gamma, \act{I} \Delta \types \abstr{\vec{x} \of \act{I} \vec{A}} \; \plug{(\act{I} \bdry)}{f_n}
  \end{equation*}
  by abstracting $\mkvar{a}_1, \ldots, \mkvar{a}_n$ in \eqref{eq:bootstrap-3}.
  There are four cases, depending on what~$\bdry$ is.

  \inCaseText{$\bdry = (\isType{\Box})$}
  \eqref{eq:bootstrap-2} and \eqref{eq:bootstrap-3} are the same.

  \inCaseText{$\bdry = (\Box : B)$}
  We convert \eqref{eq:bootstrap-2} along
  \begin{equation*}
    \Theta; \Gamma, \act{I} \Delta \types
    (\act{J} B)[\vec{a}/\vec{x}] \equiv
    (\act{I} B)[\vec{a}/\vec{x}],
  \end{equation*}
  which holds by \cref{lem:eq-inst-admit} applied to 
  $\upto{\Xi}{n}; \Gamma, \Delta \types \isType{B[\vec{a}/\vec{x}]}$
  with $\upto{\Xi}{n}$, $\upto{I}{n}$, and~$\upto{J}{n}$.

  \inCaseText{$\bdry = (B \equiv C \by \Box)$}
  Here \eqref{eq:bootstrap-2} and \eqref{eq:bootstrap-3} are respectively
  \begin{equation*}
    \Theta; \Gamma, \act{I} \Delta \types \act{J} B \equiv \act{J} C
    \quad\text{and}\quad
    \Theta; \Gamma, \act{I} \Delta \types \act{I} B \equiv \act{I} C.
  \end{equation*}
  The latter follows from the former if we can also derive
  \begin{equation}
    \label{eq:bootstrap-4}%
    \Theta; \Gamma, \act{I} \Delta \types \act{I} B \equiv \act{J} B,
    \quad\text{and}\quad
    \Theta; \Gamma, \act{I} \Delta \types \act{I} C \equiv \act{J} C.
  \end{equation}
  We invert $\upto{\Xi}{n}; \Gamma, \Delta \types B \equiv C \by \Box$ to derive
  \begin{equation}
    \label{eq:bootstrap-5}%
    \upto{\Xi}{n}; \Gamma, \Delta \types \isType{B}
    \qquad\text{and}\qquad
    \upto{\Xi}{n}; \Gamma, \Delta \types \isType{C}.
  \end{equation}
  When we apply \cref{lem:eq-inst-admit} to \eqref{eq:bootstrap-5} it gives us \eqref{eq:bootstrap-4}.

  \inCaseText{$\bdry = (s \equiv t : B \by \Box)$}
  Here \eqref{eq:bootstrap-2} and \eqref{eq:bootstrap-3} are respecetively
  \begin{equation*}
    \Theta; \Gamma, \act{I} \Delta \types \act{J} s \equiv \act{J} t : \act{J} B
    \quad\text{and}\quad
    \Theta; \Gamma, \act{I} \Delta \types \act{I} s \equiv \act{I} t : \act{I} B,
  \end{equation*}
  The latter follows from the former if we can also derive
  \begin{equation}
    \label{eq:bootstrap-7}%
    \begin{aligned}
    \Theta; \Gamma, \act{I} \Delta &\types \act{I} B \equiv \act{J} B, \\
    \Theta; \Gamma, \act{I} \Delta &\types \act{I} s \equiv \act{J} s : \act{I} B, \\
    \Theta; \Gamma, \act{I} \Delta &\types \act{I} t \equiv \act{J} t : \act{I} B.
    \end{aligned}
  \end{equation}
  We invert $\upto{\Xi}{n}; \Gamma, \Delta \types s \equiv t : B \by \Box$ to derive
  \begin{equation}
    \label{eq:bootstrap-8}%
    \upto{\Xi}{n}; \Gamma, \Delta \types \isType{B},
    \quad
    \upto{\Xi}{n}; \Gamma, \Delta \types s : B
    \quad\text{and}\quad
    \upto{\Xi}{n}; \Gamma, \Delta \types t : B.
  \end{equation}
  When we apply \cref{lem:eq-inst-admit} to \eqref{eq:bootstrap-8} it gives us \eqref{eq:bootstrap-7}.
\end{proof}

\restatepresuppositivity*

\begin{proof}
  \label{proof:presuppositivity}%
  We proceed by induction on the derivation of $\Theta; \Gamma \types \plug{\BB}{e}$.

  \inCase{TT-Var}
  By \cref{prop:ctx-inversion}.

  \inCase{TT-Meta}
  The presupposition $\Theta; \Gamma \types \bdry[\vec t/\vec x]$ is available as premise.

  \inCase{TT-Meta-Congr}
  Consider a derivation ending with an application of the congruence rule for~$\symM$ whose boundary is $\Theta(\symM) = (\abstr{x_1 \of A_1} \cdots \abstr{x_m \of A_m} \; \bdry)$:
  \begin{equation*}
    \infer
    { { \begin{aligned}
          &\Theta; \Gamma \types s_j :
              A_j[\upto{\vec{s}}{j}/\upto{\vec{x}}{j}]
          &\text{for $j = 1, \ldots, m$}
          \\
          &\Theta; \Gamma \types t_j :
              A_j[\upto{\vec{t}}{j}/\upto{\vec{x}}{j}]
          &\text{for $j = 1, \ldots, m$}
          \\
          &\Theta; \Gamma \types s_j \equiv t_j :
              A_j[\upto{\vec{s}}{j}/\upto{\vec{x}}{j}]
          &\text{for $j = 1, \ldots, m$}
          \\
          &\Theta; \Gamma \types C[\vec{s}/\vec{x}] \equiv C[\vec{t}/\vec{x}]
          &\text{if $\bdry = (\Box : C)$}
        \end{aligned} }
    }{
      \Theta; \Gamma \types
      \plug
      {(\bdry[\vec{s}/\vec{x}])}
      {\symM(\vec{s}) \equiv \symM(\vec{t})}
    }
  \end{equation*}
  If $\bdry = (\isType{\Box})$, the presupposition $\Theta; \Gamma \types \symM(\vec{s}) \equiv \symM(\vec{t}) \by \Box$ follows directly by \rref{TT-Bdry-EqTy} and two uses of \rref{TT-Meta}.
  If $\bdry  = (\Box : C)$, the presuppositions of $\types \symM(\vec{s}) \equiv \symM(\vec{t}) : C[\vec{s}/\vec{x}] \by \Box$ follow by \rref{TT-Bdry-EqTm}:
  \begin{enumerate}
  \item $\Theta; \Gamma \types \isType{C[\vec{s}/\vec{x}]}$ holds by substitution of $\vec{s}$ for $\vec{x}$ in $\Theta; \Gamma \types \abstr{\vec{x} \of \vec{A}} \; \isType C$ much like in the previous case,
  \item $\Theta; \Gamma \types \symM(\vec{s}) : C[\vec{s}/\vec{x}]$ holds by \rref{TT-Meta},
  \item $\Theta; \Gamma \types \symM(\vec{t}) : C[\vec{s}/\vec{x}]$ is derived from
    $\Theta; \Gamma \types \symM(\vec{t}) : C[\vec{t}/\vec{x}]$
    by conversion along $\Theta; \Gamma \types C[\vec{t}/\vec{x}] \equiv C[\vec{s}/\vec{x}]$, which holds by the last premise.
  \end{enumerate}
  When applying \rref{TT-Meta} above, the premise $\Theta; \Gamma \types \bdry[\vec s/\vec x]$ is required, and likewise for~$\vec t$.
  We may derive it by applying \cref{prop:ctx-inversion} to $\types \isMCtx{\Theta}$ and substituting~$\vec{s}$ for~$\vec{x}$ with the help of \rref{TT-Subst}, and analogously for~$\vec{t}$.

  \inCase{TT-Abstr}
  Consider an abstraction
  \begin{equation*}
    \infer
    {
      \Theta; \Gamma \types \isType A \\
      \mkvar{a} \not\in \vert{}\Gamma\vert{} \\
      \Theta; \Gamma, \mkvar{a} \of A \types \JJ[\mkvar{a}/x]
    }{
      \Theta; \Gamma \types \abstr{x \of A} \; \JJ
    }
  \end{equation*}
  By induction hypothesis on the last premise, we obtain
  $\Theta; \Gamma, \mkvar{a} \of A \types \BB[\mkvar{a}/x]$ after which we apply \rref{TT-Bdry-Abstr}.

  \inCaseText{of a specific rule} The presupposition is available as premise.

  \inCaseText{of a congruence rule}
  Consider a congruence rulles associated with an object rule~$R$ and instantiated with~$I$ and~$J$, as in \cref{def:congruence-rule}.

  If $R$ concludes with $\types \isType A$, the presuppositions are $\Theta; \Gamma \types \isType{\act I A}$ and $\Theta; \Gamma \types \isType{\act J A}$, which are derivable by $\act I R$ and $\act J R$, respectively.

  If~$R$ concludes with $\types t : A$, the presuppositions are
  $\Theta; \Gamma \types \isType{\act{I} A}$,
  $\Theta; \Gamma \types \act{I} t : \act{I} A$, and
  $\Theta; \Gamma \types \act{J} t : \act{I} A$.
  We derive the first one by applying the induction hypothesis to the premise $\Theta; \Gamma \types \act I B \equiv \act J B$, the second one by~$\act{I} R$, and the third one by converting the second one along the aforementioned premise.

  \inCasesText{
  \rref{TT-EqTy-Refl},
  \rref{TT-EqTy-Sym},
  \rref{TT-EqTy-Trans},
  \rref{TT-EqTm-Refl},
  \rref{TT-EqTm-Sym},
  \rref{TT-EqTm-Trans}}
  These are all dispensed with by straightforward appeals to the induction hypotheses.

  \inCase{TT-Conv-Tm}
  Consider a term conversion
  \begin{equation*}
    \infer
    { \Theta; \Gamma \types t : A \\
      \Theta; \Gamma \types A \equiv B }
    { \Theta; \Gamma \types t : B }
  \end{equation*}
  Then $\Theta; \Gamma \types \isType{B}$ holds by the induction hypothesis for the second premise.

  \inCase{TT-Conv-EqTm}
  Consider a term equality conversion
  \begin{equation*}
    \infer
    { \Theta; \Gamma \types s \equiv t : A \\
      \Theta; \Gamma \types A \equiv B }
    { \Theta; \Gamma \types s \equiv t : B }
  \end{equation*}
  As in the previous case, the induction hypothesis for the second premise provides $\Theta; \Gamma \types \isType{B}$.
  The induction hypothesis for the first premise yields
  \begin{equation*}
    \Theta; \Gamma \types s : A
    \qquad\text{and}\qquad
    \Theta; \Gamma \types t : A
  \end{equation*}
  We may convert these to
  $\Theta; \Gamma \types s : B$ and $\Theta; \Gamma \types t : B$
  using the second premise.
\end{proof}

\restatettmetaeco*

\begin{proof}
  \label{proof:tt-meta-eco}%
  To prove admissibility of \rref{TT-Meta-Eco}, note that by \cref{prop:ctx-inversion} we have
  $
    \Theta; \Gamma \types \abstr{\vec{x} \of \vec{A}}\; \bdry,
  $
  so we may derive $\Theta; \Gamma \types \bdry[\vec{t}/\vec{x}]$ by substituting $\vec{t}$ for $\vec{x}$ by repeated applications of \rref{TT-Subst} to the premises of~\rref{TT-Meta-Eco}. We can now apply \rref{TT-Meta}.

  Next, we address admissibility of \rref{TT-Meta-Congr-Eco} by deriving its conclusion with the aid of \rref{TT-Meta-Congr}. For this purpose we need to derive
  \begin{align*}
    &\Theta; \Gamma \types s_j :
      A_j[\upto{\vec{s}}{j}/\upto{\vec{x}}{j}]
    &&\text{for $j = 1, \ldots, m$}
    \\
    &\Theta; \Gamma \types t_j :
      A_j[\upto{\vec{t}}{j}/\upto{\vec{x}}{j}]
    &&\text{for $j = 1, \ldots, m$}
    \\
    &\Theta; \Gamma \types C[\vec{s}/\vec{x}] \equiv C[\vec{t}/\vec{x}]
    &&\text{if $\bdry = (\Box : C)$}
  \end{align*}
  The first group follows by \cref{prop:presuppositivity}.
  The second is established by induction on $j$: by \cref{prop:ctx-inversion}, $\Theta \types \abstr{\vec{x} \of \vec{A}}\; \bdry$ holds, and thus $\Theta \types \abstr{\upto{\vec{x}}{j} \of \upto{\vec{A}}{j}} \; \isType {A_j}$.
  By applying \cref{lem:eq-subst-n}, we obtain
  $\Theta; \Gamma
  \types A_j[\upto{\vec{s}}{j}/\upto{\vec{x}}{j}]
  \equiv A_j[\upto{\vec{t}}{j}/\upto{\vec{x}}{j}]$
  and we can convert $\Theta; \Gamma \types t_j : A_j[\upto{\vec{s}}{j}/\upto{\vec{x}}{j}]$
  which holds again by \cref{prop:presuppositivity}.
  Finally, the last premise holds again by \cref{lem:eq-subst-n}, this time applied to $\Theta \types \abstr{\vec{x} \of \vec{A}} \; \isType{C}$.
\end{proof}

\restateequalboundaryconvert*

\begin{proof}
  \label{proof:equal-boundary-convert}%
  We proceed by induction on the derivation of $\Xi; \Gamma \types \BB$.

  \inCase{TT-Bdry-Ty}
  We have $\BB = (\isType{\Box})$, the statement is trivial.

  \inCase{TT-Bdry-Tm}
  We have $\BB = (\Box : A)$. From $\Xi; \Gamma \types \isType{A}$ we obtain $\Theta; \Gamma \types \act{I} A \equiv \act{J} A$ using \cref{thm:admissibility-equality-instantiation}, and convert $\Theta; \Gamma \types e : \act{I} A$
  to $\Theta; \Gamma \types e : \act{J} A$.

  \inCase{TT-Bdry-EqTy}
  We have $\BB = (A \equiv B \by \Box)$. From \cref{prop:presuppositivity} we get $\Xi; \Gamma \types \isType{A}$, hence $\Theta; \Gamma \types \act{I} A \equiv \act{J} A$ by \cref{thm:admissibility-equality-instantiation}.
  It follows similarly that $\Theta; \Gamma \types \act{I} B \equiv \act{J} B$. We may combine these with $\Theta; \Gamma \types \act{I} A \equiv \act{I} B$ using transitivity to derive $\Theta; \Gamma \types \act{J} A \equiv \act{J} B$.

  \inCase{TT-Bdry-EqTm}
  We have $\BB = (s \equiv t : A)$. From \cref{prop:presuppositivity} we get
  \begin{equation*}
    \Xi; \Gamma \types \isType{A},
    \qquad
    \Xi; \Gamma \types s : A,
    \quad\text{and}\quad
    \Xi; \Gamma \types t : A.
  \end{equation*}
  and from \cref{thm:admissibility-equality-instantiation}
  \begin{equation*}
    \Xi; \Gamma \types \act{I} A \equiv \act{J} A,
    \qquad
    \Xi; \Gamma \types \act{I} s \equiv \act{J} s : \act{I} A,
    \quad\text{and}\quad
    \Xi; \Gamma \types \act{I} t \equiv \act{J} t : \act{I} A.
  \end{equation*}
  Together with $\Xi; \Gamma \types \act{I} s \equiv \act{I} t : \act{I} A$ this is sufficient to derive
  $\Xi; \Gamma \types \act{J} s \equiv \act{J} t : \act{J} A$ using transitivity and conversions.

  \inCase{TT-Bdry-Abstr}
  We have $\BB = (\abstr{x \of A} \; \BB')$ and $\Xi; \Gamma \types \abstr{x \of \act{I} A} \; \plug{(\act{I} \BB')}{e}$. We use induction hypothesis and abstraction to derive $\Xi; \Gamma \types \abstr{x \of \act{I} A} \; \plug{(\act{J} \BB')}{e}$ and then convert the abstracion to $\act{J} A$ using \rref{TT-Conv-Abstr}.
\end{proof}

\restatecongruenceruleeco*

\begin{proof}
  \label{proof:congruence-rule-eco}%
  We appeal to the congruence rule for~$R$,
  \begin{equation*}
    \infer{
      {\begin{aligned}
      &\Theta; \Gamma \types \plug{(\upact{I}{i} \BB_i)}{f_i}  &&\text{for $i = 1, \ldots, n$}\\
      &\Theta; \Gamma \types \plug{(\upact{J}{i} \BB_i)}{g_i}  &&\text{for $i = 1, \ldots, n$}\\
      &\Theta; \Gamma \types \plug{(\upact{I}{i} \BB_i)}{f_i \equiv g_i} &&\text{for object boundary $\BB_i$} \\
      &\Theta; \Gamma \types \act{I} B \equiv \act{J} B        &&\text{if $\bdry = (\Box : B)$}
    \end{aligned}}
    }{
      \Theta; \Gamma \types \plug{(\act{I} \bdry)}{\act{I} e \equiv \act{J} e}
    }
  \end{equation*}  
  whose premises are derived as follows.

  The equational premises of the first row are given, while the object premises follow from the corresponding equational premises in \rref{TT-Congr-Eco} by \cref{prop:presuppositivity}.

  The second row of premises is more challenging. First, for each object premise, applying \Cref{prop:presuppositivity} to the corresponding equational premise in \rref{TT-Congr-Eco} yields
  $\Theta; \Gamma \vdash \plug{(\upact{I}{i} \BB_i)}{g_i}$
  which is then converted to
  $\Theta; \Gamma \vdash \plug{(\upact{J}{i} \BB_i)}{g_i}$
  with the aid of \cref{lem:equal-boundary-convert}.
  For an equational premise, we again use \cref{lem:equal-boundary-convert}, except that we apply it to the corresponding equational premise in the first row, noting that in this case~$f_i$ and~$g_i$ are the same.

  The third row of premises is given. The last premise, when present, follows by \cref{thm:admissibility-equality-instantiation} from the fact that~$R$ is finitary.
\end{proof}

\restateinversion*

\begin{proof}
  \label{proof:inversion}%
  We proceed by induction on the derivation $\Gamma; \Theta \types \JJ$.
  If the derivation concludes with \rref{TT-Var}, \rref{TT-Meta}, a symbol rule, or \rref{TT-Abstr}, then it already has the desired form.
  The remaining case is a derivation $D$ ending with a term conversion rule
  \begin{equation*}
    \small
    \infer
    { \infer*{D_1}{\Theta; \Gamma \types t : A} \\
      \infer*{D_2}{\Theta; \Gamma \types A \equiv B}
    }{
      \Theta; \Gamma \types t : B
    }
  \end{equation*}
  By induction hypothesis we may invert $D_1$ and obtain a derivation $D'$ of $\Theta; \Gamma \types t : A$ as in the statement of the theorem:
  \begin{enumerate}
  \item
    If $D'$ ends with \rref{TT-Var}, \rref{TT-Meta} or a term symbol rule then $A = \natty{\Theta; \Gamma}{t}$. Either $\natty{\Theta; \Gamma}{t} = B$ and we use~$D'$, or~$\natty{\Theta; \Gamma}{t} \neq B$ and we use $D$.
  \item
    If $D'$ concludes with a term conversion
    \begin{equation*}
      \small
      \infer
      { \infer*{D'_1}{\Theta; \Gamma \types t : \natty{\Theta; \Gamma}{t}} \\
        \infer*{D'_2}{\Theta; \Gamma \types \natty{\Theta; \Gamma}{t} \equiv A} }
      { \Theta; \Gamma \types t : A }
    \end{equation*}
    there are again two cases. If $\natty{\Theta; \Gamma}{t} = B$ we use $D'_1$, otherwise we combine $\natty{\Theta; \Gamma}{t} \equiv A$ and $A \equiv B$ by transitivity and conversion:
    \begin{equation*}
      \small
      \infer
      { \infer*{D'_1}{\Theta; \Gamma \types t : \natty{\Theta; \Gamma}{t}} \\
        \infer*{
          \infer*{D_2'}{\Theta; \Gamma \types \natty{\Theta; \Gamma}{t} \equiv A} \\
          \infer*{D_2}{\Theta; \Gamma \types A \equiv B}
        }{
          \Theta; \Gamma \types \natty{\Theta; \Gamma}{t} \equiv B
        }
      }{
        \Theta; \Gamma \types t : B
      }
      \ifjar \else \qedhere \fi
    \end{equation*}
  \end{enumerate}
\end{proof}

\subsection{Proofs of meta-theorems about context-free type theories}
\label{sec:proofs-from-context-free-meta-theorems}

This section provides missing proofs from \cref{sec:context-free-meta-theorems}.

\restatecontextfreepreparesubst*

\begin{proof}
  \label{proof:context-free-prepare-subst}%
  We may invert the derivation of $\types \abstr{\vec{x} \of \vec{A}} \; \JJ$ to obtain a series of applications of~\rref{CF-Abstr}, yielding types $A'_1, \ldots, A'_n$ and (suitably fresh) free variables $\avars{a}{1}{A'_1}, \ldots, \avars{a}{n}{A'_n}$ where, for $i = 1, \ldots, n$,
  \begin{equation*}
    A'_i = A_i[\avars{a}{1}{A'_1}/x_1, \ldots, \avars{a}{{i-1}}{A'_{i-1}}/x_{i-1}]
    \quad\text{and}\quad
    \types \isType{A'_i}.
  \end{equation*}
  At the top of the abstractions sits a derivation~$D$ of the judgement
  \begin{equation*}
    \JJ[\avars{a}{1}{A'_1}/x_1, \ldots, \avars{a}{n}{A'_n}/x_n].
  \end{equation*}
  The proof proceeds by induction on the derivation~$D$, i.e.\ we only ever apply the induction hypotheses to
  derivations that have a series of abstractions, and on the top a derivation that is structurally smaller than~$D$.
  Let us write
  \begin{align*}
    \theta &= [\avars{a}{1}{A'_1}/x_1, \ldots, \avars{a}{n}{A'_n}/x_n],\\
    \zeta  &= [x_n/\avars{a}{n}{A'_n}, \ldots, x_1/\avars{a}{1}{A'_1}],\\
    \tau &= [t_1/x_1, \ldots, t_n/x_n].
  \end{align*}

  \inCase{CF-Var}
  Suppose the derivation ends with the variable rule
  \begin{equation*}
    \infer{
    }{
      \types \avar{b}{B} : B
    }
  \end{equation*}
  If $\avar{b}{B}$ is one of~$\avars{a}{i}{A'_i}$ then $\JJ = \abstr{\vec{x} \of \vec{A}} \; x_i : A_i$,
  hence $\JJ \tau = (t_i : A_i[\upto{\vec{t}}{i}/\upto{\vec{x}}{i}])$, which is derivable by assumption.
  If $\avar{b}{B}$ is none of~$\avars{a}{i}{A_i}$'s then $\avars{a}{i}{A_i} \notin \fv{B}$ by freshness, hence $\JJ \tau = (\avar{b}{B} : B)$, so we may reuse the same variable rule.

  \inCase{CF-Abstr}
  Suppose the derivation ends with an abstraction
  \begin{equation*}
    \infer{
      \types \isType{(A_{n+1} \theta)} \\
      \avars{a}{n+1}{A_{n+1} \theta} \notin \fv{\JJ' \theta} \\
      (\JJ' \theta)[\avars{a}{n+1}{A_{n+1} \theta}/x_{n+1}]
    }{
      \types \abstr{x_{n+1} \of A_{n+1}\theta} \; (\JJ'\theta)
    }
  \end{equation*}
  We extend the substitution by $t_{n+1} = \avars{a}{n+1}{A_{n+1}  \tau}$ and apply the induction hypothesis to the abstracted derivation of the right-hand premise, whose conclusion is
  $\abstr{\vec{x} \of \vec{A}} \abstr{x_{n+1} \of A_{n+1}} \; \JJ'$, to obtain
  $
    \types \JJ'[\tau, t_{n+1}/x_{n+1}]
  $.
  We may abstract $\avars{a}{n+1}{A_{n+1} \tau}$ to get the desired judgement
  $
    \types \abstr{x_{n+1} \of A_{n+1}\tau} \; (\JJ'\tau)
  $.

  \inCaseCustom{All other cases}
  The remaining cases all follow the same pattern: abstract the premises, apply the induction hypotheses to them, and conclude with the same rule. We demonstrate how this works in case of~$D$ ending with an instance of a specific rule
  $
    R = (\rawRule{\symM_1^{\BB_1}, \ldots, \symM_n^{\BB_n}}{\plug{\bdry}{e}})
  $
  instantiated with~$I = \finmap{\symM_1^{\BB_1} \mto e_1, \ldots, \symM_n^{\BB_n} \mto e_n}$:
  \begin{equation*}
    \infer{
      {\begin{aligned}
       &\types \plug{(\upact{I}{i} \BB_i)}{e_i}  &&\text{for $i = 1, \ldots, n$} \\
       &\types \act{I} \bdry
       \end{aligned}}
    }{
      \types \act{I} (\plug{\bdry}{e})
    }
  \end{equation*}
  Define the instantiation $J$ of the premises of~$R$ by $J(\symM_i) = e'_i = (e_i \zeta) \tau$.
  Note that $\types (\act{I} (\plug{\bdry}{e})) \tau$ equals $\types \act{J} (\plug{\bdry}{e})$, therefore
  we may derive it by $\act{J} R$.
  The last premise of~$\act{J} R$ is $\types (\act{I} \bdry) \tau$, and it follows by \cref{lem:context-free-prepare-subst-bdry} applied to the last premise of~$\act{I} R$.
  For $i = 1, \ldots, n$, abstract $\types \plug{(\upact{I}{i} \BB_i)}{e_i}$ to 
  \begin{equation*}
    \types \abstr{\vec{x} \of \vec{A}} \; (\plug{(\upact{I}{i} \BB_i)}{e_i}) \zeta
  \end{equation*}
  and apply the induction hypothesis to derive
  $
    \types ((\plug{(\upact{I}{i} \BB_i)}{e_i}) \zeta) \tau
  $,
  which equals
  $
    \types \plug{(((\upact{I}{i} \BB_i) \zeta)  \tau)}{(e_i \zeta)  \tau}
  $
  and because~$\BB_i$ does not contain any free variables, also to
  $
    \types \plug{(\upact{J}{i} \BB_i)}{e'_i}.
  $
\end{proof}

\restatecontextfreepresuppositivity*

\begin{proof}
  \label{proof:context-free-presuppositivity}%
  The proof proceeds by induction on the number of metavariables appearing in the judgement and the derivation of $\types \plug{\BB}{e}$. That is, each appeal to the induction hypothesis reduces the number of metavariables, or is applied to a subderivation.

  \inCase{CF-Var}
  Immediate, by the well-typedness of annotations.

  \inCase{CF-Meta}
  Immedate as the desired judgement is a premise of the rule.

  \inCase{CF-Meta-Congr-Tm}
  Suppose $\BB = \abstr{x_1 \of A_1} \cdots \abstr{x_m \of A_m}\; \Box : B$ and
  consider a derivation ending with the metavariable congruence rule
  \begin{equation*}
    \infer
    { { \begin{aligned}
          &\text{for $k = 1, \ldots, m$:}\\
          &\types s_k : A_k[\upto{\vec{s}}{k}/\upto{\vec{x}}{k}]
          \\
          &\types t_k : A_k[\upto{\vec{t}}{k}/\upto{\vec{x}}{k}]
          \\
          &\erase{t_k} = \erase{t'_k}
          \\
          &\types s_k \equiv t'_k : A_k[\upto{\vec{s}}{k}/\upto{\vec{x}}{k}]
                  \by \alpha_k
          \\
          &\types v : B[\vec{s}/\vec{x}]
           \qquad\erase{\symM^\BB(\vec t)} = \erase{v}
          \\
          &\suitable{\beta}
        \end{aligned} }
    }{
      \types
      \symM^\BB(\vec s) \equiv
      v
      : B[\vec{s}/\vec{x}]
      \by \beta
    }
  \end{equation*}
  The presuppositions are derived as follows:
  \begin{itemize}
  \item $\types \isType{B[\vec{s}/\vec{x}]}$ follows by \rref{CF-Subst} from $\types{\vec{x} \of \vec{A}} \; \isType{B}$, which in turn follows by inversion on $\types \BB$.
  \item $\types \symM^\BB(\vec s) : B[\vec{s}/\vec{x}]$ follows by \rref{CF-Meta}.
  \item $v : B[\vec{s}/\vec{x}]$ is a premise.
  \end{itemize}

  \inCase{CF-Abstr}
  Consider an abstraction
  \begin{equation*}
    \infer
    {
      \types \isType A \\
      \avar{a}{A} \not\in \fv{\plug{\BB}{e}} \\
      \types (\plug{\BB}{e})[\avar{a}{A}/x]
    }{
      \types \abstr{x \of A} \; \plug{\BB}{e}
    }
  \end{equation*}
  By induction hypothesis on the last premise, we obtain
  $\types \BB[\avar{a}{A}/x]$ after which we apply \rref{CF-Bdry-Abstr}.

  \inCaseText{of a specific rule}
  Immediate, as the well-formedness of the boundary is a premise.

  \inCaseText{of a congruence rule}
  Consider a congruence rulles associated with an object rule~$R$ and instantiated with~$I$ and~$J$, as in \cref{def:context-free-congruence-rule}.

  If $R$ concludes with $\types \isType A$, the presuppositions are $\types \isType{\act I A}$ and $\types \isType{\act J A}$, which are derivable by $\act I R$ and $\act J R$, respectively.

  If~$R$ concludes with $\types t : A$, the presuppositions are
  $\types \isType{\act{I} A}$,
  $\types \act{I} t : \act{I} A$, and
  $\types t' : \act{I} A$.
  We derive the first one by applying the induction hypothesis to the premise $\types t' : \act{I} A$, the second one by~$\act{I} R$, while the third one is a premise.

  \inCasesText{
  \rref{CF-EqTy-Refl},
  \rref{CF-EqTy-Sym},
  \rref{CF-EqTy-Trans},
  \rref{CF-EqTm-Refl},
  \rref{CF-EqTm-Sym},
  \rref{CF-EqTm-Trans}}
  These are all dispensed with straightforward appeals to the induction hypotheses.

  \inCase{CF-Conv-Tm}
  Consider a term conversion
  \begin{equation*}
    \infer
    { \types t : A \\
      \types A \equiv B \by \alpha
      \\\\
      \asm{t, A, B, \alpha} = \asm{t, B, \beta}
    }{
      \types \convert{t}{\beta} : B
    }
  \end{equation*}
  By induction hypothesis for the second premise, $\types \isType{B}$.

  \inCase{CF-Conv-EqTm}
  Consider a term equality conversion
  \begin{equation*}
    \infer
    {
      \types s \equiv t : A \by \alpha \\
      \types A \equiv B \by \beta \\\\
      \asm{s, A, B, \beta} = \asm{s, B, \gamma} \\\\
      \asm{t, A, B, \beta} = \asm{t, B, \delta}
    }{
      \types
      \convert{s}{\gamma} \equiv \convert{t}{\delta} : B \by \alpha
    }
  \end{equation*}
  As in the previous case, the induction hypothesis for the second premise provides $\types \isType{B}$.
  The induction hypothesis for the first premise yields
  \begin{equation*}
    \types s : A
    \qquad\text{and}\qquad
    \types t : A
  \end{equation*}
  We may convert these to
  $\types \convert{s}{\gamma} : B$ and $\types \convert{t}{\delta} : B$
  using the second premise.
\end{proof}

\restatecontextfreepreparesubsteq*

\begin{proof}
  \label{proof:context-free-prepare-subst-eq}.
  As in the proof of \cref{lem:context-free-prepare-subst},
  we invert the derivation of $\types \abstr{\vec{x} \of \vec{A}} \; \JJ$ to obtain types $A'_1, \ldots, A'_n$ and (suitably fresh) free variables $\avars{a}{1}{A'_1}, \ldots, \avars{a}{n}{A'_n}$ where, for $i = 1, \ldots, n$,
  \begin{equation*}
    A'_i = A_i[\avars{a}{1}{A'_1}/x_1, \ldots, \avars{a}{{i-1}}{A'_{i-1}}/x_{i-1}]
    \quad\text{and}\quad
    \types \isType{A'_i}
  \end{equation*}
  and a derivation~$D$ of the judgement
  \begin{equation*}
    \JJ[\avars{a}{1}{A'_1}/x_1, \ldots, \avars{a}{{i-1}}{A'_n}/x_n].
  \end{equation*}
  The proof proceeds by induction on the well-founded ordering of the rules, the number of metavariables, with a subsidiary induction on the derivation~$D$. That is, each appeal to the induction hypotheses either decreases the number of metavariables appearing in the judgement, or descends to a subderivation of~$D$.
  Let us write
  \begin{align*}
    \theta &= [\avars{a}{1}{A'_1}/x_1, \ldots, \avars{a}{{i-1}}{A'_n}/x_n], \\
    \sigma &= [s_1/x_1, \ldots, s_n/x_n], \\
    \tau &= [t_1/x_1, \ldots, t_n/x_n].
  \end{align*}

  \inCase{CF-Var}
  Suppose the derivation ends with the variable rule
  \begin{equation*}
    \infer{
    }{
      \types \avar{b}{B} : B
    }
  \end{equation*}
  If $\avar{b}{B}$ is one of~$\avars{a}{i}{A'_i}$ then $\JJ = \abstr{\vec{x} \of \vec{A}} \; x_i : A_i$,
  hence~\eqref{it:cf-prep-subst-eq-2} is satisfied by~\eqref{eq:cf-prep-subst-eq-5}.
  If $\avar{b}{B}$ is none of~$\avars{a}{i}{A_i}$'s then $\avars{a}{i}{A_i} \notin \fv{B}$ by freshness, hence~\eqref{it:cf-prep-subst-eq-2} is satisfied by
  $\types \avar{b}{B} \equiv \avar{b}{B} : B \by \asset{}$,
  which holds by \rref{CF-EqTm-Refl}.

  \inCase{CF-Abstr}
  Suppose the derivation ends with an abstraction
  \begin{equation}
    \label{eq:cf-prep-subst-eq-abstr}
    \infer{
      \types \isType{(A_{n+1} \theta)} \\
      \avars{a}{n+1}{A_{n+1} \theta} \notin \fv{\JJ \theta} \\
      (\JJ' \theta)[\avars{a}{n+1}{A_{n+1} \theta}/y]
    }{
      \types \abstr{x_{n+1} \of A_{n+1}\theta} \; (\JJ'\theta)
    }
  \end{equation}
  We may abstract the first premise to $\types \abstr{\vec{x} \of \vec{A}} \; \isType{A_{n+1}}$, apply
  \cref{lem:context-free-prepare-subst} to derive $\types \isType{A_{n+1} \tau}$, and the induction hypothesis to obtain $\beta_{n+1}$ and $A'$ such that $\erase{A_{n+1} \tau} = \erase{A'}$,
  \begin{equation*}
    \types \isType{A'}
    \qquad\text{and}\qquad
    \types A_{n+1} \sigma \equiv A' \by \beta_{n+1}.
  \end{equation*}
  By \rref{CF-EqTy-Trans} and \rref{CF-EqTy-Refl} it follows that for some $\gamma_{n+1}$
  \begin{equation*}
    \types A_{n+1} \sigma \equiv A_{n+1} \tau \by \gamma_{n+1}.
  \end{equation*}
  Let $\avars{a}{n+1}{A_{n+1} \sigma}$ be fresh, and define
  \begin{equation*}
    s_{n+1} = t'_{n+1} = \avars{a}{n+1}{A_{n+1} \sigma},
    \qquad
    t_{n+1} = \convert{\avars{a}{n+1}{A_{n+1} \sigma}}{\gamma_{n+1}},
    \quad\text{and}\quad
    \alpha_{n+1} = \asset{}.
  \end{equation*}
  We may abstract the last premise of~\eqref{eq:cf-prep-subst-eq-abstr} to
  \begin{equation*}
    \types \abstr{x_1 \of A_1} \cdots \abstr{x_{n+1} \of A_{n+1}} \; \JJ'
  \end{equation*}
  apply the induction hypothesis with the given $s_{n+1}$, $t_{n+1}$ and $t'_{n+1}$ to derive either \eqref{it:cf-prep-subst-eq-1} or \eqref{it:cf-prep-subst-eq-2}, and abstract $\avars{a}{n+1}{A_{n+1} \sigma}$ to get the desired judgements.

  \inCase{CF-Meta}
  We consider the case of an object metavariable, and leave the easier case of a type metavariable to the reader.
  Let $\BB = (\abstr{\vec{y} \of \vec{B}} \; \Box : C)$, and suppose the derivation ends with an application of the metavariable rule,
  \begin{equation}
    \label{eq:cf-prep-subst-eq-meta}
    \infer{
      {\begin{aligned}
       &\types u_j \theta : B_j[\upto{\vec{u}}{j}\theta/\upto{\vec{y}}{j}] &&\text{for $j = 1, \ldots, m$} \\
       &\types \Box : C[\vec{u} \theta / \vec{y}]
       \end{aligned}}
    }{
      \types \symM^\BB(\vec{u} \theta) : C[\vec{u} \theta/\vec{y}]
    }
  \end{equation}
  For each $j = 1, \ldots, m$ we may abstract the premise of~\eqref{eq:cf-prep-subst-eq-meta} to
  \begin{equation*}
    \types \abstr{\vec{x} \of \vec{A}} \; u_j : B_j[\upto{\vec{u}}{j}/\upto{\vec{y}}{j}]
  \end{equation*}
  and apply \cref{lem:context-free-prepare-subst}, once with~$\vec{s}$ and once with~$\vec{t}$, to
  derive
  \begin{align*}
    &\types u_j \sigma : B_j[\upto{(\vec{u} \sigma)}{j}/\upto{\vec{y}}{j}], \\
    &\types u_j \tau : B_j[\upto{(\vec{u} \tau)}{j}/\upto{\vec{y}}{j}],
  \end{align*}
  where we took into account the fact that $B_j$ does not contain any bound variables.
  Also, by induction hypothesis there are $\delta_j$ and $u'_j$ such that
  $\erase{u_j \tau} = \erase{u'_j}$ and
  \begin{equation*}
    \types u_j \sigma \equiv u'_j
      : B_j[\upto{(\vec{u}  \sigma)}{j} / \upto{\vec{y}}{j}] \by \delta_j.
  \end{equation*}
  Next, we invert the last premise of~\eqref{eq:cf-prep-subst-eq-meta} and abstract it to
  $
    \types \abstr{\vec{x} \of \vec{A}} \; \isType{C[\vec{u}/\vec{y}]}
  $.
  By induction hypothesis we obtain~$\delta'$ and~$C'$ such that
  $\erase{C'} = \erase{C[\vec{u} \tau/\vec{y}]}$ and
  $\types C[\vec{u} \sigma/\vec{y}] \equiv C' \by \delta'$, hence $\types C[\vec{u} \sigma/\vec{y}] \equiv C[\vec{u} \tau/\vec{y}] \by \delta''$ for some~$\delta''$. Now $\eqref{it:cf-prep-subst-eq-2}$ is satisfied, for some~$\delta'''$
  \begin{align*}
    &\vdash \convert{\symM^\BB(\vec{u} \tau)}{\delta} : C[\vec{u}\sigma/\vec{y}],\\
    &\vdash \symM^\BB(\vec{u}\sigma) \equiv \convert{\symM^\BB(\vec{u} \tau)}{\delta} : C[\vec{u}\sigma/\vec{y}] \by \delta'''
  \end{align*}
  where the last judgement follows by the congruence rule for~$\symM^\BB$.

  \inCaseText{of a specific term rule}
  Suppose the derivation ends with a specific rule
  $R = (\rawRule{\symM_1^{\BB_1}, \ldots, \symM_m^{\BB_m}}{u : B})$ instantiated with
  $I' = \finmap{\symM_1^{\BB_1} \mto e'_1, \ldots, \symM_m^{\BB_m} \mto e'_m}$:
  \begin{equation*}
    \infer{
      {\begin{aligned}
       &\types \plug{(\upact{I'}{j} \BB_j)}{e'_j} &&\text{for $j = 1, \ldots, n$} \\
       &\types (\Box : \act{I'} B)
       \end{aligned}}
    }{
      \types \act{I'} u : \act{I'} B
    }
  \end{equation*}
  Let $\zeta = [x_n/\avars{a}{n}{A'_n}, \ldots, x_1/\avars{a}{1}{A'_1}]$ be the abstraction that undoes~$\theta$.
  Define $e_j = e'_j \zeta$ and $I = I' \zeta$, so that $e'_j = e_j \theta$ and $I' = I \theta$, which allows us to write the above judgement as
  \begin{equation*}
    \infer{
      {\begin{aligned}
       &\types (\plug{(\upact{I}{j} \BB_j)}{e_j}) \theta &&\text{for $j = 1, \ldots, n$} \\
       &\types (\Box : \act{I} B) \theta
       \end{aligned}}
    }{
      \types (\act{I} u) \theta : (\act{I} B) \theta
    }
  \end{equation*}
  We invert the last premise, abstract to
  $
    \types \abstr{\vec{x} \of \vec{A}} \; \isType{\act{I} B}
  $,
  and apply \cref{lem:context-free-prepare-subst} to derive
  $\types \isType{(\act{I} B) \sigma}$.
  Next, the induction hypothesis provides $\beta$ and $B'$ such that $\erase{B'} = \erase{(\act{I} B) \tau}$ and
  $\types (\act{I} B) \sigma \equiv B' \by \beta$.
  Therefore, we have $\beta'$ such that
  \begin{equation*}
    \types (\act{I} B) \sigma \equiv (\act{I} B) \tau \by \beta'.
  \end{equation*}
  It suffices to show
  \begin{equation*}
    \types
    (\act{I} u) \sigma \equiv
    \convert{(\act{I} u) \tau}{\beta'}
    : (\act{I} B) \sigma \by \delta
  \end{equation*}
  for a suitable~$\delta$.
  This is precisely the conclusion of the congruence rule for~$R$, so we derive its premises.
  For any $j = 1, \ldots, m$ we may abstract the $j$-th premise to
  \begin{equation}
    \label{eq:cf-prep-subst-eq-spec-1}
    \types \abstr{\vec{x} \of \vec{A}} \; \plug{(\upact{I}{j} \BB_j)}{e_j},
  \end{equation}
  and apply \cref{lem:context-free-prepare-subst}, once with $\vec{s}$ and once with $\vec{t}$, to derive
  \begin{equation*}
    \types \plug{(\upact{(I \sigma)}{j} \BB_j)}{e_j \sigma}
    \qquad\text{and}\qquad
    \types \plug{(\upact{(I \tau)}{j} \BB_j)}{e_j \tau}.
  \end{equation*}
  For each object premise with boundary $\BB_j$, the remaining premises are provided precisely by the induction hypotheses.

  \inCaseText{of a specific type rule}
  Suppose the derivation ends with a specific rule
  $R = (\rawRule{\symM_1^{\BB_1}, \ldots, \symM_m^{\BB_m}}{\isType{B}})$ instantiated with
  $I' = \finmap{\symM_1^{\BB_1} \mto e'_1, \ldots, \symM_m^{\BB_m} \mto e'_m}$:
  \begin{equation*}
    \infer{
      \types \plug{(\upact{I'}{j} \BB_j)}{e'_j} \quad\text{for $j = 1, \ldots, n$} \\\\
      \types \isType{\Box}
    }{
      \types \isType{\act{I'} B}
    }
  \end{equation*}
  With $\zeta$ and $I$ as in the previous case, we may write the above as
  \begin{equation*}
    \infer{
      \types (\plug{(\upact{I}{j} \BB_j)}{e_j}) \theta \quad\text{for $j = 1, \ldots, n$}
    }{
      \types \isType{(\act{I} B) \theta}
    }
  \end{equation*}
  where we elided the trivial boundary premise.
  It suffices to find a suitable $\gamma$ such that
  $
    \types (\act{I} B) \sigma \equiv (\act{I} B) \tau \by \gamma
  $,
  which is precisely the conclusion of the congruence rule for~$R$, whose premises are derived as in the previous case.

  \inCase{CF-Conv-Tm}
  Suppose the derivation ends with an application of the conversion rule
  \begin{equation*}
    \infer
    { \types u\theta : B\theta \\
      \types B\theta \equiv C\theta \by \beta\theta }
    { \types \convert{u\theta}{\beta\theta \cup \asm{B\theta}} : C\theta }
  \end{equation*}
  We abstract the first premise to $\types \abstr{\vec{x} \of \vec{A}} \; u : B$ and apply the induction hypothesis to obtain $\delta$ and $u'$ such that $\erase{u'} = \erase{u \tau}$ and
  \begin{equation*}
    \types u \sigma \equiv u' : B \sigma \by \delta.
  \end{equation*}
  We abstract the second premise to $\types \abstr{\vec{x} \of \vec{A}} \; B \equiv C \by \beta$, apply \cref{lem:context-free-prepare-subst} to derive
  $
    \types B \sigma \equiv C \sigma \by \beta \sigma
  $,
  and use \rref{CF-Conv-EqTm} to conclude, for suitable~$\gamma$ and $\gamma'$,
  \begin{equation*}
    \types \convert{u \sigma}{\gamma} \equiv
           \convert{u'}{\gamma'} : C \sigma \by \delta.
    \ifjar \else \qedhere \fi
  \end{equation*}
\end{proof}

\restatecontextfreesubsteq*

\begin{proof}
  \label{proof:context-free-subst-eq}%
  \Cref{lem:context-free-prepare-subst-eq} applied to the premises of \rref{CF-Subst-EqTy} provides~$\gamma$ and~$C'$ such that $\erase{C[\vec{t}/\vec{x}]} = \erase{C'}$ and
  \begin{equation}
    \label{eq:cf-subst-eq-1}%
    \types \abstr{\vec{y} \of \vec{B}[\vec{s}/\vec{x}]} \;
    C[\vec{s}/\vec{x}] \equiv C' \by \gamma.
  \end{equation}
  We would like to replace $C'$ in the right-hand side with $C[\vec{t}/\vec{x}]$, which we can so long as
  \begin{equation*}
    \types \abstr{\vec{y} \of \vec{B}[\vec{s}/\vec{x}]} \; \isType{C'}
    \quad\text{and}\quad
    \types \abstr{\vec{y} \of \vec{B}[\vec{s}/\vec{x}]} \; C[\vec{t}/\vec{x}].
  \end{equation*}
  The first judgement holds by \cref{prop:context-free-presuppositivity} applied to~\eqref{eq:cf-subst-eq-1} under the abstraction, while the second one is a substitution instance of the first premise. This establishes admissibility of \rref{CF-Subst-EqTy}.

  In case of \rref{CF-Subst-EqTm} the same lemma yields $\delta$ and $u'$ such that $\erase{u[\vec{t}/\vec{x}]} = \erase{u'}$ and
  \begin{equation*}
    \types \abstr{\vec{y} \of \vec{B}[\vec{s}/\vec{x}]} \;
       u[\vec{s}/\vec{x}] \equiv u' : C[\vec{s}/x] \by \delta.
  \end{equation*}
  We would like to replace $u'$ with a converted $u[\vec{t}/\vec{x}]$, which we can by an argument similar to the one above.
\end{proof}

\restatecontextfreeinstantiationadmissible*

\begin{proof}
  \label{proof:context-free-instantiation-admissible}.
  We proceed by induction on the derivation of $\types \JJ$.
  We only devote attention to the metavariable and abstraction rules, as all the other cases are straightforward.
  Suppose $I = \finmap{\symM_1^{\BB_1} \mto e_1, \ldots, \symM_n^{\BB_n} \mto e_n}$.

  \inCase{CF-Meta}
  Consider an application of the metavariable rule for~$\symM_i^{\BB_i}$ with
  $\BB_i = (\abstr{x_1 \of A_1} \cdots \abstr{x_m \of A_m}\; \bdry)$ and $e_i = \abstr{\vec{x}} e$:
  \begin{equation*}
    \infer{
      {\begin{aligned}
      &\types t_i : A_i[\upto{\vec{t}}{i}/\upto{\vec{x}}{i}] &&\text{for $i = 1, \ldots, n$} \\
      &\types \bdry[\vec{t}/\vec{x}]
      \end{aligned}}
    }{
      \types (\plug{(\bdry[\vec{t}/\vec{x}])}{\symM_i^{\BB_i}(\vec{t})})
    }
  \end{equation*}
  Because
  $
    \act{I} (\plug{(\bdry[\vec{t}/\vec{x}])}{\symM_i^{\BB_i}(\vec{t})}) =
    \plug{((\act{I} \bdry)[\act{I} \vec{t}/\vec{x}])}{e[\act{I} \vec{t}/\vec{x}]}
  $
  we need to derive
  \begin{equation}
    \label{eq:cf-instantiate-1}%
    \types \plug{((\act{I} \bdry)[\act{I} \vec{t}/\vec{x}])}{e[\act{I} \vec{t}/\vec{x}]}.
  \end{equation}
  Because $I$ is derivable, we know that
  $
    \types \abstr{\vec{x} \of \upact{I}{i} \vec{A}} \; \plug{(\upact{I}{i} \bdry)}{e}
  $.
  By induction hypothesis
  $
    \types \upact{I}{i} t_j : (\upact{I}{i} A_j)[\upact{I}{i} \upto{\vec{t}}{j}/\upto{\vec{x}}{j}])
  $
  for each $j = 1, \ldots, m$, so
  by \cref{lem:context-free-prepare-subst} we derive
  $
    \types (\plug{(\upact{I}{i} \bdry)}{e}) [\upact{I}{i} \vec{t}/\vec{x}]
  $,
  which coincides with~\eqref{eq:cf-instantiate-1}.

  \inCase{CF-Meta-Congr-Ty}
  We consider the congruence rule for types only. Suppose the derivation ends with
  an application of the congruence rule for~$\symM_i^{\BB_i}$ with
  $\BB_i = (\abstr{x_1 \of A_1} \cdots \abstr{x_m \of A_m}\; \isType{\Box})$ and $e_i = \abstr{\vec{x}} B$:
  \begin{equation*}
    \infer
    { { \begin{aligned}
          &\text{for $k = 1, \ldots, m$:} \\
          &\types s_k : A[\upto{\vec{s}}{k}/\upto{\vec{x}}{k}]
          \\
          &\types t_k : A[\upto{\vec{t}}{k}/\upto{\vec{x}}{k}]
          \\
          &\erase{t_k} = \erase{t'_k}
          \\
          &\types s_k \equiv t'_k : A[\upto{\vec{s}}{k}/\upto{\vec{x}}{k}] \by \alpha_k
          \\
          &\text{$\beta$ suitable}
        \end{aligned} }
    }{
      \types
      \symM_i^{\BB_i}(\vec s) \equiv
      \symM_i^{\BB_i}(\vec t)
      \by \beta
    }
  \end{equation*}
  Because $I$ is derivable, we know that
  $
    \types \abstr{\vec{x} \of \upact{I}{i} \vec{A}} \; \isType{B}
  $,
  hence \cref{lem:context-free-prepare-subst-eq} applies.

  \inCase{CF-Abstr}
  Suppose the derivation ends with an abstraction
  \begin{equation*}
    \inferrule
    {
      \types \isType A \\
      \avar{a}{A} \not\in \fv{\JJ} \\
      \types \JJ[\avar{a}{A}/x]
    }{
      \types \abstr{x \of A} \; \JJ
    }
  \end{equation*}
  Without loss of generality we may assume that $\avar{a}{\act{I} A} \not\in \fv{\act{I} \JJ}$. (If not, rename $\mkvar{a}$ to a fresh symbol.)
  We may apply the induction hypotheses to both premises and get
  \begin{equation*}
    \types \isType{\act{I} A}
    \qquad\text{and}\qquad
    \types (\act{I} \JJ)[\avar{a}{\act{I} A}/x].
  \end{equation*}
  and derive the desired judgement $\types \abstr{x \of \act{I} A} \; \act{I} \JJ$ by abstracting $\avar{a}{\act{I} A}$ in the right-hand judgement.
\end{proof}

\restatecontextfreenaturaltype*

\begin{proof}
  \label{proof:context-free-natural-type}%
  We proceed by induction on the derivation of $\types t : A$.

  \inCasesText{
    \rref{CF-Var},
    \rref{CF-Meta},
    and symbol rule
  }
  In these cases $t = \strip{t}$ and $\natty{}{t} = A$, so we already have $\types \strip{t} : \natty{}{t}$, while $\types \natty{}{t} \equiv A \by \asset{}$ holds by reflexivity.

  \inCase{CF-Conv-Tm}
  Consider a derivation ending with a conversion
  \begin{equation*}
    \infer
    { \types t : B \\
      \types B \equiv A \by \beta
    }{
      \types \convert{t}{\alpha} : A
    }
  \end{equation*}
  where $\asm{t, B, A, \beta} = \asm{t, A, \alpha}$.
  By induction hypothesis for the first premise we obtain $\types \strip{t} : \natty{}{t}$ and $\types \natty{}{t} \equiv B \by \residue{t}$, derived by one of the desired rules. Because $\strip{t} = \strip{\convert{t}{\alpha}}$ and $\natty{}{t} = \natty{}{\convert{t}{\alpha}}$, the first claim is established. For the second one, we apply \rref{CF-EqTy-Trans} like this:
  \begin{equation*}
    \infer{
      \types \natty{}{t} \equiv B \by \residue{t}
      \\
      \types B \equiv A \by \beta
    }{
      \types \natty{}{t} \equiv A \by \residue{t} \cup \alpha
    }
  \end{equation*}
  Suitability of $\residue{t} \cup \alpha$ is implied by $\asm{\natty{}{t}, \residue{t}} = \asm{t}$:
  \begin{align*}
    \asm{\natty{}{t}, B, \residue{t}, A, \beta}
    &= \asm{t, B, A, \beta} \\
    &= \asm{t, A, \alpha} \\
    &= \asm{\natty{}{t}, A, \residue{t}, \alpha}. \ifjar \else \qedhere \fi
  \end{align*}
\end{proof}

\restateboundaryconvert*

\begin{proof}
  \label{proof:boundary-convert}%
  We proceed by induction on the derivation of~$\types \BB_1$.

  \inCase{CF-Bdry-Ty}
  If $\BB_1 = (\isType \Box)$ then $\BB_2 = (\isType \Box)$ and we may take $e_2 = e_1$.

  \inCase{CF-Bdry-Tm}
  If $\BB_1 = (\Box : A_1)$ then $\BB_2 = (\Box : A_2)$ and $\erase{A_1} = \erase{A_2}$, therefore $\types A_1 \equiv A_2 \by \asset{}$ by \rref{CF-EqTy-Refl}. We may take $e_2 = \convert{e_1}{\asm{A_1} \setminus \asm{A_2}}$ and derive $\types e_2 : A_2$ by \rref{CF-Conv-Tm}.

  \inCase{CF-Bdry-EqTy}
  If $\BB_1 = (A_1 \equiv B_1 \by \Box)$ then $\BB_2 = (A_2 \equiv B_2 \by \Box)$, $\erase{A_1} = \erase{A_2}$ and $\erase{B_1} = \erase{B_2}$. By \rref{CF-EqTy-Refl} we obtain $\types A_2 \equiv A_1 \by \asset{}$ and $\types B_1 \equiv B_2 \by \asset{}$. 
  We take $e_2 = (e_1 \cup \asm{A_1} \cup \asm{B_1}) \setminus (\asm{A_2} \cup \asm{B_2})$ and derive $\types A_2 \equiv B_2 \by e_2$ by two applications of \rref{CF-EqTy-Trans}.

  \inCase{CF-Bdry-EqTm}
  If $\BB_1 = (s_1 \equiv t_1 : A_1 \by \Box)$ then $\BB_2 = (s_2 \equiv t_2 : A_2 \by \Box)$, $\erase{s_1} = \erase{s_2}$, $\erase{t_1} = \erase{t_2}$ and $\erase{A_1} = \erase{A_2}$. By \rref{CF-EqTy-Refl} we obtain $\types A_1 \equiv A_2 \by \asset{}$, then by \rref{CF-Conv-EqTm}
  \begin{equation*}
    \types \convert{s_1}{\gamma} \equiv \convert{t_1}{\delta} : A_2 \by e_1
  \end{equation*}
  where $\gamma = \asm{A_1} \setminus \asm{s_1, A_2}$ and
  $\delta = \asm{A_1} \setminus \asm{t_1, A_2}$.
  Next, by reflexivity
  \begin{align*}
    &\types s_2 \equiv \convert{s_1}{\gamma} : A_2 \by \asset{} \\
    &\types \convert{t_1}{\delta} \equiv t_2 : A_2 \by \asset{}
  \end{align*}
  We may chain these together by transitivity to derive
  \begin{equation*}
    \types s_2 \equiv t_2 : A_2 \by e_2
  \end{equation*}
  where $e_2 = \asm{e_1, s_1, t_1, A_1} \setminus \asm{s_2, t_2, A_2}$.

  \inCase{CF-Bdry-Abstr}
  If $\BB_1 = (\abstr{x \of A_1}\; \BB'_1)$ then $e_1 = \abstr{x} e'_1$, $\BB_2 = \abstr{x \of A_2}\; \BB'_2$, $\erase{A_1} = \erase{A_2}$, and $\erase{\BB'_1} = \erase{\BB'_2}$.
  There is $\avar{a}{A_2} \not\in \fv{\BB'_2}$ such that $\types \BB'_2[\avar{a}{A_2}/x]$.
  We may apply \cref{lem:context-free-prepare-subst} to $\types \abstr{x \of A_1}\; \plug{\BB'_1}{e_1'}$ and $\types \convert{\avar{a}{A_2}}{\asset{}} : A_1$ to derive
  \begin{equation*}
    \types
    \plug{(\BB'_1[\convert{\avar{a}{A_2}}{\asset{}}/x])}
         {e_1'[\convert{\avar{a}{A_2}}{\asset{}}/x]}.
  \end{equation*}
  By \rref{CF-Bdry-Subst} we have $\types \BB'_1[\convert{\avar{a}{A_2}}{\asset{}}/x]$, hence we may apply
  the induction hypothesis to obtain $e''_2$ such that
  $\erase{e''_2} = \erase{e_1'[\convert{\avar{a}{A_2}}{\asset{}}/x]}$,
  $\asm{e''_2} \subseteq \asm{\BB'_1[\convert{\avar{a}{A_2}}{\asset{}}/x]}$,
  and $\types \plug{(\BB'_2[\avar{a}{A_2}/x])}{e''_2}$. Set $e'_2 = e''_2[x/\avar{a}{A_2}]$ and apply \rref{CF-Bdry-Abstr} to derive
  $\types \abstr{x \of A_2}\; \plug{\BB'_2}{e'_2}$.
  Thus we may take $e_2 = \abstr{x} e'_2$.
\end{proof}

\subsection{Proofs of theorems about translation betweeen tt- and cf-type theories}
\label{sec:proofs-from-context-free-cons-comp}

This section provides missing proofs from \cref{sec:context-free-cons-comp}.

\restatecftottbdryjdg*

\begin{proof}
  \label{proof:cf-to-tt-bdry-jdg}%
  We proceed by mutual structural induction on all three statements.

  To prove statement~\eqref{it:cf-to-tt-1},
  consider a finitary cf-theory $T = (R_i)_{i \in I}$, and let $(I, {\prec})$ be a well-founded order witnessing the finitary character of~$T$ (\cref{def:context-free-finitary}). We prove that $\trantt{T}$ is finitary with respect to $(I, {\prec})$ by a well-founded induction on the order.
  Given any $i \in I$, with
  \begin{equation*}
    R_i = (\rawRule{\symM_1^{\BB_1}, \ldots, \symM_n^{\BB_n}}{\J}),
  \end{equation*}
  let $\Theta = [\symM_1^{\BB_1} \of \erase{\BB_1}, \ldots, \symM_n^{\BB_n} \of \erase{\BB_n}]$.
  We verifty that $\trantt{(R_i)} = (\rawRule{\Theta}{\erase{\J}})$ is finitary in~$T' = \trantt{((R_j)_{j \prec i})}$ as follows:
  \begin{itemize}
  \item $\types_{T'} \isMCtx{\Theta}$ holds by induction on $k = 1, \ldots, n$: assuming $\types_{T'} \isMCtx{\upto{\Theta}{k}}$ has been established, apply \eqref{it:cf-to-tt-2} to a cf-derivation of $\types_{(R_i)_{j \prec i}} \BB_k$ and the suitable context $\upto{\Theta}{k}; \emptyCtx$.
  \item $\types_{T'} \erase{\J}$ holds by application of \eqref{it:cf-to-tt-2} to a cf-derivation of $\types_{(R_i)_{j \prec i}} \J$ and the suitable context $\Theta; \emptyCtx$.
  \end{itemize}

  We next address statement~\eqref{it:cf-to-tt-2}, which we prove by structural induction on the derivation of $\types_{T} \JJ$.

  \inCase{CF-Var}
  A cf-derivation ending with the variable rule
  \begin{equation*}
    \inferrule{ }{
      \types_T \avar{a}{A} : A
    }
  \end{equation*}
  is translated to an application of \rref{TT-Var}
  \begin{equation*}
    \inferrule{ \avar a A \in \vert{}\Gamma\vert{} }
    { \Theta; \Gamma \types_{\trantt{T}} \avar{a}{A} : \erase{A} }
  \end{equation*}
  By suitability of $\Gamma$ the side-condition $\avar a A \in \vert{}\Gamma\vert{}$ is satisfied, and $\Gamma(\avar a A) = \erase{A}$.

  \inCase{CF-Meta}
  Consider a cf-derivation ending in 
  \begin{equation*}
    \infer{
      {\begin{aligned}
      &\types_T t_i : A_i[\upto{\vec{t}}{i}/\upto{\vec{x}}{i}] &&\text{for $i = 1, \ldots, n$} \\
      &\types_T \bdry[\vec{t}/\vec{x}]
      \end{aligned}}
    }{
      \types \plug{(\bdry[\vec{t}/\vec{x}])}{\symM^{\BB}(\vec{t})}
    }
  \end{equation*}
  Because erasure commutes with substitution we have
  \begin{align*}
    \erase{A_i[\upto {\vec t} i/\upto {\vec x} i]}
    &= \erase{A_i}[\erase{\upto {\vec t} i}/\upto {\vec x} i],
    \\
    \erase{\bdry[\vec t/\vec x]} &= \erase{\bdry}[\erase{\vec t}/\vec x],
    \\
    \erase{(\plug{\bdry}{\symM^{\abstr{\vec{x} \of \vec{A}} \bdry}(\vec{x})})[\vec{t}/\vec{x}]}
    &= (\plug{\erase \bdry} {\erase{\symM^{\abstr{\vec{x} \of \vec{A}} \bdry}(\vec{x})}})[\erase{\vec{t}}/\vec{x}].
  \end{align*}
  Applying \rref{TT-Meta} to the translation of the premises obtained by the induction hypothesis thus yields the desired result. Suitability of $\Theta; \Gamma$ is ensured because all premises are recorded in the conclusion.

  \inCasesText{\rref{CF-Meta-Congr-Ty} and \rref{CF-Meta-Congr-Tm}}
  We spell out the translation of the latter rule, where
  $\BB = \abstr{x_1 \of A_1} \cdots \abstr{x_m \of A_m}\; \Box : B$:
  \begin{equation}
    \label{eq:cf-tt-meta-congr}%
    \infer
    { { \begin{aligned}
          &\types s_k : A_k[\upto{\vec{s}}{k}/\upto{\vec{x}}{k}]
          &&\text{for $k = 1, \ldots, m$}
          \\
          &\types t_k : A_k[\upto{\vec{t}}{k}/\upto{\vec{x}}{k}]
          &&\text{for $k = 1, \ldots, m$}
          \\
          &\erase{t_k} = \erase{t'_k}
          &&\text{for $k = 1, \ldots, m$}
          \\
          &\types s_k \equiv t'_k : A[\upto{\vec{s}}{k}/\upto{\vec{x}}{k}]
                  \by \alpha_k
          &&\text{for $k = 1, \ldots, m$}
        \end{aligned} }
      \\\\
      \types v : B[\vec{s}/\vec{x}] \\
      \erase{\symM^\BB(\vec t)} = \erase{v} \\
      \suitable{\beta}
    }{
      \types
      \symM^\BB(\vec s) \equiv
      v
      : B[\vec{s}/\vec{x}]
      \by \beta
    }
  \end{equation}
  The context $\Theta; \Gamma$ is suitable for the premises because~$\beta$ is suitable.
  We apply \rref{TT-Meta-Congr} as follows:
  \begin{equation*}
    \infer
    {
     { \begin{aligned}
          &\Theta; \Gamma \types \erase{s_k} :
              \erase{A_k}[\upto{\erase{\vec{s}}}{k}/\upto{\vec{x}}{k}]
          &\text{for $k = 1, \ldots, m$}
          \\
          &\Theta; \Gamma \types \erase{t_k} :
              \erase{A_k}[\upto{\erase{\vec{t}}}{k}/\upto{\vec{x}}{k}]
          &\text{for $k = 1, \ldots, m$}
          \\
          &\Theta; \Gamma \types \erase{s_k} \equiv \erase{t_k} :
              \erase{A_k}[\upto{\erase{\vec{s}}}{k}/\upto{\vec{x}}{k}]
          &\text{for $k = 1, \ldots, m$}
          \\
          &\Theta; \Gamma \types \erase{B}[\erase{\vec{s}}/\vec{x}] \equiv \erase{B}[\erase{\vec{t}}/\vec{x}]
        \end{aligned} }
    }{
      \Theta; \Gamma \types
      \symM_k^\BB(\erase{\vec{s}}) \equiv \symM_k^\BB(\erase{\vec{t}}) :
      \erase{B}[\erase{\vec{s}}/\vec{x}]
    }
  \end{equation*}
  The first two rows of premises are secured by the induction hypotheses for the corresponding rows in~\eqref{eq:cf-tt-meta-congr}, and
  the premises in the third row are derivable by the side conditions in the third row and induction hypotheses for the fourth row.
  The last premise follows by \cref{thm:substitution-admissible} applied to $\Theta; \Gamma \types_{\trantt{T}} \isType{\erase{B}}$, which holds because we assumed $\types_{\trantt{T}} \isMCtx{\Theta}$.

  \inCase{CF-Abstr}
  A cf-derivation ending with an abstraction
  \begin{equation*}
    \inferrule
    {
      \types_T \isType A \\
      \avar{a}{A} \not\in \fv{\JJ} \\
      \types_T \JJ[\avar a A /x]
    }{
      \types_T \abstr{x \of A} \; \JJ
    }
  \end{equation*}
  is translated to a tt-derivation ending with \rref{TT-Abstr}
  \begin{equation*}
    \inferrule
    {
      \Theta; \Gamma \types_{\trantt{T}} \isType{\erase{A}} \\
      \avar{b}{A} \not\in \vert{}\Gamma\vert{} \\
      \Theta; \Gamma, \avar{b}{A} \of \erase{A} \types_{\trantt{T}} \erase{\JJ}[\avar{b}{A}/x]
    }{
      \Theta; \Gamma \types_{\trantt{T}} \abstr{x \of \erase{A}} \; \erase{\JJ}
    }
  \end{equation*}
  The premises get their derivations from induction hypotheses, where $\avar{b}{A} \not\in \vert{}\Gamma\vert{}$ ensures that $\Gamma, \avar{b}{A} \of \erase{A}$ is suitable for~$\JJ[\avar{b}{A}/x]$.

  \inCaseText{of a specific rule}
  Consider a derivation ending with an instantiation $I = \finmap{\symM_1^{\BB_1} \mto e_1, \ldots, \symM_n^{\BB^n} \mto e_n}$ of a raw cf-rule $R = (\rawRule{\symM_1^{\BB_1}, \ldots, \symM_n^{\BB^n}}{\plug \bdry e})$:
  \begin{equation*}
    \inferrule
    { { \begin{aligned}
          &\types_T
            \plug
            {(\upact I i \BB_i)}
            {e_i}
            \quad\text{for i = 1, \ldots, n}
          \\
          &\types_T \act I \bdry
        \end{aligned} } }
    { \types_T \act{I} (\plug \bdry e) }
  \end{equation*}
  Let $\erase{I} = \symM_1^{\BB_1} \mto \erase{e_1}, \ldots, \symM_n^{\BB^n} \mto \erase{e_n}$.
  Because erasure commutes with instantiation we have
  \begin{equation*}
    \erase{\plug{(\upact{I}{i} \BB_i)}{e_i}}
    =
    \plug{(\upact{\erase{I}}{i} \erase{\BB_i})}{\erase{e_i}}
  \end{equation*}
  and $\erase{\act{I} (\plug{\bdry}{e})} = \act{\erase{I}} \erase{\plug{\bdry}{e}}$.
  Thus we may appeal to the induction hypotheses for the premises and conclude by~$\trantt{R}$, so long as we remember to check that $\Theta; \Gamma$ is suitable for the premises, which it is because \cref{def:context-free-raw-rule} of raw cf-rules requires $\mv{\plug{\bdry}{e}} = \asset{\symM_1^{\BB_1}, \ldots, \symM_n^{\BB^n}}$.

  \inCaseText{of a congruence rule}
  Consider an application of the congruence rule associated with a cf-rule
  \begin{equation*}
    R = (\rawRule{\symM_1^{\BB_1}, \ldots, \symM_n^{\BB_n}}{t : A}),
  \end{equation*}
  as in \cref{def:context-free-congruence-rule}:
  \begin{equation}
    \label{eq:cf-tt-congr}%
    \infer{
      {\begin{aligned}
      &\types_T \plug{(\upact{I}{i} \BB_i)}{f_i}  &&\text{for $i = 1, \ldots, n$}\\
      &\types_T \plug{(\upact{J}{i} \BB_i)}{g_i}  &&\text{for $i = 1, \ldots, n$}\\
      &\erase{g'_i} = \erase{g_i} &&\text{for object boundary $\BB_i$} \\
      &\types_T \plug{(\upact{I}{i} \BB_i)}{f_i \equiv g'_i \by \alpha_i} &&\text{for object boundary $\BB_i$} \\
      &\types_T t' : \act{I} A &&\erase{t'} = \erase{\act{J} t} \\
      & &&\suitable{\beta}
    \end{aligned}}
    }{
      \types_T \act{I} t \equiv t' : \act{I} A \by \beta
    }
  \end{equation}
  The context $\Theta; \Gamma$ is suitable for the premises because~$\beta$ is suitable.
  We apply the corresponding congruence for $\trantt{R}$ (\cref{def:congruence-rule}):
  \begin{equation*}
    \infer{
      {\begin{aligned}
      &\Theta; \Gamma \types_{\trantt{T}} \plug{(\upact{\erase{I}}{i} \erase{\BB_i})}{\erase{f_i}}  &&\text{for $i = 1, \ldots, n$}\\
      &\Theta; \Gamma \types_{\trantt{T}} \plug{(\upact{\erase{J}}{i} \erase{\BB_i})}{\erase{g_i}}  &&\text{for $i = 1, \ldots, n$}\\
      &\Theta; \Gamma \types_{\trantt{T}} \plug{(\upact{\erase{I}}{i} \erase{\BB_i})}{\erase{f_i} \equiv \erase{g_i}} &&\text{for object boundary $\BB_i$} \\
      &\Theta; \Gamma \types_{\trantt{T}} \act{\erase{I}} \erase{A} \equiv \act{\erase{J}} \erase{A}
    \end{aligned}}
    }{
      \Theta; \Gamma \types_{\trantt{T}} \act{\erase{I}} \erase{t} \equiv \act{\erase{J}} \erase{t} : \act{\erase{I}} \erase{A}
    }
  \end{equation*}
  The first and the second row of premises are derivable by induction hypotheses for the corresponding rows in~\eqref{eq:cf-tt-congr}, while the third row is derivable because of the side conditions on the third row and induction hypotheses for the fourth row. The last premise follows by \cref{thm:admissibility-equality-instantiation} applied to $\Theta; \Gamma \types_{\trantt{T}} \isType{A}$, which in turn follows by induction hypothesis applied to a derivation of $\types_T \isType{A}$ witnessing the finitary character of~$R$.

  \inCase{CF-Conv-Tm}
  Consider a term conversion
  \begin{equation*}
    \inferrule{
      \types_T t : A \\
      \types_T A \equiv B \by \alpha
    }{
      \types_T \convert{t}{\beta} : B
    }
  \end{equation*}
  The side condition $\asm{t, A, B, \alpha} = \asm{t, B, \beta}$ ensures that $\Theta; \Gamma$ is suitable for both premises, hence we may apply the induction hypotheses to the premisess and conclude by \rref{TT-Conv-Tm}.

  \inCase{CF-Conv-EqTm}
  Consider an equality conversion
  \begin{equation*}
    \infer
    {
      \types s \equiv t : A \by \alpha \\
      \types A \equiv B \by \beta
    }{
      \types
      \convert{s}{\gamma} \equiv \convert{t}{\delta} : B \by \alpha
    }
  \end{equation*}
  The side conditions
  \begin{equation*}
    \asm{s, A, B, \beta} = \asm{s, B, \gamma}
    \qquad\text{and}\qquad
    \asm{t, A, B, \beta} = \asm{t, B, \delta}
  \end{equation*}
  ensure that $\Theta; \Gamma$ is suitable for both premises, hence we may apply the induction hypotheses to the premises and conclude by \rref{TT-Conv-EqTm}.
  As in the preceding case all assumptions in the premises already appear in the conclusion, and suitability is preserved.

  \inCasesText{\rref{CF-EqTy-Refl}, \rref{CF-EqTy-Sym}, \rref{CF-EqTy-Trans},
    \rref{CF-EqTm-Refl}, \rref{CF-EqTm-Sym}, \rref{CF-EqTm-Trans}}
  These all proceed by application of induction hypotheses to the premises, followed by the corresponding tt-rule, where crucially we rely on recording metavariables in the assumption sets to make sure that $\Theta$ and $\Gamma$ are suitable for the premises.

  \medskip

  Finally, we address statement~\eqref{it:cf-to-tt-2}, which is proved by structural induction on $\types_T \BB$.
  The base cases \rref{CF-Bdry-Ty}, \rref{CF-Bdry-Tm}, \rref{CF-Bdry-EqTy}, \rref{CF-Bdry-EqTm} reduce to translation of term and type judgements, while the induction step \rref{CF-Bdry-Abstr} is similar to the case \rulename{CF-Abstr} above.
\end{proof}

\restatetttocf*

\begin{proof}
\label{proof:tt-to-cf}%
We prove the above existence statements by explicit constructions, e.g., we prove \eqref{item:tt-cf-theory}) by constructing a specific $T'$ which meets the criteria, and similarly for the remaining parts.
We proceed by simultaneous structural induction on all the parts.

\medskip
\noindent
\emph{Proof of part (\ref{item:tt-cf-theory}):}
We proceed by induction on a well-founded order $(I, {\prec})$ witnessing the finitary character of $T = (R_i)_{i \in I}$.
Consider any $i \in I$, with the corresponding specific rule
\begin{equation*}
  R_i = (\rawRule{\Theta}{\plug{\bdry}{e}}),
\end{equation*}
and let $T_i = (R_j)_{j \prec i}$. By induction hypothesis the tt-theory $T'_i$ eligible for~$T_i$ has been constructred.
Because $\types_{T_i} \isMCtx{\Theta}$, by~\eqref{item:tt-cf-extension} there is an eligible labeling $\theta = \finmap{\symM_1 \mto \BB'_1, \ldots, \symM_n \mto \BB'_n}$ for~$\Theta$ such that $\types_{T'_k} \BB'_k$ for each~$k = 1, \ldots, n$.
The empty map $\gamma = \finmap{}$ is an eligible labeling for the empty context~$\emptyCtx$.
Because $\Theta; \emptyCtx \types_{T_i} \bdry$, by~\eqref{item:tt-cf-boundary} there is an eligible cf-boundary~$\bdry'$ for~$\bdry$ with respect to $\theta, \gamma$ such that $\types_{T'_i} \bdry'$.
We now are in possession of the cf-rule-boundary
\begin{equation}
  \label{eq:tt-to-cf-rule-boundary}
  \rawRule{\symM_1^{\BB'_1}, \ldots, \symM_n^{\BB'_n}}{\bdry'}
\end{equation}
eligible for the tt-rule-boundary $\rawRule{\Theta}{\bdry}$. Let
\begin{equation*}
  R'_i = (\rawRule{\symM_1^{\BB'_1}, \ldots, \symM_n^{\BB'_n}}{\plug{\bdry'}{e'}})
\end{equation*}
be the symbol or equality cf-rule induced by~\eqref{eq:tt-to-cf-rule-boundary}, as in \cref{def:context-free-symbol-rule,def:context-free-equality-rule}. Comparison with
\cref{def:symbol-rule,def:equality-rule} shows that $\erasex{e'} = e$, as required.

\medskip
\noindent
\emph{Proof of part (\ref{item:tt-cf-extension}):}
We proceed by induction on the derivation of $\types_T \isMCtx{\Theta}$.
The empty map is an eligible labeling for the empty metavariable context.
If $\types_T \isMCtx{\finmap{\Theta, \symM \of \BB}}$ then by inversion $\types_T \isMCtx{\Theta}$ and $\Theta; \emptyCtx \types_T \BB$. By induction hypothesis there exists an eligible labeling $\theta$ for $\Theta$, and by~\eqref{item:tt-cf-boundary}
applied to $T$, $T'$, $\Theta$, $\theta$, $\emptyCtx$, $\finmap{}$ a cf-boundary $\BB'$  eligible for $\BB$ such that $\types_{T'} \BB'$.
The map $\theta' = \finmap{\theta, \symM \mto \BB'}$ is eligible for $\finmap{\Theta, \symM \of \BB}$, and moreover $\types_{T'} \theta'(\symM')$ for every $\symM' \in \vert{}\theta'\vert{}$.

\medskip
\noindent
\emph{Proof of part (\ref{item:tt-cf-context})} is analogous to part (\ref{item:tt-cf-extension}).

\medskip
\noindent
\emph{Proof of part (\ref{item:tt-cf-boundary}):}
The non-abstracted boundaries reduce to instances of \eqref{item:tt-cf-judgement} by inversion, while the case of \rref{TT-Bdry-Abstr} is analogous to the case \rref{TT-Abstr} below.

\medskip
\noindent
\emph{Part (\ref{item:tt-cf-judgement}):}
Let $T$, $T'$, $\Theta$, $\theta$, $\Gamma$, $\gamma$ be as in~\eqref{item:tt-cf-judgement} with
\begin{align*}
  \Theta &= [\symM_1 \of \BB_1, \ldots, \symM_m \of \BB_p], \\
  \theta &= \finmap{\symM_1 \mto \BB'_1, \ldots, \symM_p \mto \BB'_p}, \\
  \Gamma &= [\mkvar{a}_1 \of A_1, \ldots, \mkvar{a}_p \of A_r], \\
  \gamma &= \finmap{\mkvar{a}_1 \mto A'_1, \ldots, \mkvar{a}_r \mto A'_r}.
\end{align*}
We have the further assumption that each $\symM_i$ has a cf-derivation $D_{\symM_i}$ of $\types_{T'} \BB'_i$, and each $\mkvar{a}_j$ a cf-derivation $D_{\mkvar{a}_j}$ of $\types_{T'} \isType{A'_j}$.
We proceed by structural induction on the derivation of $\Theta; \Gamma \types_T \JJ$.
In each case we construct a cf-derivation concluding with $\types_{T'} \JJ'$ such that~$\JJ'$ is eligible for~$\JJ$.

\inCase{TT-Var} Consider a tt-derivation ending with the variable rule
\begin{equation*}
  \infer{
  }{
    \Theta; \Gamma \types_T \mkvar{a}_j : A_j
  }
\end{equation*}
The corresponding cf-derivation is the application of~\rref{CF-Var}
\begin{equation*}
  \infer{ }{
    \types_{T'} \mkvar{a}_j^{A_j'} : A_j'
  }
\end{equation*}

\inCase{TT-Meta}
Consider a tt-derivation ending with the metavariable rule, where
$\BB_k = \abstr{x_1 \of B_1} \cdots \abstr{x_m \of B_m}\; \bdry$ and
$\BB'_k = \abstr{x_1 \of B'_1} \cdots \abstr{x_m \of B'_m}\; \bdry'$:
\begin{equation*}
  \infer
    {{\begin{aligned}
     &\Theta; \Gamma \types_T t_j : B_j[\upto{\vec{t}}{j}/\upto{\vec{x}}{j}]
      &&\text{for $j = 1, \ldots, m$}
     \\
     &\Theta; \Gamma \types_T \bdry[\vec{t}/\vec{x}]
    \end{aligned}}
    }{
      \Theta; \Gamma \types_T \plug{(\bdry[\vec{t}/\vec{x}])}{\symM_k(\vec{t})}
    }
\end{equation*}
The correspond cf-derivation ends with and application of~\rref{CF-Meta},
\begin{equation*}
  \infer{
    {\begin{aligned}
    &\types_{T'} t'_j : B'_j[\upto{\vec{t}'}{j}/\upto{\vec{x}}{j}] &&\text{for $j = 1, \ldots, m$} \\
    &\types_{T'} \bdry'[\vec{t}'/\vec{x}]
    \end{aligned}}
  }{
    \types_{T'} \plug{\bdry'}{\symM^{\BB'}(\vec{t}')}
  }
\end{equation*}
where the cf-terms $\vec{t}' = (t'_1, \ldots, t'_m)$ are constructed inductively as follows. Assuming we already have $\upto{\vec{t}'}{j}$, we apply the induction hypothesis to the $j$-th premise and obtain its eligible counterpart
$
  \types_{T'} t''_j : B''_j
$,
so that $\erasex{t''_j} = t_j$ and $\erasex{B''_j} = B_j[\upto{\vec{t}}{j}/\upto{\vec{x}}{j}]$.
It follows that $\erase{B''_j} = \erase{B'_j[\upto{\vec{t}'}{j}/\upto{\vec{x}}{j}]}$, therefore
we may use \cref{lem:boundary-convert} to modify $t''_j$ to a term $t'_j$ which fills $B'_j[\upto{\vec{t}'}{j}/\upto{\vec{x}}{j}]$.

\inCase{TT-Meta-Congr}
We consider a tt-derivation ending with a metavariable term congruence rule, where
$\BB_k = \abstr{x_1 \of B_1} \cdots \abstr{x_m \of B_m}\; \Box : C$ and
$\BB'_k = \abstr{x_1 \of B'_1} \cdots \abstr{x_m \of B'_m}\; \Box : C'$:
\begin{equation}
  \label{eq:tt-cf-meta-congr}
  \infer
  {
   { \begin{aligned}
        &\Theta; \Gamma \types_T s_j :
            B_j[\upto{\vec{s}}{j}/\upto{\vec{x}}{j}]
        &\text{for $j = 1, \ldots, m$}
        \\
        &\Theta; \Gamma \types_T t_j :
            B_j[\upto{\vec{t}}{j}/\upto{\vec{x}}{j}]
        &\text{for $j = 1, \ldots, m$}
        \\
        &\Theta; \Gamma \types_T s_j \equiv t_j :
            B_j[\upto{\vec{s}}{j}/\upto{\vec{x}}{j}]
        &\text{for $j = 1, \ldots, m$}
        \\
        &\Theta; \Gamma \types_T C[\vec{s}/\vec{x}] \equiv C[\vec{t}/\vec{x}]
      \end{aligned} }
  }{
    \Theta; \Gamma \types_T
    \symM_k(\vec{s}) \equiv \symM_k(\vec{t}) : C[\vec{s}/\vec{x}]
  }
\end{equation}
The corresponding cf-derivation ends with \rref{CF-Meta-Congr-Tm}
\begin{equation}
  \label{eq:tt-cf-meta-congr-cf}
  \infer
  { { \begin{aligned}
        &\types_{T'} s'_j : B'_j[\upto{\vec{s}'}{j}/\upto{\vec{x}}{j}]
           &&\text{for $j = 1, \ldots, m$}\\
        &\types_{T'} t'_j : B'_j[\upto{\vec{t}'}{j}/\upto{\vec{x}}{j}]
           &&\text{for $j = 1, \ldots, m$}\\
        &\erase{t'_k} = \erase{t''_j}
           &&\text{for $j = 1, \ldots, m$}\\
        &\types_{T'} s'_j \equiv t''_j
              : B'_j[\upto{\vec{s}'}{j}/\upto{\vec{x}}{j}] \by \alpha_j
           &&\text{for $j = 1, \ldots, m$}\\
        &\types_{T'} v : C'[\vec{s}'/\vec{x}]
           &&\erase{\symM^\BB(\vec{t}')} = \erase{v}
      \end{aligned} }
  }{
    \types_{T'}
    \symM^\BB(\vec{s}') \equiv v : C'[\vec{s'}/\vec{x}]
    \by \beta
  }
\end{equation}
where suitable $\vec{s}'$, $\vec{t}'$, $\vec{t}''$, $\vec{\alpha}$, $v$, and $\beta$ remain to be constructed.
The terms $\vec{s}'$ and $\vec{t}'$ are obtained as in the previous case, using the first two rows of premises of~\eqref{eq:tt-cf-meta-congr}.
The induction hypotheses for the third row give us judgements, for $j = 1, \ldots, m$,
\begin{equation*}
  \types_{T'} s''_j \equiv t''_j : B''_j
\end{equation*}
such that $\erase{B''_j} = \erase{B_j[\upto{\vec{s}'}{j}/\upto{\vec{x}}{j}]}$.
We convert the above equality along $\types_{T'} B''_j \equiv B_j[\upto{\vec{s}'}{j}/\upto{\vec{x}}{j}]$ to derive
\begin{equation*}
  \types_{T'} s'''_j \equiv t'_j : B_j[\upto{\vec{s}'}{j}/\upto{\vec{x}}{j}]
\end{equation*}
and since $\erase{s'''_j} = \erase{s'_j}$ by reflexivity and transitivity
\begin{equation*}
  \types_{T'} s'_j \equiv t'_j : B_j[\upto{\vec{s}'}{j}/\upto{\vec{x}}{j}].
\end{equation*}

It remains to construct~$v$ and~$\beta$. For the former, we apply $\rref{CF-Subst-EqTy}$ to $\types_{T'} \abstr{\vec{x} : \vec{B}'} \; \isType{C'}$ to derive
\begin{equation*}
  \types_{T'} C'[\vec{s}'/\vec{x}] \equiv C'[\vec{t}'/\vec{x}] \by \delta
\end{equation*}
ands use it to convert $\types_{T'} \symM_k(\vec{t}') : C'[\vec{t}'/\vec{x}]$ to
$\types_{T'} \convert{\symM_k(\vec{t}')}{\epsilon} : C'[\vec{s}'/\vec{x}]$ for a suitable~$\epsilon$.
We take $v = \convert{\symM_k(\vec{t}')}{\epsilon}$ and the minimal suitable~$\beta$.

\inCase{TT-Abstr} Consider a tt-derivation ending with an abstraction
\begin{equation*}
  \infer
  {
    \Theta; \Gamma \types_T \isType A \\
    \mkvar{a} \not\in \vert{}\Gamma\vert{} \\
    \Theta; \Gamma, \mkvar{a} \of A \types_T \JJ[\mkvar{a}/x]
  }{
    \Theta; \Gamma \types_T \abstr{x \of A} \; \JJ
  }
\end{equation*}
By induction hypothesis we obtain a derivation of $\types_{T'} \isType{A'}$ which is eligible for the first premise. The extended map $\finmap{\gamma, \mkvar{a} \mto A'}$ is eligible for $\Gamma, \mkvar{a} \of A$, and so by induction hypothesis we obtain a derivble $\types_{T'} \JJ'$ which is
eligible for the second premise with respect to $(\theta, \finmap{\gamma, \mkvar{a} \mto A'})$.
We form the desired abstraction by \rref{CF-Abstr},
\begin{equation*}
  \infer
  {
    \types_{T'} \isType{A'} \\
    \avar{a}{{}A'} \not\in \fv{\JJ'} \\
    \types_{T'} \JJ'
  }{
    \types_{T'} \abstr{x \of A'} \; \JJ'[x/\avar{a}{{}A'}]
  }
\end{equation*}

\inCaseText{of a specific rule}
Consider a specific tt-rule
\begin{equation*}
  R = (\rawRule{\sym{N}_1 \of \DD_1, \ldots, \sym{N}_m \of \DD_m}{\J}),
\end{equation*}
and the corresponding cf-rule
\begin{equation*}
  R' = (\rawRule{{\sym{N}_1}^{\DD'_1}, \ldots, {\sym{N}_k}^{\DD'_m}}{\J'})
\end{equation*}
Consider a tt-derivation ending with $\act{I} R$ where
$I = \finmap{\sym{N}_1 \mto e_1, \ldots, \sym{N}_m \to e_m}$:
\begin{equation}
  \label{eq:tt-cf-specific}
  \infer
  {
    { \begin{aligned}[t]
      & \Theta; \Gamma \types_T \plug{(\upact{I}{j} \DD_j)}{e_j} \quad \text{for $j = 1, \ldots, m$} \\
      & \Theta; \Gamma \types_T \act{I} \bdry
    \end{aligned} }
  }{
    \Theta; \Gamma \types_T \act{I} (\plug{\bdry}{e})
  }
\end{equation}
The corresponding cf-derivation is obtained by an application of $R'$ instantiated with
\begin{equation*}
  I' = \finmap{{\sym{N}_1}^{\DD'_1} \mto e'_1, \ldots, {\sym{N}_k}^{\DD'_k} \mto e'_m},
\end{equation*}
which is constructed inductively as follows.
Suppose $\upto{\vec{e'}}{j}$ have already been constructed in such a way that $\erasex{e'_k} = e_k$ and
$
  \types_{T'}
     \plug
       {(\upact{I'}{k} \DD'_k)}
       {e'_k}
$
for all $k < j$.
The induction hypothesis for the $j$-th premise of~\eqref{eq:tt-cf-specific} yields
$
  \types_{T'} \plug{\DD''_j}{e''_j}
$
such that $\erase{\DD''_j} = \erase{\upto{I'}{j} \DD'_j}$. We apply \cref{lem:boundary-convert} to modify~$e''_j$ to~$e'_j$ such that $\types_{T'} \plug{(\upto{I'}{j} \DD'_j)}{e'_j}$ and $\erasex{e'_j} = \erase{e''_j}$.
Lastly, the premise $\types_{T'} \act{I'} \bdry'$ is derivable because~$R'$ is finitary.

\inCaseText{of a congruence rule}
Consider a term tt-rule
\begin{equation*}
  R = (\rawRule{\sym{N}_1 \of \DD_1, \ldots, \sym{N}_m \of\DD_m}{t : C}),
\end{equation*}
and the corresponding cf-rule
\begin{equation*}
  R' = (\rawRule{\sym{N}_1^{\DD'_1}, \ldots, \sym{N}_m^{\DD'_m}}{t' : C'}),
\end{equation*}
Given instantiations
\begin{equation*}
  I = \finmap{\sym{N}_1 \mto f_1, \ldots, \sym{N}_m \to f_m}
  \quad\text{and}\quad
  J = \finmap{\sym{N}_1 \mto g_1, \ldots, \sym{N}_m \to g_m},
\end{equation*}
suppose the tt-derivation ends with the congruence rule for~$R$:
\begin{equation}
  \label{eq:tt-cf-congr-rule}
  \infer{
    {\begin{aligned}
      &\Theta; \Gamma \types_T \plug{(\upact{I}{j} \DD_j)}{f_j}  &&\text{for $i = 1, \ldots, m$}\\
      &\Theta; \Gamma \types_T \plug{(\upact{J}{j} \DD_j)}{g_j}  &&\text{for $i = 1, \ldots, m$}\\
      &\Theta; \Gamma \types_T \plug{(\upact{I}{j} \DD_j)}{f_j \equiv g_i} &&\text{for object boundary $\DD_j$} \\
      &\Theta; \Gamma \types_T \act{I} C \equiv \act{J} C
    \end{aligned}}
  }{
    \Theta; \Gamma \types \act{I} t \equiv \act{J} t : \act{I} C
  }
\end{equation}
The corresponding cf-derivation ends with the congruence rule for~$R'$,
\begin{equation*}
  \infer{
    {\begin{aligned}
    &\types_{T'} \plug{(\upact{I'}{i} \DD'_i)}{f'_i}  &&\text{for $i = 1, \ldots, m$}\\
    &\types_{T'} \plug{(\upact{J'}{i} \DD'_i)}{g'_i}  &&\text{for $i = 1, \ldots, m$}\\
    &\erase{g''_i} = \erase{g'_i} &&\text{for object boundary $\DD'_i$} \\
    &\types_{T'} \plug{(\upact{I'}{i} \DD'_i)}{f'_i \equiv g''_i \by \alpha_i} &&\text{for object boundary $\DD'_i$} \\
    &\types_{T'} t'' : \act{I'} C' &&\erase{t''} = \erase{\act{J'} t'} \\
    & &&\suitable{\beta}
  \end{aligned}}
  }{
    \types_{T'} \act{I'} t' \equiv t'' : \act{I}' C \by \beta
  }
\end{equation*}
where
\begin{equation*}
  I' = \finmap{\sym{N}_1^{\DD'_1} \mto f'_1, \ldots, \sym{N}_m^{\DD'_m} \to f'_m}
  \quad\text{and}\quad
  J' = \finmap{\sym{N}_1^{\DD'_1} \mto g'_1, \ldots, \sym{N}_m^{\DD'_m} \to g'_m}.
\end{equation*}
It remains to determine $\vec{f}'$, $\vec{g}'$, $\vec{g}''$, and $t''$.

The terms $\vec{f}'$ and $\vec{g}'$ are constructed from the first two rows of premises of the tt-derivation in the same way as $\vec{e}'$ in the previous case. The third row of premises yields equations, which after an application of~\cref{lem:boundary-convert}, take the form
\begin{equation*}
  \types_{T'} \plug{(\upact{I'}{i} \DD'_i)}{f''_i \equiv g''_i \by \beta_i}.
\end{equation*}
As $\erase{f'_i} = \erase{f''_i}$, these can be rectified by reflexivity and transitivity to the desired form
\begin{equation*}
  \types_{T'} \plug{(\upact{I'}{i} \DD'_i)}{f'_i \equiv g''_i \by \alpha_i}.
\end{equation*}

Finally, we construct $t''$ by converting $\types_{T'} \act{J} t' : \act{J} C$ along
$
  \types_{T'} \act{J'} C' \equiv \act{I'} C' \by \gamma
$,
which is derived as follows.
The induction hypothesis for the last premise of~\eqref{eq:tt-cf-congr-rule} gives
\begin{equation*}
  \types_{T'} C_1 \equiv C_2
\end{equation*}
such that $\erase{C_1} = \erase{\act{I'} C'}$ and $\erase{C_2} = \erase{\act{J'} C'}$.
Because $\types_{T'} \isType{\act{I'} C'}$ and $\types_{T'} \isType{\act{J'} C'}$, as well as
$\types_{T'} \isType{C_1}$ and $\types_{T'} \isType{C_2}$ by \cref{prop:context-free-presuppositivity},
we may adjust the above equation to
\begin{equation*}
  \types_{T'} \act{I'} C' \equiv \act{J'} C',
\end{equation*}
which is only a symmetry away from the desired one.

The case of a type specific rule is simpler and dealt with in a similar fashion.

\inCasesText{\rref{TT-EqTy-Refl}, \rref{TT-EqTy-Sym}, \rref{TT-EqTm-Refl}, \rref{TT-EqTm-Sym}}
each of these is taken care of by applying the induction hypotheses to the premises, followed by application of the corresponding cf-rule.

\inCasesText{\rref{TT-EqTy-Trans} and \rref{TT-EqTm-Trans}}
Consider a derivation ending with term transitivity
\begin{equation*}
  \infer
  { \Theta; \Gamma \types_T s \equiv t : A \\
    \Theta; \Gamma \types_T t \equiv u : A }
  { \Theta; \Gamma \types_T s \equiv u : A }
\end{equation*}
The induction hypotheses for the premises produce eligible judgements
\begin{equation*}
  \types_{T'} s' \equiv t' : A' \by \alpha
  \qquad\text{and}\qquad
  \types_{T'} t'' \equiv u'' : A'' \by \beta
\end{equation*}
Because $\erase{A'} = \erase{A''}$ and $\erase{t'} = \erase{t''}$, we may convert the second judgement to~$A'$, and rectify the left-hand side, which results in
\begin{equation*}
  \types_{T'} t' \equiv u' : A' \by \gamma.
\end{equation*}
Now \rref{CF-EqTm-Trans} applies.
The case of transitivity of type equality similar and easier.

\inCase{TT-Conv-Tm}
Consider a conversion
\begin{equation*}
  \infer
  { \Theta; \Gamma \types_T t : A \\
    \Theta; \Gamma \types_T A \equiv B }
  { \Theta; \Gamma \types_T t : B }
\end{equation*}
The induction hypotheses for the premises produce eligible judgements
\begin{equation*}
  \types_{T'} t'' : A'
  \qquad\text{and}\qquad
  \types_{T'} A'' \equiv B' \by \alpha
\end{equation*}
Because $\erase{A'} = \erase{A''}$, we obtain $A' \equiv B' \by \beta$, after which \rref{CF-Conv-Tm} can be used to convert~$\types_{T'} t'' : A'$ to a judgement $\types_{T'} t' : B' \by \beta$ which is eligible for the conclusion.

\inCase{TT-Conv-EqTm}
This case follows the same pattern as the previous one.
\end{proof}

\end{document}